\documentclass[12pt, reqno]{amsart}

\usepackage{amsmath, amsthm, amssymb}
\usepackage{enumerate}
\usepackage[margin=1.0in]{geometry}
\usepackage{xcolor}
\definecolor{cite}{rgb}{0.30,0.60,1.00}
\definecolor{url}{rgb}{0.00,0.00,0.80}
\definecolor{link}{rgb}{0.40,0.10,0.20}

\usepackage{zref-clever}

\AddToHook{env/theorem/begin}{%
	\zcsetup{countertype={theorem=theorem}}}
\AddToHook{env/claim/begin}{%
	\zcsetup{countertype={theorem=theorem}}}  
\AddToHook{env/proposition/begin}{%
	\zcsetup{countertype={theorem=proposition}}}
\AddToHook{env/lemma/begin}{%
	\zcsetup{countertype={theorem=lemma}}}
\AddToHook{env/conjecture/begin}{%
	\zcsetup{countertype={theorem=theorem}}}  
\AddToHook{env/corollary/begin}{%
	\zcsetup{countertype={theorem=corollary}}}
\AddToHook{env/definition/begin}{%
	\zcsetup{countertype={theorem=definition}}}
\AddToHook{env/remark/begin}{%
	\zcsetup{countertype={theorem=remark}}}
\AddToHook{env/example/begin}{%
	\zcsetup{countertype={theorem=example}}}

\zcsetup{nameinlink=false}

\usepackage[pdfusetitle,colorlinks,linkcolor=link,urlcolor=url,citecolor=cite,breaklinks]{hyperref}
\usepackage{graphicx}
\usepackage{subcaption}
\usepackage{mathdots}
\usepackage{tikz-cd}
\usepackage{comment}
\usepackage{xypic}
\usepackage{mathtools}
\usepackage{physics}
\usepackage{float}

\newcommand{\etale}{\'etale }
\newtheorem{theorem}{Theorem}[section]

\newtheorem{proposition}[theorem]{Proposition}

\newtheorem{corollary}[theorem]{Corollary}

\theoremstyle{definition}

\theoremstyle{definition}
\newtheorem{remark}[theorem]{Remark}
\theoremstyle{definition}
\newtheorem{example}[theorem]{Example}

\newcommand{\cref}[1]{\zcref{#1}}
\newcommand{\Cref}[1]{\zcref[S]{#1}}

\newcommand{\zIntegers}{\mathbb{Z}}

\newcommand{\cComplex}{\mathbb{C}}
\newcommand{\multiplicativegroup}[1]{#1^{\times}}

\newcommand{\Hom}{\mathrm{Hom}}

\newcommand{\Supp}{\mathrm{supp}}

\newcommand{\idmap}{\mathrm{id}}
\newcommand{\conjugate}[1]{\overline{#1}}

\newcommand{\lengthof}{\ell}
\renewcommand{\abs}[1]{\left|#1\right|}
\newcommand{\sizeof}[1]{\left|#1\right|}
\newcommand{\lcm}{\operatorname{lcm}}

\renewcommand{\innerproduct}[2]{\left\langle#1,#2\right\rangle}

\newcommand{\fieldCharacter}{\psi}
\newcommand{\centralCharacter}[1]{\omega_{#1}}
\newcommand{\Ind}[3]{\mathrm{Ind}_{#1}^{#2}\left(#3\right)}

\newcommand{\Whittaker}{\mathcal{W}}
\newcommand{\Contragradient}[1]{#1^{\vee}}

\newcommand{\besselFunction}{\mathcal{J}}
\newcommand{\besselFunctionOfFiniteFieldRepresentation}[1]{\besselFunction_{#1, \fieldCharacter}}
\newcommand{\SpehRepresentation}[2]{\Delta\left(#1, #2\right)}

\newcommand{\besselSpehFunction}[2]{\mathcal{BS}_{\SpehRepresentation{#1}{#2}, \fieldCharacter}}
\newcommand{\specialBesselSpeh}[2][\fieldCharacter]{\mathcal{K}_{{#2}, {#1}}}
\newcommand{\specialBesselSpehNormalized}[2][\fieldCharacter]{\mathcal{K}^{\ast}_{{#2}, {#1}}}

\newcommand{\GKGammaFactor}[3]{\gamma^{\mathrm{GK}}\left(#1 \times #2, #3\right)}

\newcommand{\LocalGKGammaFactor}[4]{\gamma^{\mathrm{GK}}\left(#1, #2 \times #3, #4\right)}

\newcommand{\GKPreGammaFactor}[3]{\Gamma^{\mathrm{GK}}\left(#1 \times #2, #3\right)}

\newcommand{\Irr}{\mathrm{Irr}}

\newcommand{\grpIndex}[2]{\left[#1:#2\right]}

\newcommand{\pHallLittlewood}{\mathrm{P}}
\newcommand{\qHallLittlewood}{\mathrm{Q}}
\newcommand{\hHallLittlewood}{\mathrm{H}}

\newcommand{\htHallLittlewood}{\tilde{\mathrm{H}}}
\newcommand{\ptHallLittlewood}{\tilde{\mathrm{P}}}

\newcommand{\IdentityMatrix}[1]{I_{#1}}
\newcommand{\diag}{\mathrm{diag}}

\renewcommand{\trace}{\operatorname{tr}}
\newcommand{\GL}{\mathrm{GL}}
\newcommand{\UnipotentSubgroup}{U}
\newcommand{\UnipotentRadicalForWss}[2]{N_{\qty(#2^{#1})}}

\newcommand{\UnipotentRadical}{N}

\newcommand{\ParabolicSubgroup}{P}
\newcommand{\weylElement}[1]{w_{({#1})}}

\makeatletter
\newcommand{\ostar}{\!{\mathpalette\make@circled *}\!}
\newcommand{\make@circled}[2]{%
	\ooalign{$\m@th#1\smallbigcirc{#1}$\cr\hidewidth$\m@th#1#2$\hidewidth\cr}%
} 
\newcommand{\smallbigcirc}[1]{%
	\vcenter{\hbox{\scalebox{0.77778}{$\m@th#1\bigcirc$}}}%
}
\makeatother

\newcommand{\Erdelyi}{Erd{\'e}lyi}
\newcommand{\Toth}{T{\'o}th}

\newcommand{\FieldNorm}[2]{\mathrm{N}_{#1:#2}}
\newcommand{\aFieldNorm}{\mathrm{N}}

\newcommand{\finiteField}{\mathbb{F}}
\newcommand{\finiteFieldExtension}[1]{\finiteField_{#1}}

\newcommand{\algebraicClosure}[1]{\overline{#1}}
\newcommand{\charactergroup}[1]{\widehat{\multiplicativegroup{\finiteFieldExtension{#1}}}}

\newcommand{\Galois}{\operatorname{Gal}}
\newcommand{\Frobenius}{\operatorname{Fr}}
\newcommand{\restrictionOfScalars}[3]{\operatorname{Res}_{#1 \slash #2}{#3}}
\newcommand{\multiplcativeScheme}{\mathbb{G}_m}
\newcommand{\affineLine}{\mathbb{A}^1}
\newcommand{\squareMatrix}{\operatorname{Mat}}

\newcommand{\ProjectionOperator}{\operatorname{pr}}
\newcommand{\SymmetricGroup}{\mathfrak{S}}

\newcommand{\ParabolicForSpeh}[2]{P_{\left({#1}^{#2}\right)}}

\newcommand{\Partitions}{\mathrm{Par}}
\newcommand{\IrrCuspidal}{\mathrm{Irr}_{\mathrm{cusp}}}
\newcommand{\IrrCuspidalAll}{\IrrCuspidal \left(\GL_{\bullet}\left(\finiteField\right)\right)}
\newcommand{\ClassFunctionsRing}{\mathcal{C}}
\newcommand{\ClassFunctionsGLRing}{\ClassFunctionsRing\left(\GL_{\bullet}\left(\finiteField\right)\right)}
\newcommand{\CompactlyClassFunctionsGLRing}{\ClassFunctionsRing_c\left(\GL_{\bullet}\left(\finiteField\right)\right)}
\newcommand{\FullClassFunctionsGL}{\ClassFunctionsRing_{\ast}\left(\GL_{\bullet}\left(\finiteField\right)\right)}

\newcommand{\DualClassFunctionsGLRing}{\ClassFunctionsRing^{\ast}\left(\GL_{\bullet}\left(\finiteField\right)\right)}
\newcommand{\CharacteristicMapF}{\mathrm{ch^{\finiteField}}}
\newcommand{\CharacteristicMapFHat}{\mathrm{ch^{\widehat{\finiteField}}}}
\newcommand{\nOfPartition}{\mathfrak{n}}

\newcommand{\localField}{F}
\newcommand{\ringOfIntegers}{\mathfrak{o}}

\newcommand{\maximalIdeal}{\mathfrak{p}}
\newcommand{\depthZeroRepresentation}{\mathcal{T}}
\newcommand{\depthZeroCharacter}{\mathcal{X}}

\newcommand{\quotientMap}{\nu}
\newcommand{\Lift}{\mathcal{L}}
\newcommand{\uniformizer}{\varpi}

\newcommand{\intertwiningOperator}{M^{\left(z_1, \dots, z_c\right)}}
\newcommand{\holomorphicRepresentation}{\depthZeroRepresentation^{\left(z_1, \dots, z_c\right)}}

\newcommand{\GKGaussSum}[3]{\mathcal{G}^{\mathrm{GK}}\left(#1 \times #2, #3\right)}
\newcommand{\TwistedGaussSum}[3]{\mathcal{G}\left(#1 \times #2, #3\right)}

\newcommand{\GKGaussSumScalar}[3]{\mathrm{G}\left(#1 \times #2, #3\right)}
\newcommand{\GaussSumScalar}[2]{\mathrm{G}\left(#1, #2\right)}
\newcommand{\GaussSumCharacter}[4]{\tau_{#1}\left(#2 \times #3, #4\right)}
\newcommand{\kcNotation}[3]{(#1,#2)_{#3}}
\newcommand{\fieldCharacterkc}[2]{\fieldCharacter_{\qty(#2^{#1})}}

\newcommand{\ExoticKloosterman}{\mathrm{Kl}}
\newcommand{\ExoticKloostermanNormalized}{\mathrm{Kl}^{\ast}}

\newcommand{\SymmetricPower}{\mathrm{Sym}}
\newcommand{\TensorProductNormOne}[1]{\mathrm{N}_{1}^{#1}}
\newcommand{\TensorProductNormTwo}[1]{\mathrm{N}_{2}^{#1}}
\newcommand{\KloostermanGlobalClassFunction}{K^{\ast}_{\alpha, \fieldCharacter}}

\newcommand{\ladicnumbers}{\algebraicClosure{\mathbb{Q}_{\ell}}}
\newcommand{\artinScrier}[1]{\operatorname{AS}_{{\fieldCharacter_{#1}}}}
\newcommand{\convolutionWithCompactSupport}{\boldsymbol{\mathrm{R}}}

\newcommand{\conjugacyClass}[1]{\mathrm{Cl}\left(#1\right)}

\newcommand{\jordanSupportedCharacter}[1]{\mathcal{X}_{#1}}

\hypersetup{pdfauthor={Elad Zelingher},
	pdfsubject={Representation theory},
	pdfkeywords={Bessel functions, Kloosterman sums, Gauss sums}}

\title[Exotic matrix exponential sums]{On exotic matrix exponential sums and Bessel--Speh functions}

\author{Elad Zelingher}
\address{Department of Mathematics, University of Michigan, 1844 East Hall, 530 Church Street, Ann Arbor, MI 48109-1043 USA}
\email{eladz@umich.edu}


\begin{document}

\begin{abstract}
	In a previous work with Carmon, we defined Bessel--Speh functions. These are matrix coefficients of irreducible Speh representations of $\GL_{kc}(\finiteField)$, where $\finiteField$ is a finite field. They arise from $(k,c)$ models, which are models that generalize the Whittaker model to Speh representations attached to irreducible generic representations. These constructions are finite field analogs of objects arising naturally in the generalized doubling method over $p$-adic fields, a recently active area of the Langlands program. In this article we study special values of Bessel--Speh functions which were used in our previous work with Carmon to define Ginzburg--Kaplan gamma factors. Our main result computes the special values of interest explicitly in terms of new arithmetic objects we introduce, called exotic matrix Kloosterman sums, which generalize both Katz's exotic Kloosterman sums and twisted matrix Kloosterman sums. We then show that exotic matrix Kloosterman sums can be expressed as products of modified Hall--Littlewood polynomials evaluated at roots of the characteristic polynomial of the Frobenius acting on Katz's exotic Kloosterman sheaf. As an application of our results, we establish new identities for Bessel functions of irreducible generic representations.
\end{abstract}

\maketitle

\section{Introduction}\label{sec:introduction}

This work establishes a relation between two objects from different areas in mathematics. On the number theory (or more precisely, exponential sums) side, the objects of interest are (exotic) matrix Kloosterman sums. On the representation theory side, the objects of interest are special matrix coefficients attached to Speh representations of finite general linear groups. Matrix coefficients of this sort are often called ``Bessel functions'', since they can be expressed as finite field sums that resemble the contour integral representation of the Bessel function. We show that certain special values of the matrix coefficients of interest are given by the matrix Kloosterman sums of interest.

This is not the first time a relation between Kloosterman sums and representation theory of general linear groups has been established. This was done before by Curtis and Shinoda~\cite{curtis2004zeta}, where they showed that certain special values of special matrix coefficients, called Bessel functions, attached to irreducible generic representations of finite general linear groups are given by exotic Kloosterman sums. This relation was later reproved by Soudry and the author in~\cite{SoudryZelingher2023} using the theory of Shahidi gamma factors. A vast generalization of the latter method was used by the author in~\cite{zelingher2022values} in order to establish a relation between other special values of Bessel functions and traces of the action of Frobenius on exterior powers of exotic Kloosterman sheaves.

In this work we use a similar method to establish the relation between (exotic) matrix Kloosterman sums and the special matrix coefficients of interest. In our case, we use the theory of Ginzburg--Kaplan gamma factors developed by Carmon and the author in~\cite{CarmonZelingher2025} and the results of the author from~\cite{Zelingher2024b}.

We move to discuss the objects involved in more detail.

Let $\finiteField$ be a finite field with cardinality $q$ and let $\fieldCharacter \colon \finiteField \to \multiplicativegroup{\cComplex}$ be a non-trivial character.

\subsection{Matrix Kloosterman sums}\label{subsec:matrix-kloosterman-sums}

Kloosterman sums are central objects in number theory. They are used in analytic number theory to bound exponential sums and appear in the Kuznetsov trace formula~\cite{KowalskiMichelSawin2017}. Kloosterman sheaves are sheaf-theoretic incarnations of Kloosterman sums. They were defined by Deligne~\cite{deligne569cohomologie} and studied extensively by Katz~\cite{katz2016gauss}. Kloosterman sheaves are important objects in the geometric Langlands program and afford a source of automorphic representations in the function field case~\cite{HeinlothNgoYun2013, Yun2016}.

Recently, \Erdelyi{}--\Toth{} and their collaborators started studying matrix variants of Kloosterman sums, called \emph{matrix Kloosterman sums}~\cite{erdelyi2021matrix, erdelyi2022purity, ErdelyiTothZabradi2024}. These sums were defined in a slightly more general context by Hodges~\cite{Hodges1956} and studied mostly for the special case of scalar matrices by various authors~\cite{Kim1998, ChaeKim2003, Fulman2001}. Matrix Kloosterman sums show up in the study of expanding horospheres on $\GL_n$~\cite{Marklof2010, el2022effective}.

In~\cite{erdelyi2021matrix}, \Erdelyi{}--\Toth{} studied the matrix Kloosterman sum $$\ExoticKloosterman\left(\fieldCharacter, h\right) = \sum_{x \in \GL_c\left(\finiteField\right)} \fieldCharacter\left( \trace\left(x + h x^{-1}\right) \right),$$
where $h \in \squareMatrix_c\left(\finiteField\right)$. Note that $\ExoticKloosterman\left(\fieldCharacter, h\right)$ is a class function, i.e., it is constant on conjugacy classes of $\squareMatrix_c\left(\finiteField\right)$. \Erdelyi{}--\Toth{} proved that this sum satisfies a multiplicative property: if $h_1 \in \GL_{c_1}\left(\finiteField\right)$ and $h_2 \in \GL_{c_2}\left(\finiteField\right)$ are such that $h_1$ and $h_2$ do not have common eigenvalues over the algebraic closure $\algebraicClosure{\finiteField}$ then $$\ExoticKloosterman\left(\fieldCharacter, \diag\left(h_1, h_2\right)\right) = q^{c_1 c_2} \ExoticKloosterman \left(\fieldCharacter, h_1\right) \ExoticKloosterman \left(\fieldCharacter, h_2\right).$$
In view of this multiplicativity property and the classification of conjugacy classes of $\squareMatrix_c\left(\finiteField\right)$, it suffices to know how to compute matrix Kloosterman sums for generalized Jordan matrices. \Erdelyi{}--\Toth{} computed explicitly the matrix Kloosterman sums corresponding to singular Jordan matrices, and gave a recursive formula/algorithm that allows one to compute matrix Kloosterman sums of Jordan matrices with a single eigenvalue in terms of classical Kloosterman sums. They also gave effective bounds for these matrix Kloosterman sums. In~\cite{erdelyi2022purity}, together with Sawin, they studied these matrix Kloosterman sums using geometric tools.

In~\cite{Zelingher2023}, the author studied twisted matrix Kloosterman sums: for $h \in \GL_c\left(\finiteField\right)$ and $\alpha = \alpha_1 \times \dots \times \alpha_k$ where $\alpha_1,\dots,\alpha_k \colon \multiplicativegroup{\finiteField} \to \multiplicativegroup{\cComplex}$ are characters, the twisted matrix Kloosterman sum is $$\ExoticKloosterman\left(\alpha, \fieldCharacter, h\right) = \sum_{\substack{x_1, \dots, x_k \in \GL_c\left(\finiteField\right) \\
		x_1 \dots x_k = h}} \left(\prod_{j=1}^k{\alpha_j\left(\det x_j\right)}\right)\fieldCharacter\left(\sum_{j=1}^k \trace x_j\right).$$
Note that this is a class function of $\GL_c\left(\finiteField\right)$. In~\cite{Zelingher2023}, the author showed that if $h$ is a generalized Jordan matrix, then the twisted matrix Kloosterman sum attached to it can be expressed in terms of a modified Hall--Littlewood polynomial evaluated at the inverses of the roots of the $L$-function associated with the family of twisted Kloosterman sums $\ExoticKloosterman\left(\alpha, \fieldCharacter, \xi\right)$, where $\xi$ is any eigenvalue of $h$ over the algebraic closure. He used representation theoretic methods, which were inspired by the theory of gamma factors developed by Carmon and him in~\cite{CarmonZelingher2025}.

\subsection{Speh representations and Bessel--Speh functions}
In the representation theory of $\GL_n\left(\finiteField\right)$, a fruitful way to study irreducible representations is by attaching invariants to them. Families of gamma factors form a useful source of invariants. These invariants can be thought of as finite field analogs of $L$-functions; Similarly to global $L$-functions, gamma factors often fit into functional equations. Analogously to global $L$-functions, ``poles'' of gamma factors (inputs for which the absolute value is not $1$) capture interesting properties of the representations in question. For example, the tensor product gamma factor allows one to recover the cuspidal support of any irreducible representation, see \Cref{cor:cuspidal-support-and-absolute-value}. As another example, the exterior square gamma factor allows one to determine whether an irreducible cuspidal representation is self-dual or, equivalently, if it admits a non-zero Shalika vector, see~\cite[Section 2.3.1, Remark 2.17]{YeZeligher18} and~\cite[Proposition 4.2]{Jo2024}.

One commonly studied family of gamma factors is the one associated to the tensor product representation attached to a pair of representations $\pi$ and $\tau$ of a classical group $G\left(\finiteField\right)$ and of $\GL_k\left(\finiteField\right)$, respectively. Let us refer to this family as ``tensor product gamma factors'' for short. Many constructions of gamma factors require the irreducible representations in question to be generic. This is some genericity condition which we will explain below.

In the past few years, new constructions for tensor product gamma factors emerged. In these constructions, which we call ``constructions of doubling type'', the irreducible representation $\tau$ is required to be generic, but no assumption is required for the irreducible representation $\pi$. 

Constructions of doubling type make use of $\kcNotation{k}{c}{\fieldCharacter}$ models of Speh representations attached to $\tau$. Let us explain this briefly. Let $c \ge 1$ and $k \ge 1$. Let $\UnipotentRadicalForWss{k}{c}$ be the unipotent radical of the parabolic subgroup of $\GL_{kc}\left(\finiteField\right)$ corresponding to the composition $\left(c^k\right)$. We define a character $\fieldCharacterkc{k}{c} \colon \UnipotentRadicalForWss{k}{c} \to \multiplicativegroup{\cComplex}$ by the formula
$$\fieldCharacterkc{k}{c} \begin{pmatrix}
	\IdentityMatrix{c} & X_1 & \ast & \ast  & \ast \\
	& \IdentityMatrix{c} & X_2 & \ast & \ast \\
	& & \ddots & \ddots & \ast  \\
	& & & \IdentityMatrix{c} &  X_{k-1} \\
	& & & & \IdentityMatrix{c}
\end{pmatrix} = \fieldCharacter\left( \sum_{j=1}^{k-1} \trace X_j \right).$$ A \emph{$\kcNotation{k}{c}{\fieldCharacter}$ vector} of a representation $\rho$ of $\GL_{kc}\left(\finiteField\right)$ is a vector $0 \ne v \in \rho$ such that $\rho\left(u\right)v = \fieldCharacterkc{k}{c}\left(u\right)v$ for every $u \in \UnipotentRadicalForWss{k}{c}$. An irreducible representation $\tau$ of $\GL_k\left(\finiteField\right)$ is called \emph{generic} if it admits a $\kcNotation{k}{1}{\fieldCharacter}$ vector, in which case, by the uniqueness of Whittaker vectors~\cite[Corollary 5.6]{SilbergerZink00} due to Gelfand--Graev~\cite{GelfandGraev1962}, this vector is unique (up to scalar multiplication). For any $c \ge 1$, and any irreducible generic representation $\tau$, one may attach an irreducible representation $\SpehRepresentation{\tau}{c}$ of $\GL_{kc}\left(\finiteField\right)$. These representations are called \emph{Speh representations}. They were defined in the context of finite fields by Carmon in his master's thesis~\cite{Carmon2023}. Carmon showed that for any irreducible generic representation $\tau$ and any $c \ge 1$, the Speh representation $\SpehRepresentation{\tau}{c}$ admits a $\kcNotation{k}{c}{\fieldCharacter}$ vector, and that this vector is unique (up to scalar multiplication). By Frobenius reciprocity, it follows that there exists a unique subspace of $\Ind{\UnipotentRadicalForWss{k}{c}}{\GL_{kc}\left(\finiteField\right)}{\fieldCharacterkc{k}{c}}$ which is isomorphic to $\SpehRepresentation{\tau}{c}$. We call this subspace the \emph{$\kcNotation{k}{c}{\fieldCharacter}$ model of $\SpehRepresentation{\tau}{c}$}.

Using the uniqueness of the $\kcNotation{k}{c}{\fieldCharacter}$ vector, we may define a distinguished matrix coefficient of $\SpehRepresentation{\tau}{c}$, as follows. Fix a $\kcNotation{k}{c}{\fieldCharacter}$ vector $v \in \SpehRepresentation{\tau}{c}$ and an inner product $\innerproduct{\cdot}{\cdot}$ on $\SpehRepresentation{\tau}{c}$ which is invariant under the $\GL_{kc}\left(\finiteField\right)$ action. Consider the following matrix coefficient of $\SpehRepresentation{\tau}{c}$, given by $$\besselSpehFunction{\tau}{c}\left(g\right) = \frac{\innerproduct{\SpehRepresentation{\tau}{c}\left(g\right)v}{v}}{\innerproduct{v}{v}},$$
where $g \in \GL_{kc}\left(\finiteField\right)$. 
This definition does not depend on the choice of the $\kcNotation{k}{c}{\fieldCharacter}$ vector $v$, nor on the choice of the inner product $\innerproduct{\cdot}{\cdot}$, since both are unique up to scalar multiplication. We mention that $\besselSpehFunction{\tau}{c}$ is also an element of the $\kcNotation{k}{c}{\fieldCharacter}$ model of $\SpehRepresentation{\tau}{c}$. Consider the special values given by $$\specialBesselSpeh{\tau}\left(h\right) = \begin{cases}
	\besselSpehFunction{\tau}{c}\begin{pmatrix}
		& \IdentityMatrix{\left(k-1\right)c}\\
		h
	\end{pmatrix} & k \ge 2,\\
	\tau\left(\det h\right) \fieldCharacter\left(\trace h^{-1}\right) & k = 1,
\end{cases}$$
where $h \in \GL_c\left(\finiteField\right)$. Note that the notation $\specialBesselSpeh{\tau}\left(h\right)$ does not include $k$ nor $c$, and these should be inferred from $\tau$ and $h$, respectively.

In~\cite{CarmonZelingher2025}, Carmon and the author showed that $\specialBesselSpeh{\tau}$ can be used to define a tensor product gamma factor attached to $\pi$ and $\tau$, where $\pi$ is an irreducible representation of $\GL_c\left(\finiteField\right)$. They showed that this gamma factor is multiplicative in both arguments and fits into a functional equation. 
Moreover, they showed that if $\tau$ is a generic principal series representation, that is, $\tau$ is the unique irreducible generic subrepresentation of the parabolically induced representation $\alpha_1 \circ \dots \circ \alpha_k$, where $\alpha_1,\dots,\alpha_k \colon \multiplicativegroup{\finiteField} \to \multiplicativegroup{\cComplex}$ are characters, then the special values $\specialBesselSpeh{\tau}$ are closely related to the twisted matrix Kloosterman sums discussed in \Cref{subsec:matrix-kloosterman-sums}:
\begin{equation}\label{eq:identity-for-principal-series}
	\specialBesselSpeh{\tau}\left(h\right) = q^{-\left(k-1\right)c^2} \ExoticKloosterman\left(\alpha^{-1},\fieldCharacter, \left(-1\right)^{k-1} h^{-1}\right),
\end{equation}
where $\alpha^{-1} = \alpha_1^{-1} \times \dots \times \alpha_k^{-1}$ and $h \in \GL_c\left(\finiteField\right)$.
It is now natural to ask whether a similar relation holds for an arbitrary irreducible generic representation $\tau$. We discuss this in the next section.

\subsection{Exotic matrix Kloosterman sums}
Suppose that $k \ge 2$. For $c=1$, denote $\besselSpehFunction{\tau}{1} = \besselFunction_{\tau, \fieldCharacter}$. Curtis--Shinoda computed the values $$\specialBesselSpeh{\tau}\left(h\right) = \besselFunction_{\tau, \fieldCharacter}\begin{pmatrix}
	& \IdentityMatrix{k-1}\\
	h
\end{pmatrix},$$
where $h \in \multiplicativegroup{\finiteField}$. Let us explain their formula first for the special case where $\tau$ is an irreducible cuspidal representation of $\GL_k\left(\finiteField\right)$. In this case, $\tau$ corresponds to a Frobenius orbit of a regular character $\alpha \colon \multiplicativegroup{\finiteFieldExtension{k}} \to \multiplicativegroup{\cComplex}$, where $\finiteFieldExtension{k} \slash \finiteField$ is a field extension of degree $k$. Then Curtis--Shinoda showed that~\cite{curtis2004zeta}
\begin{equation}\label{eq:curtis-shinoda-cuspidal}
	\besselFunction_{\tau, \fieldCharacter}\begin{pmatrix}
		& \IdentityMatrix{k-1}\\
		h
	\end{pmatrix} = \left(-1\right)^{k-1} q^{-\left(k-1\right)} \sum_{\substack{x \in \multiplicativegroup{\finiteFieldExtension{k}}\\
			\FieldNorm{k}{1}\left(x\right) = \left(-1\right)^{k-1} h^{-1}}} \alpha^{-1}\left(x\right) \fieldCharacter_k\left(x\right),
\end{equation}
where $\fieldCharacter_k = \fieldCharacter \circ \trace_{\finiteFieldExtension{k} \slash \finiteField}$ and where $\FieldNorm{k}{1} \colon \multiplicativegroup{\finiteFieldExtension{k}} \to \multiplicativegroup{\finiteField}$ is the norm map.

More generally, for any choice of irreducible cuspidal representation $\tau_1$, $\dots$, $\tau_s$ of $\GL_{k_1}\left(\finiteField\right)$, $\dots$, $\GL_{k_s}\left(\finiteField\right)$, respectively, with $k_1 + \dots + k_s = k$, the parabolically induced representation $\tau_1 \circ \dots \circ \tau_s$ admits a unique irreducible generic subrepresentation $\tau$ of $\GL_k\left(\finiteField\right)$, and every irreducible generic representation can be realized in this way. If $\tau$ is given as above and $\tau_j$ corresponds to the Frobenius orbit of a regular character $\alpha_j \colon \multiplicativegroup{\finiteFieldExtension{k_j}} \to \multiplicativegroup{\cComplex}$ for every $j$, then Curtis--Shinoda showed that
\begin{equation}\label{eq:curtis-shinoda-general}
\besselFunction_{\tau, \fieldCharacter}\begin{pmatrix}
		& \IdentityMatrix{k-1}\\
		h
	\end{pmatrix} = \left(-1\right)^{k+s} q^{-\left(k-1\right)} \sum_{\substack{x_1 \in \multiplicativegroup{\finiteFieldExtension{k_1}}, \dots, x_s \in \multiplicativegroup{\finiteFieldExtension{k_s}}\\
			\FieldNorm{k_1}{1}\left(x_1\right) \dots \FieldNorm{k_s}{1}\left(x_s\right) = \left(-1\right)^{k-1} h^{-1}}} \prod_{j=1}^s \alpha_j^{-1}\left(x_j\right) \fieldCharacter_{k_j}\left(x_j\right).
\end{equation}

We call the sums appearing on the right hand side of \eqref{eq:curtis-shinoda-cuspidal} and \eqref{eq:curtis-shinoda-general} \emph{exotic Kloosterman sums}. These were introduced by Deligne in~\cite[Sommes trig. 7.18]{deligne569cohomologie}, and studied by Katz in~\cite[8.8.5]{katz2016gauss}. The term was coined by Katz in~\cite[Pages 152 and 160]{katz1993estimates}. Note that the sum \eqref{eq:curtis-shinoda-general} is a convolution of sums of the form \eqref{eq:curtis-shinoda-cuspidal}. This suggests that in order to generalize \eqref{eq:identity-for-principal-series} to arbitrary irreducible generic representations, we need to first define a matrix analog of the exotic Kloosterman sum in \eqref{eq:curtis-shinoda-cuspidal}, and then define a matrix analog of the exotic Kloosterman sum in \eqref{eq:curtis-shinoda-general} using convolution.

The exotic Kloosterman sum appearing in \eqref{eq:curtis-shinoda-cuspidal} relies on the notion of the norm map. In order to define its matrix analog, we will need a matrix version of the norm map. Fortunately, such map was defined by Shintani~\cite{shintani1976two}. The definition is subtle, since the norm map is not a map  $\GL_c\left(\finiteFieldExtension{k}\right) \to \GL_c\left(\finiteField\right)$ but rather a map between the conjugacy classes of these groups. However, this is compatible with the fact that twisted matrix Kloosterman sums are class functions.

\subsection{Main results}
For any character $\chi \colon \multiplicativegroup{\finiteFieldExtension{k}} \to \multiplicativegroup{\cComplex}$ we define for $h \in \GL_c\left(\finiteField\right)$ an exotic matrix Kloosterman sum via the formula
$$\ExoticKloosterman\left(\chi, \fieldCharacter, h\right) = \sum_{\substack{x \in \GL_c\left(\finiteFieldExtension{k}\right)\\
h \in \conjugacyClass{\FieldNorm{k}{1}\left(x\right)}}} \frac{1}{\#\conjugacyClass{\FieldNorm{k}{1}\left(x\right)}} \chi\left(\det x\right) \fieldCharacter_{k}\left(\trace x\right).$$
Here $\conjugacyClass{\FieldNorm{k}{1}\left(x\right)}$ is the conjugacy class of $\GL_c\left(\finiteField\right)$ associated to the norm of $x \in \GL_c\left(\finiteFieldExtension{k}\right)$ via the Shintani norm map~\cite{shintani1976two}.

Given a character $\alpha = \alpha_1 \times \dots \times \alpha_s \to \cComplex$ where $\alpha_j \colon \multiplicativegroup{\finiteFieldExtension{k_j}} \to \multiplicativegroup{\cComplex}$ we define for $h \in \GL_c\left(\finiteField\right)$ $$\ExoticKloosterman\left(\alpha, \fieldCharacter, h\right) = \sum_{\substack{x_1,\dots,x_s \in \GL_c\left(\finiteField\right)\\
x_1 \dots x_s = h}} \ExoticKloosterman\left(\alpha_1, x_1\right) \dots \ExoticKloosterman\left(\alpha_s, x_s\right).$$

We show that the values $\specialBesselSpeh{\tau}$ are given by these exotic matrix Kloosterman sums (\Cref{thm:bessel-speh-is-an-exotic-kloosterman-sum-generic}):
\begin{theorem}\label{thm:introduction-bessel-speh-is-an-exotic-kloosterman-sum-generic}
	Let $\tau$ be an irreducible generic representation of $\GL_k\left(\finiteField\right)$ with cuspidal support $\left\{\tau_1,\dots,\tau_s\right\}$ where $\tau_j$ is an irreducible cuspidal representation of $\GL_{k_j}\left(\finiteField\right)$ for every $j$, such that $k_1 + \dots + k_s = k$. Suppose that for every $j$, $\tau_j$ corresponds to the Frobenius orbit of a regular character $\alpha_j \colon \multiplicativegroup{\finiteFieldExtension{k_j}} \to \multiplicativegroup{\cComplex}$. Then for any $h \in \GL_c\left(\finiteField\right)$,
	$$\specialBesselSpeh{\tau}\left(h\right) = \left(-1\right)^{c\left(k+s\right)} q^{-\left(k-1\right)c^2} \ExoticKloosterman\left(\alpha^{-1}, \fieldCharacter, \left(-1\right)^{k-1} h^{-1}\right),$$
	where $\alpha^{-1} = \alpha_1^{-1} \times \dots \times \alpha_s^{-1}$.
\end{theorem}

In addition, we show that these exotic matrix Kloosterman sums can be expressed as products of modified Hall--Littlewood polynomials evaluated at the roots of $\ExoticKloosterman\left(\alpha, \fieldCharacter\right)$ (\Cref{thm:general-formula-for-exotic-kloosterman-sum-of-conjugacy-class}). This generalizes the author's previous results~\cite{Zelingher2023}.

\begin{theorem}\label{thm:introduction-general-formula-for-exotic-kloosterman-sum-of-conjugacy-class}
	Suppose that $h$ is conjugate to $\diag\left(J_{\mu_1}\left(h_{\xi_1}\right), \dots, J_{\mu_r}\left(h_{\xi_r}\right)\right)$ where
	\begin{enumerate}
		\item For every $i$, $h_{\xi_i} \in \GL_{a_i}\left(\finiteField\right)$ is a regular elliptic matrix whose eigenvalues multiset is the Frobenius orbit $\left[\xi_i\right]$ where $\xi_i \in \multiplicativegroup{\finiteFieldExtension{a_i}}$ is of degree $a_i$.
		\item For $i \ne j$, the Frobenius orbits $\left[\xi_i\right]$ and $\left[\xi_j\right]$ are different.
		\item For every $i$, $\mu_i \vdash b_i > 0$ and $J_{\mu_i}\left(h_{\xi_i}\right) \in \GL_{a_i b_i}\left(\finiteField\right)$ is the generalized Jordan matrix corresponding to $h_{\xi_i}$ and the partition $\mu_i$ (see \Cref{subsec:conjugacy-classes-of-gl-n}).
		\item $c = \sum_{i=1}^r a_i b_i$.
	\end{enumerate}
	Then for any character $\alpha = \alpha_1 \times \dots \times \alpha_s$ as above
	$$ \ExoticKloosterman\left(\alpha, \fieldCharacter, h\right) = \left(-1\right)^{\left(k-1\right)c} q^{\left(k-1\right) \binom{c}{2}} \prod_{j=1}^r \htHallLittlewood_{\mu_j}\left(\omega_{1, \left[\xi_j\right]}, \dots, \omega_{k, \left[\xi_j\right]}; q^{a_j}\right),$$
	where $\htHallLittlewood_{\mu_j}$ is the modified Hall--Littlewood polynomial corresponding to the partition $\mu_j$, and where $\omega_{1, \left[\xi_j\right]}$, $\dots$, $\omega_{k, \left[\xi_j\right]}$ are the eigenvalues of the geometric Frobenius action on the geometric stalk at $\xi_j$ of the exotic Kloosterman sheaf $\ExoticKloosterman_{\finiteFieldExtension{a_j}}\left(\prod_{i=1}^s \finiteFieldExtension{k_i}, \alpha, \fieldCharacter\right)$.
\end{theorem}

\subsubsection{Relation to the representation theory of $p$-adic groups}
Since our objects and constructions are finite field analogs of similar ones originating in the representation theory of $p$-adic groups, it is natural to ask how they are related.

In~\cite{Zelingher2024b} the author established such relation using the mechanism of level zero representations. We summarize these results now.

Let $\localField$ be a non-archimedean local field with ring of integers $\ringOfIntegers$ and maximal ideal $\maximalIdeal$. Suppose that the residue field $\ringOfIntegers \slash \maximalIdeal$ of $\localField$ is isomorphic to $\finiteField$ and let $\quotientMap \colon \ringOfIntegers \slash \maximalIdeal \to \finiteField$ be the quotient map. Given an irreducible cuspidal representation $\tau$ of $\GL_k\left(\finiteField\right)$ and a character $\depthZeroCharacter \colon \multiplicativegroup{\localField} \to \multiplicativegroup{\cComplex}$ such that $\depthZeroCharacter \restriction_{\multiplicativegroup{\ringOfIntegers}} = \centralCharacter{\pi} \circ \quotientMap\restriction_{\multiplicativegroup{\ringOfIntegers}}$ we can construct an irreducible supercuspidal representation $\depthZeroRepresentation = \depthZeroRepresentation_{\tau, \depthZeroCharacter}$ of $\GL_k\left(\localField\right)$ via compact induction. The Speh representation $\SpehRepresentation{\tau}{c}$ of $\GL_{kc}\left(\finiteField\right)$ is a subrepresentation of the parabolic induction $\tau^{\circ c} = \tau \circ \dots \circ \tau$, while the Speh representation $\SpehRepresentation{\depthZeroRepresentation}{c}$ is a quotient of the parabolically induced representation $\depthZeroRepresentation^{\left(\frac{c-1}{2}, \frac{c-3}{2}, \dots, \frac{-\left(c-1\right)}{2}\right)}$, where for complex numbers $z_1,\dots,z_c \in \cComplex$,
$$\depthZeroRepresentation^{\left(z_1,\dots,z_c\right)} = \Ind{\ParabolicForSpeh{k}{c}\left(\localField\right)}{\GL_{kc}\left(\localField\right)}{\abs{\det}^{z_1}\depthZeroRepresentation \boxtimes \dots \boxtimes \abs{\det}^{z_c} \depthZeroRepresentation}.$$
In more detail, there exists a standard intertwining operator $\intertwiningOperator_{\weylElement{1^c}} \colon \depthZeroRepresentation^{\left(z_1,\dots,z_c\right)} \to \depthZeroRepresentation^{\left(z_c,\dots,z_2,z_1\right)}$. It converges for $\left(z_1,\dots,z_c\right)$ in a positive cone and has a meromorphic continuation to all $\left(z_1,\dots,z_c\right)$. This operator is holomorphic at the point $\left(z_1,\dots,z_c\right) = \left(\frac{c-1}{2}, \frac{c-3}{2}, \dots, -\frac{c-1}{2}\right)$. The image of the operator at that point is $\SpehRepresentation{\depthZeroRepresentation}{c}$. Using the Iwasawa decomposition, given $f \in \tau^{\circ c}$ we can define a flat section $\Lift^{\left(z_1,\dots,z_c\right)} f \in \holomorphicRepresentation$.

In~\cite{Zelingher2025} the author showed that the following diagram commutes
		$$\xymatrix{
	\tau^{\circ c} \ar[d]_{\Lift^{\left(\frac{c-1}{2},\dots,-\frac{c-1}{2}\right)}} \ar[rr]^{\ProjectionOperator_{\SpehRepresentation{\tau}{c}}}  & & \SpehRepresentation{\tau}{c} \ar[d]^{\Lift^{\left(-\frac{c-1}{2},\dots,\frac{c-1}{2}\right)}}\\
	\depthZeroRepresentation^{\left(\frac{c-1}{2},\dots,-\frac{c-1}{2}\right)} \ar[rr]^{M^{\ast}_{\weylElement{1^c}}} & & \depthZeroRepresentation^{\left(-\frac{c-1}{2},\dots,\frac{c-1}{2}\right)},}$$
	where $M^{\ast}_{\weylElement{1^c}}$ is an explicit scalar multiple of $M^{\left(\frac{c-1}{2}, \frac{c-3}{2}, \dots, -\frac{c-1}{2}\right)}_{\weylElement{1^c}}$ and $\ProjectionOperator_{\SpehRepresentation{\tau}{c}} \colon \tau^{\circ c} \to \SpehRepresentation{\tau}{c}$ is the projection operator. This shows a relation between the two versions of Speh representations. In particular, we have that $f \mapsto \Lift f \coloneq \Lift^{\left(-\frac{c-1}{2},\dots,\frac{c-1}{2}\right)} f$ is a map $\SpehRepresentation{\tau}{c} \to \SpehRepresentation{\depthZeroRepresentation}{c}$.
	
	The author also showed a relation between the $\kcNotation{k}{c}{\fieldCharacter}$ models of these Speh representations. If $\fieldCharacter \colon \ringOfIntegers \to \multiplicativegroup{\cComplex}$ is a non-trivial character with conductor $\maximalIdeal$, then $\fieldCharacter$ can be regarded as a character of $\finiteField = \ringOfIntegers \slash \maximalIdeal$. The author showed that one can choose $\kcNotation{k}{c}{\fieldCharacter}$ functionals $\ell_{\SpehRepresentation{\tau}{c}}$ and $\ell_{\SpehRepresentation{\depthZeroRepresentation}{c}}$ of $\SpehRepresentation{\tau}{c}$ and $\SpehRepresentation{\depthZeroRepresentation}{c}$, respectively, such that for every $f \in \SpehRepresentation{\tau}{c}$ and every $k_0 \in \GL_{kc}\left(\ringOfIntegers\right)$ the following equality holds:
	$$W_{\Lift f}\left(k_0\right) = W_f\left(\quotientMap\left(k_0\right)\right).$$
	Here $W_f \in \Whittaker\left(\SpehRepresentation{\tau}{c}, \fieldCharacterkc{k}{c}\right)$ is the $\kcNotation{k}{c}{\fieldCharacter}$ function corresponding to $f$ and $\ell_{\SpehRepresentation{\tau}{c}}$ via Frobenius reciprocity and $W_{\Lift f} \in \Whittaker\left(\SpehRepresentation{\depthZeroRepresentation}{c}, \fieldCharacterkc{k}{c} \right)$ is defined similarly.
	
	Finally, the author showed that for irreducible level zero supercuspidal representations $\Pi$ and $\depthZeroRepresentation$ of $\GL_c\left(\localField\right)$ and $\GL_k\left(\localField\right)$, respectively, constructed from irreducible cuspidal representations $\pi$ and $\tau$ of $\GL_c\left(\finiteField\right)$ and $\GL_k\left(\finiteField\right)$, respectively, we have that the Ginzburg--Kaplan gamma factors are related via the relation
	$$\LocalGKGammaFactor{s}{\Pi}{\depthZeroRepresentation}{\fieldCharacter} = \begin{dcases}
		\GKGammaFactor{\pi}{\tau}{\fieldCharacter} & \pi \ncong \Contragradient{\tau},\\
		\centralCharacter{\pi}\left(-1\right)^{c-1} \frac{q^{c\left(s-\frac{1}{2}\right)}}{\left(\centralCharacter{\Pi} \cdot \centralCharacter{\depthZeroRepresentation}\right)\left(\uniformizer\right) } \frac{1 - \left(\centralCharacter{\Pi} \cdot \centralCharacter{\depthZeroRepresentation}\right)\left(\uniformizer\right) q^{-cs}}{1 - \left(\centralCharacter{\Pi} \cdot \centralCharacter{\depthZeroRepresentation}\right)^{-1}\left(\uniformizer\right) q^{-c\left(1-s\right)}} & \pi \cong \Contragradient{\tau}.
	\end{dcases}$$
	
	It is interesting to ask if there exists an analog of the relation between Bessel--Speh functions and matrix Kloosterman sums in the local field case for more general representations of general linear groups. For this one would need to define the notions of Bessel--Speh functions and matrix Kloosterman sums in the local field case. For $c=1$, Bessel functions were defined for irreducible supercuspidal representations by Paskunas--Stevens in~\cite{PaskunasStevens2008}, but even in this case we are not aware of explicit formulas expressing their special values as Kloosterman sums. The closest formulas we are aware of are those related to the work of Ichino--Templier on the Voronoi summation formula~\cite{IchinoTemplier2013}, see for example~\cite[Section 4]{Assing2024}. 
\subsection{Outline of the paper}
In this paper we define the notion of exotic matrix Kloosterman sums, study their properties, and relate them to the special values  $\specialBesselSpeh{\tau}$ of Bessel--Speh functions. It turns out that it is convenient to study exotic matrix Kloosterman sums with their counterpart non-abelian exotic Gauss sums.

In \Cref{sec:classical-exotic-exponential-sums}, we define the notion of what we call exotic Gauss sums. These sums appear when one computes the tensor product $\varepsilon_0$-factor attached to a pair of irreducible representations of $\GL_c\left(\finiteField\right)$ and $\GL_k\left(\finiteField\right)$~\cite{ye2021epsilon}. We prove a Hasse--Davenport lifting relation for these exotic Gauss sums (\Cref{thm:hasse-davenport-for-exotic-gauss-sums}). We then recall the definition and properties of exotic Kloosterman sums and their corresponding sheaves. Most importantly, in \Cref{subsec:roots-of-kloosterman-sum-at-xi} we define the notion of the roots of $\ExoticKloosterman_{\finiteFieldExtension{a}}\left(\alpha, \fieldCharacter\right)$ at $\xi \in \multiplicativegroup{\finiteFieldExtension{a}}$, which is needed for the explicit formula for exotic matrix Kloosterman sums in terms of modified Hall--Littlewood polynomials that we give later (\Cref{thm:general-formula-for-exotic-kloosterman-sum-of-conjugacy-class}).

In \Cref{sec:exotic-matrix-exponential-sums}, we define the notions of what we call non-abelian exotic Gauss sums and exotic Kloosterman sums. As discussed above, these definitions make use of Shintani's norm map. We begin with recalling the representation theory of $\GL_n\left(\finiteField\right)$. Then we recall Kondo's non-abelian Gauss sums and their explicit computation. We then explain the Shintani norm map and Shintani lifts and state results of Silberger--Zink~\cite{SilbergerZink08} regarding an explicit description of the cuspidal support of the Shintani lift of a given irreducible representation. Next in \Cref{subsec:non-abelian-exotic-gauss-sums}, we define the notion of non-abelian exotic Gauss sums and use the results above to show that they can be expressed as products of exotic Gauss sums (\Cref{thm:exotic-gauss-sum-of-composite-character}). In \Cref{subsec:exotic-matrix-kloosterman-sum}, we use these non-abelian exotic Gauss sums to define the notion of exotic matrix Kloosterman sums.

Sections \ref{sec:relation-to-bessel-speh-functions} and \ref{sec:relation-to-hall-littlewood-polynomials} are independent. In \Cref{sec:relation-to-bessel-speh-functions}, we prove a relation similar to \eqref{eq:identity-for-principal-series} for arbitrary irreducible generic representations $\tau$ of $\GL_k\left(\finiteField\right)$. We begin with recalling the notions of $\varepsilon_0$-factors, $\kcNotation{k}{c}{\fieldCharacter}$-vectors, generic representations, Speh representations and their corresponding Bessel--Speh functions. We then use our results with Carmon from~\cite{CarmonZelingher2025} and our results from~\cite{Zelingher2024b} to deduce \Cref{thm:introduction-bessel-speh-is-an-exotic-kloosterman-sum-generic} (\Cref{thm:bessel-speh-is-an-exotic-kloosterman-sum-generic}), which generalizes \eqref{eq:identity-for-principal-series}. We then deduce the multiplicativity property (\Cref{thm:multiplicativity-of-exotic-kloosterman-sums-without-proof}) of exotic matrix Kloosterman sums, analogous to the one proved in~\cite[Theorem 4.1]{CarmonZelingher2025} for twisted matrix Kloosterman sums.

In \Cref{sec:relation-to-hall-littlewood-polynomials}, we generalize our previous results from~\cite{Zelingher2023} and prove \Cref{thm:introduction-general-formula-for-exotic-kloosterman-sum-of-conjugacy-class} which shows that exotic matrix Kloosterman sums can be expressed as products of modified Hall--Littlewood polynomials evaluated at the roots of $\ExoticKloosterman\left(\alpha, \fieldCharacter\right)$ at the eigenvalues of the matrix in question (\Cref{thm:general-formula-for-exotic-kloosterman-sum-of-conjugacy-class}). Our method of proof is different from the method we used in~\cite{Zelingher2023}. We use Macdonald's characteristic maps and compute the element that the ``global exotic matrix Kloosterman sum function'' corresponds to under these maps (Theorems \ref{thm:macdonald-exotic-kloosterman-sum-geometric-basis} and \ref{thm:global-matrix-kloosterman-sum-in-character-basis}). We would like to mention that Theorems \ref{thm:macdonald-exotic-kloosterman-sum-geometric-basis} and \ref{thm:global-matrix-kloosterman-sum-in-character-basis} are the first results we are aware of that explain how theories of tensor product gamma factors are related to the classical treatment of representation theory of finite general linear groups due to Green~\cite{Green55} and Macdonald~\cite{macdonald1998symmetric}.

In \Cref{sec:applications} we show some applications. In \Cref{sec:identities-for-bessel-functions} we combine our results with our previous results from~\cite{zelingher2022values} to find formulas expressing Bessel functions for irreducible generic representations of general linear groups in terms of exotic matrix Kloosterman sums. In \Cref{example:curtis-shinoda-trick} we give an interesting explicit formula for values of Bessel functions at elements of the form $\left(\begin{smallmatrix}
	& & \IdentityMatrix{k - c} \\
	& t_2 \IdentityMatrix{c-1} \\	
	t_1
\end{smallmatrix}\right)$ in terms of effective zero-cycles and symmetric powers of exotic Kloosterman sheaves. In \Cref{subsec:inequalities} we use Deligne's Weil bound to deduce upper bounds for the special values $\specialBesselSpeh{\tau}$ and for exotic matrix Kloosterman sums. In \Cref{subsec:formulas-for-spherical-elements} we show that the formulas from this present paper are very similar to a recent Casselman--Shalika type formula we found in~\cite{Zelingher2025} for spherical $\kcNotation{k}{c}{\fieldCharacter}$ functions attached to unramified Speh representations.

Most of our proofs are very short and simple. This is because we rely on heavy results from other papers such as Shintani's correspondence~\cite{shintani1976two}, the classification of Shintani lifts by Silberger--Zink~\cite{SilbergerZink08}, Kondo's Gauss sum computation~\cite{Kondo1963}, Macdonald's characteristic maps~\cite{macdonald1998symmetric}, and our previous results with Carmon~\cite{CarmonZelingher2025}. Direct proofs of the results of this present paper seem difficult without these established powerful tools.

\subsection*{Acknowledgments}
I would like to thank Oded Carmon for many conversations about Speh representations and $\kcNotation{k}{c}{\fieldCharacter}$ vectors and for his comments on earlier versions of this paper. I would also like to thank Eyal Kaplan for his helpful comments on an earlier version of this paper. Finally, I would like to thank the anonymous referee for their suggestions which improved the mathematical exposition of this paper.

\tableofcontents

\section{Exotic exponential sums}\label{sec:classical-exotic-exponential-sums}
In this section we recall the notions of what we call exotic Gauss sums and of exotic Kloosterman sums and sheaves and their properties.

Let $\finiteField$ be a finite field with $q$ elements. Fix an algebraic closure $\algebraicClosure{\finiteField}$ of $\finiteField$. For any $k \ge 1$, let $\finiteFieldExtension{k} \slash \finiteField$ be the unique field extension of degree $k$ in $\algebraicClosure{\finiteField}$. If $m \mid n$, let $\FieldNorm{n}{m} \colon \multiplicativegroup{\finiteFieldExtension{n}} \to \multiplicativegroup{\finiteFieldExtension{m}}$ be the norm map.

Let $\fieldCharacter \colon \finiteField \to \multiplicativegroup{\cComplex}$ be a non-trivial character. For any $k \ge 1$, let $\fieldCharacter_{k} \colon \finiteFieldExtension{k} \to \multiplicativegroup{\cComplex}$ be the character $\fieldCharacter_k = \fieldCharacter \circ \trace_{\finiteFieldExtension{k} \slash \finiteField}$.

Recall that a partition $\lambda = \left(n_1,\dots,n_r\right)$ of a non-negative integer $n$ is a weakly decreasing sequence of positive integers such that $n_1 + \dots + n_r = n$. We denote by $\sizeof{\lambda} = n$ the \emph{size} of the partition $\lambda$, and by $\lengthof\left(\lambda\right) = r$ the \emph{length} of the partition $\lambda$. We will write $\lambda \vdash n$ to indicate that $\lambda$ is a partition of $n$. We will denote by $\left(\right)$ the empty partition of $0$. We denote by $\Partitions$ the set consisting of all partitions of all non-negative integers.

\subsection{Tensor product of finite fields}\label{subsec:tensor-product-of-finite-fields}

Let $R$ and $S$ be commutative $\finiteField$-algebras. Suppose that $R$ and $S$ are of finite dimension over $\finiteField$. Consider the tensor product $R \otimes_{\finiteField} S$. It is a finitely generated $R$-algebra equipped with the action $r \cdot_1 x = \left(r \otimes 1\right)x$, and similarly a finitely generated $S$-algebra equipped with the action $s \cdot_2 x = \left(1 \otimes s\right)x$, where $r \in R$, $s \in S$ and $x \in R \otimes_{\finiteField} S$. Given an element $z \in R \otimes_{\finiteField} S$, consider the multiplication by $z$ map, $T_z \colon R \otimes_{\finiteField} S \to R \otimes_{\finiteField} S$, defined by $T_z\left(x\right) = zx$. We define two norm maps $\TensorProductNormOne{R \otimes_{\finiteField} S} \colon R \otimes_{\finiteField} S \to R$ and $\TensorProductNormTwo{R \otimes_{\finiteField} S} \colon R \otimes_{\finiteField} S \to S$ as follows. Given $z \in R \otimes_{\finiteField} S$, we define $\TensorProductNormOne{R \otimes_{\finiteField} S}\left(z\right)$ to be the determinant of $T_z$, regarded as an $R$-linear map. Similarly, we define $\TensorProductNormTwo{R \otimes_{\finiteField} S}\left(z\right)$ to be the determinant of $T_z$, regarded as an $S$-linear map. We also define $\trace \colon R \otimes_{\finiteField} S \to \finiteField$ as follows. For $z \in R \otimes_{\finiteField} S$, let $\trace z = \trace_{\slash \finiteField}z$ be the trace of the map $T_z$, regarded as an $\finiteField$-linear map.

Suppose that $R = \finiteFieldExtension{n}$ and $S = \finiteFieldExtension{m}$. In this case, we denote $\TensorProductNormOne{n,m} = \TensorProductNormOne{\finiteFieldExtension{n} \otimes_{\finiteField} \finiteFieldExtension{m}} \colon \finiteFieldExtension{n} \otimes_{\finiteField} \finiteFieldExtension{m} \to \finiteFieldExtension{n}$ and $\TensorProductNormTwo{n,m} = \TensorProductNormTwo{\finiteFieldExtension{n} \otimes_{\finiteField} \finiteFieldExtension{m}} \colon \finiteFieldExtension{n} \otimes_{\finiteField} \finiteFieldExtension{m} \to \finiteFieldExtension{m}$.

Let $l = \lcm\left(n,m\right)$ and let $d=\gcd\left(n,m\right)$. As explained in~\cite[Section 3.2]{zelingher2022values}, the tensor product $\finiteField$-algebra $\finiteFieldExtension{n} \otimes_{\finiteField} \finiteFieldExtension{m}$ is (non-canonically) isomorphic to the ring $\finiteFieldExtension{l}^{d}$ equipped with the actions $$r \cdot_1 \left(z_1,\dots,z_d\right) = \left(r z_1, \dots, r z_d\right)$$ and $$s \cdot_2 \left(z_1,\dots, z_d\right) = \left(s z_1, s^{1/q} z_2, \dots, s^{1 / q^{d-1}} z_d\right),$$ where $\left(z_1,\dots,z_d\right) \in \finiteFieldExtension{l}^d$, $r \in \finiteFieldExtension{n}$ and $s \in \finiteFieldExtension{m}$. Moreover, under this identification, the norm maps take the form $$ \TensorProductNormOne{n,m}\left(z_1,\dots,z_d\right) = \FieldNorm{l}{n}\left(z_1 z_2 \dots z_d\right),$$ and
$$\TensorProductNormTwo{n,m}\left(z_1,\dots,z_d\right) = \FieldNorm{l}{m}\left(z_1 z_2^q \dots z_d^{q^{d-1}}\right),$$
and the trace map takes the form $$\trace\left(z_1,\dots,z_d\right) = \trace_{\finiteFieldExtension{l} \slash \finiteField}\left(z_1 + \dots + z_d\right).$$ 

Let $k$ be a positive integer and let $\lambda = \left(k_1,\dots,k_s\right) \vdash k$ be a partition of $k$. Let $\finiteFieldExtension{\lambda}$ be the $\finiteField$-\etale algebra $\finiteFieldExtension{\lambda} = \prod_{i=1}^s \finiteFieldExtension{k_i}$. Take $R = \finiteFieldExtension{\lambda}$ and $S = \finiteFieldExtension{m}$. Then in this case for $\left(x_1,\dots,x_s\right) \in \finiteFieldExtension{\lambda} \otimes_{\finiteField} \finiteFieldExtension{m} = \prod_{i=1}^s \finiteFieldExtension{k_i} \otimes_{\finiteField} \finiteFieldExtension{m}$ we have $$\TensorProductNormOne{\finiteFieldExtension{\lambda} \otimes_{\finiteField} \finiteFieldExtension{m}}\left(x_1,\dots,x_s\right) = \left(\TensorProductNormOne{k_1, m}\left(x_1\right),\dots,\TensorProductNormOne{k_s,m}\left(x_s\right)\right),$$
and $$\TensorProductNormTwo{\finiteFieldExtension{\lambda} \otimes_{\finiteField} \finiteFieldExtension{m}}\left(x_1,\dots,x_s\right) = \prod_{i=1}^s \TensorProductNormTwo{k_i, m}\left(x_i\right).$$
Also, the trace is given in this case by $$\trace \left(x_1,\dots,x_s\right) = \sum_{i=1}^s \trace x_i.$$

\subsection{Exotic Gauss sums}\label{subsec:exotic-gauss-sums}

In this section, we define the notion of exotic Gauss sums. These sort of sums appear when one considers the tensor product $\varepsilon_0$-factors attached to a pair of irreducible representations of finite general linear groups, see~\cite[Theorem 4.2]{ye2021epsilon} and~\cite[Section 3.2 and Proposition 4.2]{zelingher2022values}. Our definition involves two steps. We first define an exotic Gauss sum for a pair of characters $\alpha \colon \multiplicativegroup{\finiteFieldExtension{k}} \to \multiplicativegroup{\cComplex}$ and $\chi \colon \multiplicativegroup{\finiteFieldExtension{m}} \to \multiplicativegroup{\cComplex}$. We then use this definition to define an exotic Gauss sum for a pair of characters $\alpha \colon \multiplicativegroup{\finiteFieldExtension{\lambda}} \to \multiplicativegroup{\cComplex}$ and $\chi \colon \multiplicativegroup{\finiteFieldExtension{m}} \to \multiplicativegroup{\cComplex}$, where $\finiteFieldExtension{\lambda}$ is an \etale algebra over $\finiteField$ corresponding to a partition $\lambda$. We establish Hasse--Davenport lifting relations for these exotic Gauss sums. We begin this section with recalling the notions of Frobenius orbits of characters and of Gauss sums.

\subsubsection{Frobenius orbits of characters}
Let $m \ge 1$ and let $\chi \colon \multiplicativegroup{\finiteFieldExtension{m}} \to \multiplicativegroup{\cComplex}$ be a character. The \emph{Frobenius orbit} of $\chi$ is the set $\left[\chi\right] \coloneq \{ \chi^{q^j} \mid j \ge 0 \}$. The \emph{degree} of $\chi$, denoted $\deg \chi$, is the size of its Frobenius orbit. We say that $\chi$ is \emph{regular} if $\deg \chi = m$. If $\chi$ is of degree $d$, then $1 \le d \mid m$ and there exists a character $\chi' \colon \multiplicativegroup{\finiteFieldExtension{d}} \to \multiplicativegroup{\cComplex}$ such that $\chi = \chi' \circ \FieldNorm{m}{d}$. The degree of $\chi$ is the minimal number $d$ with the properties described in the last sentence.

\subsubsection{Gauss sums}
Let $\chi \colon \multiplicativegroup{\finiteFieldExtension{m}} \to \multiplicativegroup{\cComplex}$ be a character. We define the \emph{Gauss sum} associated with $\chi$ and $\fieldCharacter$ by the formula
$$\tau\left(\chi, \fieldCharacter_{m}\right) = -\sum_{x \in \multiplicativegroup{\finiteFieldExtension{m}}} \chi\left(x\right) \fieldCharacter_m\left(x\right).$$
Note that $\tau\left(\chi, \fieldCharacter_m\right)$ is constant on the Frobenius orbit $\left[\chi\right]$.

We have the classical Hasse--Davenport lifting relation~\cite{DavenportHasse1935}, which we recall now. See also~\cite[Theorem 5.14]{LidlNiederreiter1994}. For any $b \ge 1$, the following identity holds:
$$\tau\left(\chi, \fieldCharacter_{m}\right)^b = \tau\left(\chi \circ \FieldNorm{bm}{m}, \fieldCharacter_{bm}\right).$$

\subsubsection{Exotic Gauss sums for $\finiteFieldExtension{k}$}
Let $k$ and $m$ be positive integers. Given  characters $\alpha \colon \multiplicativegroup{\finiteFieldExtension{k}} \to \multiplicativegroup{\cComplex}$ and $
\chi \colon \multiplicativegroup{\finiteFieldExtension{m}} \to \multiplicativegroup{\cComplex}$, we define the \emph{exotic Gauss sum}
$$\GaussSumCharacter{k,m}{\alpha}{\chi}{\fieldCharacter} = \left(-1\right)^{k+m+km} \sum_{t \in \multiplicativegroup{\left(\finiteFieldExtension{k} \otimes_{\finiteField} \finiteFieldExtension{m}\right)}} \alpha\left( \TensorProductNormOne{\finiteFieldExtension{k} \otimes_{\finiteField} \finiteFieldExtension{m}} \left(t\right) \right) \chi\left( \TensorProductNormTwo{\finiteFieldExtension{k} \otimes_{\finiteField} \finiteFieldExtension{m}} \left(t\right)\right) \fieldCharacter\left(\trace t\right).$$
Using the explicit description of the norm and trace maps, and using the fact that $\gcd\left(k,m\right)$ and $k+m+km$ have the same parity, we have the equality
\begin{equation}\label{eq:exotic-gauss-sum-is-a-product-of-exotic-gauss-sums}
	\GaussSumCharacter{k,m}{\alpha}{\chi}{\fieldCharacter} = \prod_{j = 0}^{\gcd\left(k,m\right) - 1} \tau\left(\alpha \circ \FieldNorm{\lcm\left(k,m\right)}{k} \cdot \chi^{q^j} \circ \FieldNorm{\lcm\left(k,m\right)}{m}, \fieldCharacter_{\lcm\left(k, m\right)}\right).
\end{equation}
When $m = 1$, we have \begin{equation}\label{eq:twisted-gauss-sum}
	\GaussSumCharacter{k,1}{\alpha}{\chi}{\fieldCharacter} = \tau\left(\alpha \cdot \chi \circ \FieldNorm{k}{1}, \fieldCharacter_k\right) = -\sum_{x \in \multiplicativegroup{\finiteFieldExtension{k}}} \alpha\left(x\right) \chi\left(\FieldNorm{k}{1}\left(x\right)\right) \fieldCharacter_k\left(x\right),
\end{equation}
and we simply denote $$\tau\left(\alpha \times \chi, \fieldCharacter_k\right) = \GaussSumCharacter{k,1}{\alpha}{\chi}{\fieldCharacter},$$
and refer to $\tau\left(\alpha \times \chi, \fieldCharacter_k\right)$ as a \emph{twisted Gauss sum}.

We have the following Hasse--Davenport lifting relations.
\begin{proposition}\label{prop:hasse-davenport-cuspidal}
	For any $b \ge 1$, we have the equalities
	\begin{equation*}
		\GaussSumCharacter{k,m}{\alpha}{\chi}{\fieldCharacter}^b = \GaussSumCharacter{bk,m}{\alpha \circ \FieldNorm{bk}{k}}{\chi}{\fieldCharacter},
	\end{equation*}
	and
	\begin{equation*}
	\GaussSumCharacter{k,m}{\alpha}{\chi}{\fieldCharacter}^b = \GaussSumCharacter{k,bm}{\alpha}{\chi \circ \FieldNorm{bm}{m}}{\fieldCharacter}.
\end{equation*}
\end{proposition}
\begin{proof}
	Since $\GaussSumCharacter{k,m}{\alpha}{\chi}{\fieldCharacter} = \GaussSumCharacter{m,k}{\chi}{\alpha}{\fieldCharacter}$, it suffices to prove the first equality.
	
	Denote $d = \gcd\left(k,m\right)$ and $l = \lcm\left(k,m\right)$. Also denote $d' = \gcd\left(bk, m\right)$ and $l' = \lcm\left(bk, m\right)$. Then $l \mid l'$ and $d \mid d'$. We have that $bld = l' d'$. Write
	$$ \GaussSumCharacter{k,m}{\alpha}{\chi}{\fieldCharacter}^b = \prod_{j'=1}^{\frac{d'}{d}} \prod_{j = 0}^{d - 1} \tau\left(\alpha \circ \FieldNorm{l}{k} \cdot \chi^{q^j} \circ \FieldNorm{l}{m}, \fieldCharacter_{l}\right)^{\frac{l'}{l}}.$$
	Then by the Hasse--Davenport lifting relation,
	$$ \GaussSumCharacter{k,m}{\alpha}{\chi}{\fieldCharacter}^b = \prod_{j'=1}^{\frac{d'}{d}} \prod_{j = 0}^{d - 1} \tau\left(\alpha \circ \FieldNorm{l'}{k} \cdot \chi^{q^j} \circ \FieldNorm{l'}{m}, \fieldCharacter_{l'}\right). $$
	Note that $\alpha^{q^{j + d}} \circ \FieldNorm{l'}{k}$ lies in the same $\Frobenius^m$-orbit as $\alpha^{q^{j}} \circ \FieldNorm{l'}{k}$ for every $j$ (as $d$ is an integral linear combination of $k$ and $m$), and therefore we may rewrite the last equation as \begin{equation*}
		\GaussSumCharacter{k,m}{\alpha}{\chi}{\fieldCharacter}^b = \prod_{j = 0}^{d' - 1} \tau\left(\alpha \circ \FieldNorm{l'}{k} \cdot \chi^{q^j} \circ \FieldNorm{l'}{m}, \fieldCharacter_{l'}\right),
	\end{equation*}
	which proves the theorem because $\alpha \circ \FieldNorm{l'}{k} = \alpha \circ \FieldNorm{bk}{k} \circ \FieldNorm{l'}{bk}$.
\end{proof}

\subsubsection{Exotic Gauss sums for $\finiteFieldExtension{\lambda}$}\label{subsec:exotic-gauss-sums-for-composite}
Let $k$ be a positive integer and let $\lambda = \left(k_1,\dots,k_s\right) \vdash k$ be a partition of $k$. Let $\alpha_1 \colon \multiplicativegroup{\finiteFieldExtension{k_1}} \to \multiplicativegroup{\cComplex}$, $\dots$, $\alpha_s \colon \multiplicativegroup{\finiteFieldExtension{k_s}} \to \multiplicativegroup{\cComplex}$ be multiplicative characters. Write $\multiplicativegroup{\finiteFieldExtension{\lambda}} = \prod_{i=1}^s \multiplicativegroup{\finiteFieldExtension{k_i}}$, and let $\alpha \colon \multiplicativegroup{\finiteFieldExtension{\lambda}} \to \cComplex$ be the character given by $\alpha = \alpha_1 \times \dots \times \alpha_s$. For any $m \ge 1$ and any character $\chi \colon \multiplicativegroup{\finiteFieldExtension{m}} \to \multiplicativegroup{\cComplex}$, we define the \emph{exotic Gauss sum}
$$\GaussSumCharacter{\lambda,m}{\alpha}{\chi}{\fieldCharacter} = \prod_{i = 1}^s \GaussSumCharacter{k_i,m}{\alpha_i}{\chi}{\fieldCharacter}.$$
By the description of the norm and the trace, it is given by
$$\GaussSumCharacter{\lambda,m}{\alpha}{\chi}{\fieldCharacter} = \left(-1\right)^{k+sm+mk} \sum_{t \in \multiplicativegroup{\left(\finiteFieldExtension{\lambda} \otimes_{\finiteField} \finiteFieldExtension{m}\right)}} \alpha\left( \TensorProductNormOne{\finiteFieldExtension{\lambda} \otimes_{\finiteField} \finiteFieldExtension{m}} \left(t\right) \right) \chi\left( \TensorProductNormTwo{\finiteFieldExtension{\lambda} \otimes_{\finiteField} \finiteFieldExtension{m}} \left(t\right)\right) \fieldCharacter\left(\trace t\right).$$
Using the Hasse--Davenport lifting relation (\Cref{prop:hasse-davenport-cuspidal}), we obtain the following relation.
\begin{theorem}\label{thm:hasse-davenport-for-exotic-gauss-sums}
	For any $b \ge 1$, we have the equality
	$$\GaussSumCharacter{\lambda,m}{\alpha}{\chi}{\fieldCharacter}^b = \GaussSumCharacter{\lambda, bm}{\alpha}{\chi \circ \FieldNorm{bm}{m}}{\fieldCharacter}.$$
\end{theorem}

\subsection{Exotic Kloosterman sums and sheaves}
In this section, we review the definition of exotic Kloosterman sums and their associated $L$-functions. We then define the notion of the roots of $\ExoticKloosterman_{\finiteFieldExtension{a}}\left(\alpha, \fieldCharacter\right)$ at $\xi \in \multiplicativegroup{\finiteFieldExtension{a}}$, which will be important for the Hall--Littlewood identity of exotic matrix Kloosterman sums we will give later (\Cref{thm:general-formula-for-exotic-kloosterman-sum-of-conjugacy-class}).

\subsubsection{Exotic Kloosterman sums}\label{subsec:exotic-kloosterman-sums}

Let us keep the notation as in \Cref{subsec:exotic-gauss-sums-for-composite}. For any $m \ge 1$ and any $\xi \in \multiplicativegroup{\finiteFieldExtension{m}}$, we define the \emph{exotic Kloosterman sum}
$$ \ExoticKloosterman_{m}\left(\alpha, \fieldCharacter, \xi\right) = \sum_{\substack{x \in \multiplicativegroup{\left(\finiteFieldExtension{\lambda} \otimes_{\finiteField} \finiteFieldExtension{m}\right)}\\
\TensorProductNormTwo{\finiteFieldExtension{\lambda} \otimes_{\finiteField} \finiteFieldExtension{m}}\left(x\right) = \xi}} \alpha\left(\TensorProductNormOne{\finiteFieldExtension{\lambda} \otimes_{\finiteField} \finiteFieldExtension{m}}\left(x\right)\right) \fieldCharacter\left(\trace x\right).$$

We also define the normalized version
$$ \ExoticKloostermanNormalized_m\left(\alpha, \fieldCharacter, \xi\right) = q^{-\frac{\left(k-1\right)m}{2}} \ExoticKloosterman_{m}\left(\alpha, \fieldCharacter, \xi\right).$$
For $x \in \multiplicativegroup{\finiteField}$, we denote $$\ExoticKloosterman\left(\alpha, \fieldCharacter, x\right) = \ExoticKloosterman_1\left(\alpha, \fieldCharacter, x\right) \,\,\,\, \text{ and }\,\,\,\, \ExoticKloostermanNormalized\left(\alpha, \fieldCharacter, x\right) = \ExoticKloostermanNormalized_1\left(\alpha, \fieldCharacter, x\right).$$

\begin{remark}
	If $\lambda = \left(1,1,\dots,1\right) \vdash k$ and $\alpha_1,\dots,\alpha_k \colon \multiplicativegroup{\finiteField} \to \cComplex$, we have $\finiteFieldExtension{\lambda} = \finiteField^k$ and $\finiteFieldExtension{\lambda} \otimes_{\finiteField} \finiteFieldExtension{m} = \finiteFieldExtension{m}^k$. In this case
	$$ \ExoticKloosterman_m\left(\alpha, \fieldCharacter, \xi\right)  = \sum_{\substack{x_1,\dots,x_k \in \multiplicativegroup{\finiteFieldExtension{m}} \\
	x_1 x_2 \dots x_k = \xi}} \alpha_1\left(\FieldNorm{m}{1}\left(x_1\right)\right) \dots \alpha_k\left(\FieldNorm{m}{1}\left(x_k\right)\right) \fieldCharacter_m \left(x_1 + \dots + x_k\right).$$
This is the twisted Kloosterman sum from~\cite[Section 2.1]{Zelingher2023}.
\end{remark}

For any character $\chi \colon \multiplicativegroup{\finiteFieldExtension{m}} \to \multiplicativegroup{\cComplex}$, we have the relation \begin{equation}\label{eq:gauss-sum-as-kloosterman-sum-transform}
	\GaussSumCharacter{\lambda, m}{\alpha}{\chi}{\fieldCharacter} = \left(-1\right)^{k+sm+mk} \sum_{\xi \in \multiplicativegroup{\finiteFieldExtension{m}}} \ExoticKloosterman_m\left(\alpha, \fieldCharacter, \xi\right) \chi\left(\xi\right).
\end{equation}

We will often be interested in specifying the base field over which an exotic Kloosterman sum is taken. For a positive integer $a$ and $\xi \in \multiplicativegroup{\finiteFieldExtension{am}}$, we denote $$\ExoticKloosterman_{m, \finiteFieldExtension{a}}\left(\alpha, \fieldCharacter, \xi\right) = \sum_{\substack{x \in \multiplicativegroup{\left(\left(\finiteFieldExtension{\lambda} \otimes_{\finiteField} \finiteFieldExtension{a}\right) \otimes_{\finiteFieldExtension{a}} \finiteFieldExtension{am} \right)}\\
\TensorProductNormTwo{\left(\finiteFieldExtension{\lambda} \otimes_{\finiteField} \finiteFieldExtension{a}\right) \otimes_{\finiteFieldExtension{a}} \finiteFieldExtension{am}}\left(x\right) = \xi}} \alpha\left(\TensorProductNormOne{\finiteFieldExtension{\lambda} \otimes_{\finiteField} \finiteFieldExtension{a}}  \left(\TensorProductNormOne{\left(\finiteFieldExtension{\lambda} \otimes_{\finiteField} \finiteFieldExtension{a}\right) \otimes_{\finiteFieldExtension{a}} \finiteFieldExtension{am}}\left(x\right)\right) \right) \fieldCharacter_a\left(\trace_{\slash \finiteFieldExtension{a}}\left(x\right)\right),$$ and the normalized version
$$\ExoticKloostermanNormalized_{m, \finiteFieldExtension{a}}\left(\alpha, \fieldCharacter, \xi\right) = q^{-\frac{\left(k-1\right)am}{2}} \ExoticKloosterman_{m, \finiteFieldExtension{a}}\left(\alpha, \fieldCharacter, \xi\right).$$
Note that these sums are constant on Frobenius orbits,
and that we have the equalities $$\ExoticKloosterman_{m, \finiteFieldExtension{a}}\left(\alpha, \fieldCharacter, \xi\right) = \ExoticKloosterman_{ma}\left(\alpha, \fieldCharacter, \xi\right)\,\,\,\, \text{ and }\,\,\,\, \ExoticKloostermanNormalized_{m, \finiteFieldExtension{a}}\left(\alpha, \fieldCharacter, \xi\right) = \ExoticKloostermanNormalized_{ma}\left(\alpha, \fieldCharacter, \xi\right).$$ 

\subsubsection{The roots of $\ExoticKloosterman\left(\alpha, \fieldCharacter\right)$ at $\xi$}\label{subsec:roots-of-kloosterman-sum-at-xi}

The $L$-function attached to the family of exotic Kloosterman sums $\left(\ExoticKloosterman_{m, \finiteFieldExtension{a}}\left(\alpha,\fieldCharacter,\xi \right)\right)_{m=1}^{\infty}$, where $\xi \in \multiplicativegroup{\finiteFieldExtension{a}}$, is defined as the following generating function:
$$L\left(T, \ExoticKloosterman_{\finiteFieldExtension{a}}\left(\alpha, \fieldCharacter, \xi\right) \right) = \exp\left(\sum_{m = 1}^{\infty}  \frac{\ExoticKloosterman_{m,\finiteFieldExtension{a}}\left(\alpha, \fieldCharacter, \xi\right)}{m} T^m\right).$$
It turns out that $L\left(T, \ExoticKloosterman_{\finiteFieldExtension{a}}\left(\alpha, \fieldCharacter, \xi\right) \right)^{\left(-1\right)^k}$ is a polynomial of degree $k$. See for example the introduction of~\cite{FuWan2005}. The value of this polynomial at $T = 0$ is $1$. Let us write
$$L\left(T, \ExoticKloosterman_{\finiteFieldExtension{a}}\left(\alpha, \fieldCharacter, \xi\right) \right)^{\left(-1\right)^k} = \prod_{j=1}^k \left(1 - \omega_j T\right),$$ where $\omega_1,\dots,\omega_k$ are complex numbers. We call $\omega_1,\dots,\omega_k$ \emph{the roots of $\ExoticKloosterman_{\finiteFieldExtension{a}}\left(\alpha, \fieldCharacter\right)$ at $\xi$}. For any $m \ge 1$, we have the equality $$\omega_1^m + \dots + \omega_k^m = \left(-1\right)^{k-1} \ExoticKloosterman_{m,\finiteFieldExtension{a}}\left(\alpha, \fieldCharacter, \xi\right).$$

We introduce a normalized version of $L\left(T, \ExoticKloosterman_{\finiteFieldExtension{a}}\left(\alpha, \fieldCharacter, \xi\right) \right)$ with respect to which some of our results are cleaner. This normalized version is defined as $$L^{\ast}\left(T, \ExoticKloosterman_{\finiteFieldExtension{a}}\left(\alpha, \fieldCharacter, \xi\right) \right) = L\left(q^{-\frac{a\left(k-1\right)}{2}} T, \ExoticKloosterman_{\finiteFieldExtension{a}}\left(\alpha, \fieldCharacter, \xi\right) \right)^{\left(-1\right)^k}.$$
We remark that this is not exactly the same normalized version as in~\cite{Zelingher2023}. We have that $$L^{\ast}\left(T, \ExoticKloosterman_{\finiteFieldExtension{a}}\left(\alpha, \fieldCharacter, \xi\right) \right) = \prod_{j=1}^k \left(1 - q^{-\frac{a \left(k-1\right)}{2}} \omega_j T\right),$$ and we call $\omega_1^{\ast} = q^{-\frac{a\left(k-1\right)}{2}} \omega_1$, $\dots$, $\omega_k^{\ast} = q^{-\frac{a\left(k-1\right)}{2}} \omega_k$ \emph{the normalized roots of $\ExoticKloosterman_{\finiteFieldExtension{a}}\left(\alpha, \fieldCharacter\right)$ at $\xi$}. For any $m \ge 1$, they satisfy the equality
\begin{equation}\label{eq:power-sums-of-roots-at-xi}
	\left(\omega_1^{\ast}\right)^m + \dots + \left(\omega_k^{\ast}\right)^m = \left(-1\right)^{k-1} \ExoticKloostermanNormalized_{m, \finiteFieldExtension{a}}\left(\alpha,\fieldCharacter,\xi\right).
\end{equation}

For some of our results, it will be convenient to have a linear map from a complex vector space of dimension $k$ to itself, whose eigenvalues are $\omega_1, \dots, \omega_k$. Such linear map is provided using the mechanism of exotic Kloosterman sheaves defined by Katz~\cite[Sections 8.8.4-8.8.7]{katz2016gauss}. See also~\cite[Page 152]{katz1993estimates} and~\cite[Appendix B]{nien2021converse}. We explain this now.

Let $\ell$ be a prime number different than the characteristic of $\finiteField$. Fix an embedding $\ladicnumbers \hookrightarrow \cComplex$ and regard $\alpha \colon \multiplicativegroup{\finiteFieldExtension{\lambda}} \to \multiplicativegroup{\ladicnumbers}$ and $\fieldCharacter \colon \finiteField \to \multiplicativegroup{\ladicnumbers}$. Consider the following diagram of schemes over $\finiteFieldExtension{a}$:
$$ \xymatrix{& \restrictionOfScalars{\left(\finiteFieldExtension{\lambda} \otimes_{\finiteField} \finiteFieldExtension{a}\right)}{\finiteFieldExtension{a}}{\multiplcativeScheme} \ar[ld]_{\operatorname{Norm}} \ar[rd]^{\operatorname{Trace}} &\\
\multiplcativeScheme & & \affineLine},$$
where $\operatorname{Norm}$ and $\operatorname{Trace}$ are the norm and trace maps, respectively. For a commutative $\finiteFieldExtension{a}$-algebra $R$, we have $$\restrictionOfScalars{\left(\finiteFieldExtension{\lambda} \otimes_{\finiteField} \finiteFieldExtension{a}\right)}{\finiteFieldExtension{a}}{\multiplcativeScheme}\left(R\right) = \multiplicativegroup{\left(\left(\finiteFieldExtension{\lambda} \otimes_{\finiteField} \finiteFieldExtension{a}\right) \otimes_{\finiteFieldExtension{a}} R\right)} = \multiplicativegroup{\left(\finiteFieldExtension{\lambda} \otimes_{\finiteField} R\right)},$$ and the maps are given by $$\operatorname{Norm}\left(R\right) = \TensorProductNormTwo{\left(\finiteFieldExtension{\lambda} \otimes_{\finiteField} \finiteFieldExtension{a}\right) \otimes_{\finiteFieldExtension{a}} R } = \TensorProductNormTwo{\finiteFieldExtension{\lambda} \otimes_{\finiteField} R} \colon \multiplicativegroup{\left(\finiteFieldExtension{\lambda} \otimes_{\finiteField} R\right)} \to \multiplicativegroup{R}$$ and $$\operatorname{Trace}\left(R\right) =\trace_{\slash \finiteFieldExtension{a}} \colon \left(\finiteFieldExtension{\lambda} \otimes_{\finiteField} \finiteFieldExtension{a}\right) \otimes_{\finiteFieldExtension{a}} R \to \finiteFieldExtension{a}.$$
The additive character $\fieldCharacter_a$ gives rise to an Artin--Schrier local system $\artinScrier{a}$ on $\affineLine$. The multiplicative character $\alpha_a = \alpha \circ \TensorProductNormOne{\finiteFieldExtension{\lambda} \otimes_{\finiteField} \finiteFieldExtension{a}} \colon \multiplicativegroup{\left(\finiteFieldExtension{\lambda} \otimes_{\finiteField} \finiteFieldExtension{a}\right)} \to \multiplicativegroup{\ladicnumbers}$ gives rise to a Kummer local system $\mathcal{L}_{\alpha_a}$ on $\restrictionOfScalars{\left(\finiteFieldExtension{\lambda} \otimes_{\finiteField} \finiteFieldExtension{a}\right)}{\finiteFieldExtension{a}}{\multiplcativeScheme}$. Consider the following $\ell$-adic complex on $\multiplcativeScheme$:
$$\ExoticKloosterman_{\finiteFieldExtension{a}}\left(\finiteFieldExtension{\lambda}, \alpha, \fieldCharacter\right) = \convolutionWithCompactSupport \operatorname{Norm}_{!} \left( \operatorname{Trace}^{\ast} \artinScrier{a} \otimes \mathcal{L}_{\alpha_a} \right)\left[k-1\right].$$ Denote $\mathcal{K} = \ExoticKloosterman_{\finiteFieldExtension{a}}\left(\finiteFieldExtension{\lambda}, \alpha, \fieldCharacter\right)$. The complex $\mathcal{K}$ admits the following properties.
\begin{enumerate}
	\item $\mathcal{K}$ is a local system of rank $k$ concentrated in degree $0$.
	\item For any $\xi \in \multiplicativegroup{\finiteFieldExtension{a}}$ and any $m \ge 0$, we have \begin{equation}\label{eq:trace-of-frobenius-is-kloosterman-sum}
		\trace\left(\Frobenius^m_{\xi} \mid \mathcal{K}_{\xi} \right) = \left(-1\right)^{k-1} \ExoticKloosterman_{m, \finiteFieldExtension{a}}\left(\alpha, \fieldCharacter, \xi\right),
	\end{equation}
	where $\mathcal{K}_{\xi}$ is the geometric stalk at $\xi$ and $\Frobenius_{\xi}$ is the action of the geometric Frobenius at $\xi$ on $\mathcal{K}_{\xi}$.
	\item $\mathcal{K}$ is pure of weight $k-1$. This means that for every eigenvalue $\lambda$ of $\Frobenius_{\xi} \mid \mathcal{K}_{\xi}$, we have that $\lambda$ and all of its algebraic conjugates have absolute value $q^{\frac{a\left(k-1\right)}{2}}$.
\end{enumerate}
It follows from \eqref{eq:trace-of-frobenius-is-kloosterman-sum} that the characteristic polynomial of $F_{\xi} = \Frobenius_{\xi} \mid \mathcal{K}_{\xi}$ is $$\mathrm{CharPoly}_{F_{\xi}}\left(T\right) = \prod_{j=1}^k \left(T - \omega_j\right),$$
and therefore we have $\abs{\omega_j} = q^{\frac{a \left(k-1\right)}{2}}$ (equivalently, $\abs{\omega_j^{\ast}} = 1$) for every $1 \le j \le k$.

\section{Exotic matrix exponential sums}\label{sec:exotic-matrix-exponential-sums}

The goal of this section is to define matrix versions of the exponential sums from the previous section. Our method for doing so is as follows. We first define non-abelian exotic Gauss sums $\TwistedGaussSum{\pi}{\alpha}{\fieldCharacter}$ associated with a character $\alpha \colon \multiplicativegroup{\finiteFieldExtension{\lambda}} \to \multiplicativegroup{\cComplex}$ and an irreducible representation $\pi$ of $\GL_c\left(\finiteField\right)$. We then define the counterpart exotic matrix Kloosterman sum $\ExoticKloosterman\left(\alpha, \fieldCharacter, h\right)$ as the unique class function satisfying $$\TwistedGaussSum{\pi}{\alpha}{\fieldCharacter} = q^{-\frac{\sizeof{\lambda} c^2}{2}} \sum_{h \in \GL_c\left(\finiteField\right)} \ExoticKloosterman\left(\alpha, \fieldCharacter, h\right) \pi\left(h\right)$$ for any irreducible representation $\pi$ of $\GL_c\left(\finiteField\right)$.

For example, in the basic case where $\lambda = \left(1,1,\dots,1\right) = \left(1^k\right)$ and $\alpha = \alpha_1 \times \dots \times \alpha_k$ where $\alpha_j \colon \multiplicativegroup{\finiteField} \to \multiplicativegroup{\cComplex}$, the corresponding non-abelian Gauss sum is $$\TwistedGaussSum{\pi}{\alpha}{\fieldCharacter} = q^{-\frac{k c^2}{2}} \sum_{h_1,\dots,h_k \in \GL_c\left(\finiteField\right)} \left(\prod_{j=1}^k \alpha_j\left(\det h_j\right) \fieldCharacter\left(\trace h_j\right)\right) \pi\left(h_1 \dots h_k\right),$$
and the corresponding matrix Kloosterman sum is
$$\ExoticKloosterman\left(\alpha, \fieldCharacter ,h\right) = \sum_{\substack{h_1,\dots,h_k \in \GL_c\left(\finiteField\right)\\
h_1 \dots h_k = h}} \prod_{j=1}^k \alpha_j\left(\det h_j\right) \fieldCharacter\left(\trace h_j\right),$$
which was studied in~\cite{Zelingher2023}.

Our method relies on the representation theory of $\GL_n\left(\finiteField\right)$, on Kondo's Gauss sum, and on Shintani's norm map. We begin with reviewing these. We then use them to define the notion of non-abelian exotic Gauss sums, which, as we show, can be reduced to products of exotic Gauss sums expressed in terms of characters corresponding to the cuspidal support of the given irreducible representation. We then use these to define exotic matrix Kloosterman sums using the method described above.

\subsection{An overview of the representation theory of $\GL_n\left(\finiteField\right)$}

In this section, we quickly review the representation theory of $\GL_n\left(\finiteField\right)$. We recall the notions of parabolic induction and of cuspidal representations, and of the cuspidal support of an irreducible representation. We then recall Green's parameterization of irreducible representations of $\GL_n\left(\finiteField\right)$.

We refer the reader to \cite[Section 1]{Macdonald80} and \cite[Chapter IV]{macdonald1998symmetric} for more details. See also \cite[Appendix B]{GurevichHowe2021} for an overview and other references. See also \cite{Carmon2023} for a nice self contained exposition.

Throughout the paper, all representations of $\GL_n\left(\finiteField\right)$ are always assumed to be finite dimensional. Note that this is automatic for irreducible representations.

\subsubsection{Parabolic induction}\label{subsec:parabolic-indudction}

Recall that a composition of a non-negative integer $n$ is a sequence $(n_1,\dots,n_r)$ of positive integers such that $n_1 + \dots + n_r = n$. We denote the standard parabolic subgroup of $\GL_n\left(\finiteField\right)$ corresponding to the composition $(n_1,\dots,n_r)$ of $n$ by $\ParabolicSubgroup_{(n_1,\dots,n_r)}$. It decomposes as the semi-direct product $\ParabolicSubgroup_{(n_1,\dots,n_r)} = D_{(n_1,\dots,n_r)} \ltimes \UnipotentRadical_{(n_1,\dots,n_r)}$, where $$D_{(n_1,\dots,n_r)} = \left\{ \diag\left(g_1,\dots,g_r\right) \mid g_j \in \GL_{n_j}\left(\finiteField\right) \right\}$$
is the \emph{Levi part} of $\ParabolicSubgroup_{(n_1,\dots,n_r)}$ and
$$\UnipotentRadical_{(n_1,\dots,n_r)} = \left\{ \begin{pmatrix}
	\IdentityMatrix{n_1} & \ast & \ast & \ast \\
	& \IdentityMatrix{n_2} & \ast & \ast\\
	& & \ddots & \ast \\
	& & & \IdentityMatrix{n_r} 
\end{pmatrix} \right\}$$
is the \emph{unipotent radical} of $\ParabolicSubgroup_{(n_1,\dots,n_r)}$.

Let $\sigma_1$, $\dots$, $\sigma_r$ be representations of $\GL_{n_1}\left(\finiteField\right)$, $\dots$, $\GL_{n_r}\left(\finiteField\right)$, respectively. We define their inflation $\sigma_1 \overline{\otimes} \dots \overline{\otimes} \sigma_r$, as follows. It is a representation of $\ParabolicSubgroup_{(n_1,\dots,n_r)}$ acting on the space of $\sigma_1 \otimes \dots \otimes \sigma_r$ by the action $$\left(\sigma_1 \overline{\otimes} \dots \overline{\otimes} \sigma_r\right)\left(d u\right) = \sigma_1\left(g_1\right) \otimes \dots \otimes \sigma_r\left(g_r\right),$$
where $d = \diag\left(g_1,\dots,g_r\right) \in D_{(n_1,\dots,n_r)}$ and $u \in \UnipotentRadical_{(n_1,\dots,n_r)}$.

Given $\sigma_1$, $\dots$, $\sigma_r$ as above, we define the \emph{parabolically induced representation} $\sigma_1 \circ \dots \circ \sigma_r$ of $\GL_n\left(\finiteField\right)$ by the formula $$\sigma_1 \circ \dots \circ \sigma_r = \Ind{\ParabolicSubgroup_{(n_1,\dots,n_r)}}{\GL_n\left(\finiteField\right)}{\sigma_1 \overline{\otimes} \dots \overline{\otimes} \sigma_r}.$$

The operation $\circ$ is associative. By this we mean that we have isomorphisms $$\sigma_1 \circ \sigma_2 \circ \sigma_3 \cong \sigma_1 \circ \left(\sigma_2 \circ \sigma_3\right) \cong \left(\sigma_1 \circ \sigma_2\right) \circ \sigma_3,$$
for any representations $\sigma_1, \sigma_2$ and $\sigma_3$ of $\GL_{n_1}\left(\finiteField\right)$, $\GL_{n_2}\left(\finiteField\right)$ and $\GL_{n_3}\left(\finiteField\right)$, respectively. The operation $\circ$ is also commutative. By this we mean that we have an isomorphism
$$\sigma_1 \circ \sigma_2 \cong \sigma_2 \circ \sigma_1,$$
for any representations $\sigma_1$ and $\sigma_2$ of $\GL_{n_1}\left(\finiteField\right)$ and $\GL_{n_2}\left(\finiteField\right)$, respectively.

A representation $\sigma$ of $\GL_n\left(\finiteField\right)$ is called \emph{cuspidal} if for any composition $(n_1,\dots,n_r)$ of $n$ with $r > 1$, $\sigma$ does not admit a non-zero vector that is invariant under the action of $\UnipotentRadical_{(n_1,\dots,n_r)}$. If $\sigma$ is irreducible, $\sigma$ is cuspidal if and only if $\sigma$ is not a subrepresentation of any parabolically induced representation $\sigma_1 \circ \dots \circ \sigma_r$ for any $r > 1$ and any choice of $\sigma_1$, $\dots$, $\sigma_r$.

By~\cite[Theorem 2.4]{Gelfand70}, if $\pi$ is an irreducible representation of $\GL_n\left(\finiteField\right)$, there exist a composition $(n_1,\dots,n_t)$ of $n$ and irreducible cuspidal representations $\pi_1$, $\dots$, $\pi_t$ of $\GL_{n_1}\left(\finiteField\right)$, $\dots$, $\GL_{n_t}\left(\finiteField\right)$, respectively, such that $\pi$ is a subrepresentation of the parabolically induced representation $\pi_1 \circ \dots \circ \pi_t$. Moreover, for any such $\pi$, the multiset consisting of (the equivalence classes of) the irreducible cuspidal representations $\pi_1$, $\dots$, $\pi_t$ is uniquely determined by (the equivalence class of) $\pi$. This multiset is called the \emph{cuspidal support} of $\pi$.

\subsubsection{Parameters}
The notion of the cuspidal support of an irreducible representation of $\GL_n\left(\finiteField\right)$ discussed above has a refinement using the notion of parameters. We discuss this refinement in this section.

Recall that we denote by $\Partitions$ the set consisting of all partitions of all non-negative integers and by $()$ the empty partition of $0$.

For any $m \ge 1$, let $\IrrCuspidal\left(\GL_m\left(\finiteField\right)\right)$ denote the set consisting of (equivalence of classes of) irreducible cuspidal representations of $\GL_m\left(\finiteField\right)$. Let $$\IrrCuspidalAll = \bigcup_{m=1}^{\infty} \IrrCuspidal\left(\GL_m\left(\finiteField\right)\right).$$

A \emph{parameter} is an assignment $\varphi \colon \IrrCuspidalAll \to \Partitions$, such that for all but finitely many $\sigma \in \IrrCuspidalAll$, $\varphi\left(\sigma\right) = \left(\right)$. The support of $\varphi$, denoted $\Supp \varphi$, is the set consisting of $\sigma$ such that $\varphi\left(\sigma\right) \ne \left(\right)$. We define $$\sizeof{\varphi} = \sum_{m = 1}^{\infty} \sum_{\sigma \in \IrrCuspidal\left(\GL_m\left(\finiteField\right)\right)} m \cdot \sizeof{\varphi\left(\sigma\right)}.$$

Green proved in~\cite{Green55} that there exists a bijection between parameters $\varphi$ with $\sizeof{\varphi} = n$ and irreducible representations of $\GL_n\left(\finiteField\right)$. Let us describe this bijection quickly.

First, if $\varphi$ is a parameter supported on a single $\sigma \in \IrrCuspidal\left(\GL_n\left(\finiteField\right)\right)$ with $\varphi\left(\sigma\right) = \left(1\right)$ then $\varphi$ corresponds to the irreducible cuspidal representation $\sigma$. Next, let $\sigma \in \IrrCuspidal\left(\GL_n\left(\finiteField\right)\right)$ and $m \ge 1$, and consider the following parabolically induced representation of $\GL_{nm}\left(\finiteField\right)$: $$\sigma^{\circ m} = \sigma \circ \dots \circ \sigma.$$
By~\cite[Appendix B]{GurevichHowe2021}, the different equivalence classes of irreducible subrepresentations of $\sigma^{\circ m}$ are in bijection with irreducible representations of the symmetric group $\SymmetricGroup_m$, which in turn are in bijection with partitions of $m$. Given a partition $\mu \vdash m$, let us denote by $\sigma^{\mu}$ the irreducible representation of $\sigma^{\circ m}$ corresponding to $\mu$.

Given a parameter $\varphi$ with $\sizeof{\varphi} = n$, let $\Supp \varphi = \left\{\sigma_1, \dots, \sigma_r\right\}$. We define $$\pi^{\varphi} = \sigma_1^{\varphi\left(\sigma_1\right)} \circ \dots \circ \sigma_r^{\varphi\left(\sigma_r\right)}.$$
Note that the cuspidal support of $\pi^{\varphi}$ consists of $\sizeof{\varphi\left(\sigma_1\right)}$ copies of $\sigma_1$, $\sizeof{\varphi\left(\sigma_2\right)}$ copies of $\sigma_2$, $\dots$, $\sizeof{\varphi\left(\sigma_r\right)}$ copies of $\sigma_r$.
Green proved that the map $\varphi \mapsto \pi^{\varphi}$ is a bijection between parameters $\varphi$ with $\sizeof{\varphi} = n$ and irreducible representations of $\GL_n\left(\finiteField\right)$. The contragredient representation $\left(\pi^{\varphi}\right)^{\vee}$ is the representation corresponding to the parameter $\varphi^{\vee}$ defined by the formula
$$\varphi^{\vee}\left(\sigma\right) = \varphi\left(\sigma^{\vee}\right)$$ for every $\sigma \in \IrrCuspidalAll$.

By~\cite[Section 6]{Gelfand70}, irreducible cuspidal representations of $\GL_n\left(\finiteField\right)$ are in bijection with Frobenius orbits of regular characters $\alpha \colon \multiplicativegroup{\finiteFieldExtension{n}} \to \multiplicativegroup{\cComplex}$. If $\pi$ is an irreducible cuspidal representation of $\GL_n\left(\finiteField\right)$ corresponding to the Frobenius orbit $\left[\alpha\right]$ of a regular character $\alpha \colon \multiplicativegroup{\finiteFieldExtension{n}} \to \multiplicativegroup{\cComplex}$, then the contragredient representation $\pi^{\vee}$ corresponds to the Frobenius orbit $\left[\alpha^{-1}\right]$.

\subsection{Kondo's Gauss sum}\label{subsec:kondo-gauss-sum}
In this section, we recall the definition of Kondo's non-abelian Gauss sum and its explicit computation~\cite{Kondo1963}.

\subsubsection{Class functions}\label{subsubsec:class-functions}
Recall that a (complex valued) class function $f$ of a finite group $G$ is a function $f \colon G \to \cComplex$ that is constant on conjugacy classes. Let $\ClassFunctionsRing\left(G\right)$ be the ring consisting of class functions of $G$. We equip $\ClassFunctionsRing\left(G\right)$ with the inner product $$\innerproduct{f_1}{f_2} = \frac{1}{\sizeof{G}} \sum_{g \in G} f_1\left(g\right) \conjugate{f_2\left(g\right)}.$$
The space $\ClassFunctionsRing\left(G\right)$ has two natural bases. The first basis consists of characteristic functions of conjugacy classes of $G$, i.e., it consists of $\left(\delta_C\right)_C$, where $C$ runs over all conjugacy class of $G$ and $\delta_{C}\left(g\right) = 1$ if $g \in C$ and $\delta_{C}\left(g\right) = 0$ otherwise. The second basis consists of characters of irreducible representations of $G$, i.e., it consists of $\left(\trace \pi\right)_{\pi}$, where $\pi$ runs over all equivalence classes of irreducible representations of $G$. Both bases are orthogonal with respect to the inner product $\innerproduct{\cdot}{\cdot}$. Moreover, the basis $\left(\trace \pi\right)_{\pi}$ is an orthonormal basis with respect to $\innerproduct{\cdot}{\cdot}$.

\subsubsection{Kondo's Gauss sum}
Let $\chi \colon \multiplicativegroup{\finiteField} \to \multiplicativegroup{\cComplex}$ be a character. The assignment $\GL_c\left(\finiteField\right) \to \cComplex$ given by $h \mapsto \chi\left(\det h\right)\fieldCharacter\left(\trace h\right)$ is a class function. Therefore, if $\pi$ is a representation of $\GL_n\left(\finiteField\right)$, the operator
$$\TwistedGaussSum{\pi}{\chi}{\fieldCharacter} = q^{-\frac{c^2}{2}} \sum_{h \in \GL_c\left(\finiteField\right)} \pi\left(h\right) \chi\left(\det h\right) \fieldCharacter\left(\trace h\right)$$
defines an element of $\Hom_{\GL_c\left(\finiteField\right)}\left(\pi, \pi\right)$. If $\pi$ is irreducible, then by Schur's lemma there exists a complex number $\GKGaussSumScalar{\pi}{\chi}{\fieldCharacter} \in \cComplex$ such that $$\TwistedGaussSum{\pi}{\chi}{\fieldCharacter} = \GKGaussSumScalar{\pi}{\chi}{\fieldCharacter} \cdot \idmap_{\pi}.$$

This is a twisted version of the non-abelian Gauss sum explicitly computed by Kondo~\cite{Kondo1963}. Let us recall Kondo's explicit computation.

\begin{theorem}\label{thm:kondo-godement-jacquet-computation}
	Let $\pi$ be an irreducible representation of $\GL_c\left(\finiteField\right)$. 
	Suppose that $\pi$ has cuspidal support $\left\{ \pi_1, \dots, \pi_t \right\}$, where for every $1 \le j \le t$, $\pi_j$ is an irreducible cuspidal representation of $\GL_{c_j}\left(\finiteField\right)$ corresponding to a Frobenius orbit of size $c_j$ of the form $\left[\beta_j\right]$, where $\beta_j \colon \multiplicativegroup{\finiteFieldExtension{c_j}} \to \multiplicativegroup{\cComplex}$ is a regular character. 
	Then $$\GKGaussSumScalar{\pi}{\chi}{\fieldCharacter} = \left(-1\right)^c q^{-\frac{c}{2}} \prod_{j=1}^t \tau\left(\beta_j \times \chi, \fieldCharacter_{c_j}\right),$$ where for every $j$, $\tau\left(\beta_j \times \chi, \fieldCharacter_{c_j}\right)$ is the twisted Gauss sum (see \eqref{eq:twisted-gauss-sum}).
\end{theorem}

Our next goal is to define an exotic version of Kondo's Gauss sum. We will do so using Shintani's norm map, which we recall in the next section.

\subsection{Shintani's correspondence}\label{sec:shintani-correspondence}

In this section, we explain Shintani's correspondence. We first review Shintani's norm map and use it to explain the notion of a Shintani lift of a representation. Then we recall results due to Silberger--Zink~\cite{SilbergerZink08} regarding the cuspidal support of a Shintani lift.

\subsubsection{Shintani's norm map}

Let $k \ge  1$. We say that two elements $x,y \in \GL_c\left(\finiteFieldExtension{k}\right)$ are \emph{$\Frobenius$-twisted conjugate} if there exists $h \in \GL_c\left(\finiteFieldExtension{k}\right)$, such that $$x = h y \Frobenius\left(h^{-1}\right),$$
where $\Frobenius \colon \GL_c\left(\finiteFieldExtension{k}\right) \to \GL_c\left(\finiteFieldExtension{k}\right)$ is the Frobenius automorphism, defined using the Frobenius automorphism $\Frobenius \in \Galois\left(\finiteFieldExtension{k} \slash \finiteField\right)$ ($x \mapsto x^q$) by applying it to every entry of a given matrix. For $x \in \GL_c\left(\finiteFieldExtension{k}\right)$, let us denote the $\Frobenius$-twisted conjugacy class of $x$ by $\conjugacyClass{x}_{\finiteFieldExtension{k} \slash \finiteField}$.

In~\cite{shintani1976two}, Shintani defined a map from $\Frobenius$-twisted conjugacy classes of $\GL_c\left(\finiteFieldExtension{k}\right)$ to conjugacy classes of $\GL_c\left(\finiteField\right)$. This map is called the \emph{norm map}. We move to explain its definition.

First, define for $h \in \GL_c\left(\finiteFieldExtension{k}\right)$ its norm by the formula $$\FieldNorm{k}{1}\left(h\right) =  \Frobenius^{k-1}h \cdot \Frobenius^{k-2}h \cdot \dots \cdot \Frobenius h  \cdot h.$$
If $c = 1$, then $\FieldNorm{k}{1}$ is the usual norm map $\FieldNorm{k}{1} \colon \multiplicativegroup{\finiteFieldExtension{k}} \to \multiplicativegroup{\finiteField}$. Generally speaking, for $c \ge 2$ and $h \in \GL_c\left(\finiteFieldExtension{k}\right)$, the element $\FieldNorm{k}{1}\left(h\right)$ does not necessarily lie in $\GL_c\left(\finiteField\right)$. However, the element $\FieldNorm{k}{1}\left(h\right)$ is conjugate (using elements in $\GL_c\left(\finiteFieldExtension{k}\right)$) to an element in $\GL_c\left(\finiteField\right)$, and thus defines a unique conjugacy class of $\GL_c\left(\finiteField\right)$ (recall that if two elements of $\GL_c\left(\finiteField\right)$ are conjugate in $\GL_c\left(\finiteFieldExtension{k}\right)$, then they are conjugate in $\GL_c\left(\finiteField\right)$, see~\cite[Corollary 18 on Page 477]{DummitFoote2004}). We denote this conjugacy class by $\conjugacyClass{\FieldNorm{k}{1}\left(h\right)}$. It is clear that if two elements in $\GL_c\left(\finiteFieldExtension{k}\right)$ lie in the same $\Frobenius$-twisted conjugacy class, then their norms define the same conjugacy class in $\GL_c\left(\finiteField\right)$. Note that for $h \in \GL_c\left(\finiteFieldExtension{k}\right)$, we have $\FieldNorm{k}{1}\left(\Frobenius h\right) = h \FieldNorm{k}{1}\left(h\right) h^{-1}$, and therefore $\conjugacyClass{\FieldNorm{k}{1}\left(\Frobenius h\right)} = \conjugacyClass{\FieldNorm{k}{1}\left(h\right)}$. Shintani showed that the map $$\conjugacyClass{h}_{\finiteFieldExtension{k} \slash \finiteField} \mapsto \conjugacyClass{\FieldNorm{k}{1}\left(h\right)}$$
is a bijection between $\Frobenius$-twisted conjugacy classes of $\GL_c\left(\finiteFieldExtension{k}\right)$ and conjugacy classes of $\GL_c\left(\finiteField\right)$.

\subsubsection{Shintani lift}
Let $\Pi$ be an irreducible representation of $\GL_c\left(\finiteFieldExtension{k}\right)$. Using the Frobenius automorphism of $\GL_c\left(\finiteFieldExtension{k}\right)$ discussed in the previous section, we can obtain another irreducible representation $\Pi^{\Frobenius}$ given by $\Pi^{\Frobenius}\left(h\right) = \Pi\left(\Frobenius h\right)$. If $\Pi^{\Frobenius}$ is isomorphic to $\Pi$, we say that $\Pi$ is \emph{$\Frobenius$-invariant}.

Suppose that $\Pi$ is $\Frobenius$-invariant and let $I_{\Pi} \colon \Pi \to \Pi^{\Frobenius}$ be an isomorphism. By Schur's lemma, such isomorphism is unique, up to scalar multiplication.

Shintani proved in~\cite{shintani1976two} the following theorem.

\begin{theorem}
	Let $\Pi$ be a $\Frobenius$-invariant representation of $\GL_c\left(\finiteFieldExtension{k}\right)$. Then there exists a unique isomorphism $I_{\Pi} \colon \Pi \to \Pi^{\Frobenius}$ and an irreducible representation $\pi$ of $\GL_c\left(\finiteField\right)$, such that for any $h \in \GL_c\left(\finiteFieldExtension{k}\right)$,
	\begin{equation}\label{eq:shintani-relation}
		\trace \pi\left(\conjugacyClass{\FieldNorm{k}{1}\left(h\right)}\right) = \trace \left( I_{\Pi} \circ \Pi\left(h\right) \right).
	\end{equation}
	
	Moreover, the assignment $\Pi \mapsto \pi$ is a bijection between (equivalence classes of) $\Frobenius$-invariant representations of $\GL_c\left(\finiteFieldExtension{k}\right)$ and (equivalence classes of) irreducible representations of $\GL_c\left(\finiteField\right)$.
\end{theorem}
For $\pi$ and $\Pi$ as in the theorem, we say that $\Pi$ is the \emph{Shintani lift of $\pi$ to $\GL_c\left(\finiteFieldExtension{k}\right)$}. Note that by choosing $h = \IdentityMatrix{c}$, we get that $\trace I_{\Pi} = \dim \pi$ for $I_{\Pi}$ as in the theorem.

\subsubsection{The cuspidal support of a Shintani lift}\label{subsec:cuspidal-suppoort-of-shintani}
Suppose that $\pi$ is an irreducible representation of $\GL_c\left(\finiteField\right)$, and let $\Pi$ be its Shintani lift to $\GL_c\left(\finiteFieldExtension{k}\right)$. In this section, we describe the cuspidal support of $\Pi$ in terms of the (Frobenius orbits associated to the) cuspidal support of $\pi$. This description is given in~\cite[Section 9]{SilbergerZink08}.

Let us first assume that $\pi$ is cuspidal. In this case, $\pi$ corresponds to a Frobenius orbit of a regular character $\beta \colon \multiplicativegroup{\finiteFieldExtension{c}} \to \multiplicativegroup{\cComplex}$. For any $i$, consider the multiplicative character $$\beta^{q^i} \circ \FieldNorm{\lcm\left(c,k\right)}{c} \colon \multiplicativegroup{\finiteFieldExtension{\lcm\left(c, k\right)}} \to \multiplicativegroup{\cComplex}.$$ This is a regular character of $\multiplicativegroup{\finiteFieldExtension{\lcm\left(c,k\right)}}$ of degree $$\frac{c}{\gcd\left(c,k\right)} = \frac{\lcm\left(c,k\right)}{k},$$ with respect to $\Frobenius^k$, the Frobenius automorphism corresponding to the base field $\finiteFieldExtension{k}$. Note that we have that the Frobenius orbit $\left[\beta\right] = \{ \beta, \beta^q, \dots, \beta^{q^{c-1}} \}$ is the disjoint union of the sets $$\left\{ \left(\beta^{q^i}\right)^{q^{jk}} \mid 0 \le j < \frac{\lcm\left(c,k\right)}{k} \right\}_{i=0}^{\gcd\left(c,k\right) - 1}.$$ Since the $\Frobenius^k$-orbit of $\beta^{q^i} \circ \FieldNorm{\lcm\left(c,k\right)}{c}$ is of size $\frac{\lcm\left(c,k\right)}{k}$, this orbit defines an irreducible cuspidal representation of $\GL_{\frac{\lcm\left(c,k\right)}{k}}\left(\finiteFieldExtension{k}\right)$. Let us denote this representation by $\Pi\left(\beta, k, i\right)$. Then the cuspidal support of $\Pi$ is the set $\{ \Pi\left(\beta,k,i\right) \mid 0 \le i \le \gcd\left(c,k\right)-1 \}$.

Suppose now that $\pi$ is an irreducible representation of $\GL_c\left(\finiteField\right)$ with cuspidal support $\left\{\pi_1,\dots,\pi_t\right\}$, where $\pi_j$ is an irreducible cuspidal representation of $\GL_{c_j}\left(\finiteField\right)$ and $c_1 + \dots + c_t = c$. For every $1 \le j \le t$, let $\Pi_j$ be the Shintani lift of $\pi_j$ to $\GL_{c_j}\left(\finiteFieldExtension{k}\right)$. Then the cuspidal support of $\Pi$ is the multiset given by summing the cuspidal supports of $\Pi_1$, $\dots$, $\Pi_t$.

\subsection{Non-abelian exotic Gauss sums}\label{subsec:non-abelian-exotic-gauss-sums}
In this section, we define exotic versions of Kondo's Gauss sum. As in \Cref{subsec:exotic-gauss-sums}, the definition involves two steps. Let $\pi$ be an irreducible representation of $\GL_c\left(\finiteField\right)$. We first define an exotic Gauss sum associated to the pair $\pi$ and $\chi$, where $\chi \colon \multiplicativegroup{\finiteFieldExtension{k}} \to \multiplicativegroup{\cComplex}$ is a character. We then use this definition to define an exotic Gauss sum associated to the pair $\pi$ and $\chi$, where $\chi \colon \multiplicativegroup{\finiteFieldExtension{\lambda}} \to \multiplicativegroup{\cComplex}$ is a character, where $\finiteFieldExtension{\lambda}$ is an \etale algebra over $\finiteField$ associated with the partition $\lambda$. We use the results from the previous sections to reduce these exotic Gauss sums to products of the exotic Gauss sums from \Cref{subsec:exotic-gauss-sums} (\Cref{thm:hasse-davenport-for-hall-littlewood-equality}). We also establish a Hasse--Davenport type lifting relation for these sums (\Cref{thm:hasse-davenport-kondo-gauss-sums}).

\subsubsection{Non-abelian exotic Gauss sums for $\finiteFieldExtension{k}$}\label{subsec:non-abelian-exotic-gauss-sums-for-f-k}

For an irreducible representation $\pi$ of $\GL_c\left(\finiteField\right)$, and a multiplicative character $\chi \colon \multiplicativegroup{\finiteFieldExtension{k}} \to \multiplicativegroup{\cComplex}$, denote
\begin{equation}\label{eq:exotic-gauss-sum-definition}
	\TwistedGaussSum{\pi}{\chi}{\fieldCharacter} = q^{-\frac{k c^2}{2}} \sum_{x \in \GL_c\left(\finiteFieldExtension{k}\right)} \sum_{h \in \conjugacyClass{\FieldNorm{k}{1}\left(x\right)}} \frac{1}{ \#{\conjugacyClass{\FieldNorm{k}{1}\left(x\right)}}} \chi\left(\det x\right) \fieldCharacter_{k}\left(\trace x\right) \pi\left(h\right).
\end{equation}
One can easily show that for $k=1$ this coincides with the definition from \Cref{subsec:kondo-gauss-sum}. It will also follow from the results below. Since $\conjugacyClass{\FieldNorm{k}{1}\left(h\right)} = \conjugacyClass{\FieldNorm{k}{1}\left(\Frobenius h\right)}$, we have that $\TwistedGaussSum{\pi}{\chi}{\fieldCharacter}$ is invariant under replacing $\chi$ with any element from its Frobenius orbit.

For $h \in \GL_c\left(\finiteField\right)$, let us denote $$\ExoticKloosterman\left(\chi, \fieldCharacter, h\right) = \sum_{x} \frac{1}{\#{\conjugacyClass{\FieldNorm{k}{1}\left(x\right)}}} \chi\left(\det x\right) \fieldCharacter_k\left(\trace x\right),$$
where the sum is all over $x \in \GL_c\left(\finiteFieldExtension{k}\right)$, such that $h$ belongs to the conjugacy class $\conjugacyClass{\FieldNorm{k}{1}\left(x\right)}$. We will discuss this sum in more detail in the next section. We have the equality
\begin{equation}\label{eq:exotic-gauss-sum-as-kloosterman-sum}
	\TwistedGaussSum{\pi}{\chi}{\fieldCharacter} = q^{-\frac{k c^2}{2}} \sum_{h \in \GL_c\left(\finiteField\right)} \ExoticKloosterman\left(\chi, \fieldCharacter, h\right) \pi\left(h\right).
\end{equation}
By its definition, the map $h \mapsto \ExoticKloosterman\left(\chi, \fieldCharacter, h\right)$ is a class function of $\GL_c\left(\finiteField\right)$, and therefore $\TwistedGaussSum{\pi}{\chi}{\fieldCharacter} \in \Hom_{\GL_c\left(\finiteField\right)}\left(\pi, \pi\right)$. Since $\pi$ is irreducible, by Schur's lemma that there exists a constant $\GKGaussSumScalar{\pi}{\chi}{\fieldCharacter} \in \cComplex$, such that \begin{equation*}
	\TwistedGaussSum{\pi}{\chi}{\fieldCharacter} = \GKGaussSumScalar{\pi}{\chi}{\fieldCharacter} \cdot \idmap_\pi.
\end{equation*}

The following theorem expresses $\GKGaussSumScalar{\pi}{\chi}{\fieldCharacter}$ as a twisted Kondo Gauss sum corresponding to the Shintani lift of $\pi$.
\begin{theorem}\label{prop:exotic-gauss-sum-is-gauss-sum-of-shintani-lift}
	Let $\Pi$ be the Shintani lift of $\pi$ to $\GL_c\left(\finiteFieldExtension{k}\right)$. Then $$\GKGaussSumScalar{\pi}{\chi}{\fieldCharacter} = \GKGaussSumScalar{\Pi}{\chi}{\fieldCharacter_k}.$$
\end{theorem}
\begin{proof}
	By taking the trace of \eqref{eq:exotic-gauss-sum-definition} we get
	$$ \dim \pi \cdot  \GKGaussSumScalar{\pi}{\chi}{\fieldCharacter} = q^{-\frac{k c^2}{2}} \sum_{x \in \GL_c\left(\finiteFieldExtension{k}\right)} \trace \pi\left(\conjugacyClass{\FieldNorm{k}{1}\left(x\right)}\right) \chi\left(\det x\right) \fieldCharacter_{k}\left(\trace x\right).$$
	By the Shintani relation \eqref{eq:shintani-relation}, there exists a unique $I_{\Pi} \in \Hom_{\GL_c\left(\finiteFieldExtension{k}\right)}\left(\Pi, \Pi^{\Frobenius}\right)$, such that $$\GKGaussSumScalar{\Pi}{\chi}{\fieldCharacter} = \frac{q^{-\frac{k c^2}{2}}}{\dim \pi}\sum_{x \in \GL_c\left(\finiteFieldExtension{k}\right)} \trace\left(I_{\Pi} \circ \Pi\left(x\right)\right) \chi\left(\det x\right) \fieldCharacter_{k}\left(\trace x\right).$$
	This can be rewritten as
	$$\GKGaussSumScalar{\pi}{\chi}{\fieldCharacter} = \frac{1}{\dim \pi} \trace\left(I_{\Pi} \circ \TwistedGaussSum{\Pi}{\chi}{\fieldCharacter_{k}} \right) = \GKGaussSumScalar{\Pi}{\chi}{\fieldCharacter_{k}} \frac{\trace I_{\Pi}}{\dim \pi}.$$
	Since $\dim \pi = \trace I_{\Pi}$, the result follows.
\end{proof}

By \Cref{thm:kondo-godement-jacquet-computation} and the discussion in \Cref{subsec:cuspidal-suppoort-of-shintani}, we may express $\GKGaussSumScalar{\pi}{\chi}{\fieldCharacter}$ explicitly as follows.

\begin{theorem}\label{thm:hasse-davenport-for-hall-littlewood-equality}
	Let $\pi$ be an irreducible representation of $\GL_c\left(\finiteField\right)$ with cuspidal support $\left\{\pi_1, \dots, \pi_t\right\}$, where for every $j$, $\pi_j$ is an irreducible cuspidal representation of $\GL_{c_j}\left(\finiteField\right)$ corresponding to the Frobenius orbit of a regular character $\beta_j \colon \multiplicativegroup{\finiteFieldExtension{c_j}} \to \multiplicativegroup{\cComplex}$. Then $$ \GKGaussSumScalar{\pi}{\chi}{\fieldCharacter} = q^{-\frac{kc}{2}} \left(-1\right)^{c} \prod_{j = 1}^t \GaussSumCharacter{c_j,k}{\beta_j}{\chi}{\fieldCharacter}.$$
\end{theorem}
\begin{proof}
	Since the cuspidal support of the Shintani lift of $\pi$ is the sum of the cuspidal supports of the Shintani lifts of the elements in the cuspidal support of $\pi$, it follows from Theorems \ref{thm:kondo-godement-jacquet-computation} and \ref{prop:exotic-gauss-sum-is-gauss-sum-of-shintani-lift} that it suffices to prove this theorem for the case where $\pi$ is an irreducible cuspidal representation.
	
	Let $\pi$ be an irreducible cuspidal representation of $\GL_c\left(\finiteField\right)$ corresponding to the Frobenius orbit of a regular character $\beta \colon \multiplicativegroup{\finiteFieldExtension{c}} \to \multiplicativegroup{\cComplex}$. Then by \Cref{thm:kondo-godement-jacquet-computation} and the discussion in \Cref{subsec:cuspidal-suppoort-of-shintani}, we have that
	$$ \GKGaussSumScalar{\Pi}{\chi}{\fieldCharacter_k} =  q^{-\frac{kc}{2}} \left(-1\right)^{c} \prod_{j=0}^{\gcd\left(c,k\right)-1} \tau\left( \beta^{q^j} \circ \FieldNorm{\lcm\left(c,k\right)}{c} \cdot \chi \circ \FieldNorm{\lcm\left(c,k\right)}{k}, \fieldCharacter_{\lcm\left(c,k\right)} \right),$$
	which by \eqref{eq:exotic-gauss-sum-is-a-product-of-exotic-gauss-sums} equals $$\GKGaussSumScalar{\Pi}{\chi}{\fieldCharacter_k} =  q^{-\frac{kc}{2}} \left(-1\right)^{c} \GaussSumCharacter{c,k}{\beta}{\chi}{\fieldCharacter}.$$
\end{proof}

\subsubsection{Non-abelian exotic Gauss sums for $\finiteFieldExtension{\lambda}$}
Let $\lambda = \left(k_1,\dots,k_s\right) \vdash k$ and $\alpha \colon \multiplicativegroup{\finiteFieldExtension{\lambda}} \to \multiplicativegroup{\cComplex}$ be a character of the form $\alpha = \alpha_1 \times \dots \times \alpha_s$ as in \Cref{subsec:exotic-gauss-sums-for-composite}.

If $\pi$ is an irreducible representation of $\GL_c\left(\finiteField\right)$, then we define $\TwistedGaussSum{\pi}{\alpha}{\fieldCharacter}$ to be the following composition of scalar operators
$$\TwistedGaussSum{\pi}{\alpha}{\fieldCharacter} = \TwistedGaussSum{\pi}{\alpha_1}{\fieldCharacter} \circ \dots \circ \TwistedGaussSum{\pi}{\alpha_s}{\fieldCharacter}.$$
It is clear that $\TwistedGaussSum{\pi}{\alpha}{\fieldCharacter}$ is invariant under simultaneous permutation of $\alpha_j$ and $k_j$. Moreover, as above, $\TwistedGaussSum{\pi}{\alpha}{\fieldCharacter}$ is invariant under replacing any $\alpha_j$ with any element from its Frobenius orbit.

Let us denote $$\TwistedGaussSum{\pi}{\alpha}{\fieldCharacter} = \GKGaussSumScalar{\pi}{\alpha}{\fieldCharacter} \cdot \idmap_{\pi}.$$
We have the identity \begin{equation*}
	\GKGaussSumScalar{\pi}{\alpha}{\fieldCharacter} = \prod_{j=1}^s \GKGaussSumScalar{\pi}{\alpha_j}{\fieldCharacter},
\end{equation*}
and the following straightforward generalization of \Cref{thm:hasse-davenport-for-hall-littlewood-equality}: \begin{theorem}\label{thm:exotic-gauss-sum-of-composite-character}
	Let be $\pi$ an irreducible representation of $\GL_c\left(\finiteField\right)$ with cuspidal support as in \Cref{thm:hasse-davenport-for-hall-littlewood-equality}. Then
	$$\GKGaussSumScalar{\pi}{\alpha}{\fieldCharacter} = q^{-\frac{kc}{2}} \left(-1\right)^{cs} \cdot \prod_{i=1}^s \prod_{j=1}^t \GaussSumCharacter{c_j, k_i}{\beta_j}{\alpha_i}{\fieldCharacter}.$$
\end{theorem}

The following theorem shows that we may always assume that $\alpha_1$, $\dots$, $\alpha_s$ are regular characters.
\begin{theorem}[Hasse--Davenport type lifting relation]\label{thm:hasse-davenport-kondo-gauss-sums}
	Let $\pi$ be an irreducible representation of $\GL_c\left(\finiteField\right)$ and let $k = m k'$. Suppose that $\chi' \colon \multiplicativegroup{\finiteFieldExtension{k'}} \to \multiplicativegroup{\cComplex}$ is a character such that $\chi = \chi' \circ \FieldNorm{k}{k'}$. Then
	$$ \GKGaussSumScalar{\pi}{\chi}{\fieldCharacter} = \left(-1\right)^{c \left(m-1\right)} \GKGaussSumScalar{\pi}{\chi'}{\fieldCharacter}^{m}.$$
\end{theorem}
\begin{proof}
	Suppose that the cuspidal support of $\pi$ is as in the statement of \Cref{thm:hasse-davenport-for-hall-littlewood-equality}. We have that \begin{equation*}
		\GKGaussSumScalar{\pi}{\chi'}{\fieldCharacter} = q^{-\frac{k' c}{2}} \left(-1\right)^{c} \cdot \prod_{j = 1}^t \GaussSumCharacter{c_j,k'}{\beta_j}{\chi'}{\fieldCharacter}.
	\end{equation*}
	Using \Cref{thm:hasse-davenport-for-exotic-gauss-sums}, we get \begin{equation*}
		\GKGaussSumScalar{\pi}{\chi'}{\fieldCharacter}^m = q^{-\frac{mk'c}{2}} \left(-1\right)^{mc} \cdot \prod_{j = 1}^t \GaussSumCharacter{c_j,mk'}{\beta_j}{\chi' \circ \FieldNorm{mk'}{k'}}{\fieldCharacter},
	\end{equation*}
	which yields the desired equality by \Cref{thm:hasse-davenport-for-hall-littlewood-equality}, since $k = mk'$ and $\chi = \chi' \circ \FieldNorm{k}{k'}$.
\end{proof}
\begin{remark}
	\Cref{thm:hasse-davenport-kondo-gauss-sums} can be seen as a generalization of~\cite[Corollary 5.2]{curtis2004zeta}. To see this, take (in our notation) $k'=1$ (and then $k=m$), $\chi = 1 \colon \multiplicativegroup{\finiteFieldExtension{m}} \to \multiplicativegroup{\cComplex}$ and use \Cref{prop:exotic-gauss-sum-is-gauss-sum-of-shintani-lift} to relate $\GKGaussSumScalar{\Pi}{1}{\fieldCharacter_m}$ to $\left(-1\right)^{m\left(c-1\right)} \GKGaussSumScalar{\pi}{1}{\fieldCharacter}^m$.
\end{remark}

\subsection{Exotic matrix Kloosterman sums}\label{subsec:exotic-matrix-kloosterman-sum}

Using the non-abelian exotic Gauss sums defined above, we define the notion of what we call exotic matrix Kloosterman sums. If $\lambda = \left(k_1,\dots,k_s\right) \vdash k$ and $\alpha \colon \multiplicativegroup{\finiteFieldExtension{\lambda}} \to \multiplicativegroup{\cComplex}$ is given by $\alpha = \alpha_1 \times \dots \times \alpha_s$ as in \Cref{subsec:exotic-gauss-sums-for-composite}, we define the exotic matrix Kloosterman sum $\ExoticKloosterman\left(\alpha, \fieldCharacter, \cdot\right)$ to be the unique class function $\GL_c\left(\finiteField\right) \to \cComplex$, such that for any irreducible representation $\pi$ of $\GL_c\left(\finiteField\right)$,
$$\TwistedGaussSum{\pi}{\alpha}{\fieldCharacter} = q^{-\frac{kc^2}{2}} \sum_{h \in \GL_c \left(\finiteField\right)} \ExoticKloosterman\left(\alpha, \fieldCharacter, h\right) \pi\left(h\right).$$
The existence and uniqueness of this class function follows from the fact that $\left(\trace \pi\right)_{\pi}$ where $\pi$ goes over all (equivalence classes of) irreducible representations of $\GL_c\left(\finiteField\right)$ forms a basis of the space of class functions of $\GL_c\left(\finiteField\right)$ (see \Cref{subsubsec:class-functions}).
We also define a normalized version by
$$\ExoticKloostermanNormalized\left(\chi, \fieldCharacter, h\right) = q^{-\frac{\left(k-1\right)c^2}{2}} \ExoticKloosterman\left(\chi, \fieldCharacter, h\right).$$
In the next two subsections, we describe exotic matrix Kloosterman sums more explicitly.

\subsubsection{Exotic matrix Kloosterman sums for $\finiteFieldExtension{k}$}\label{sec:exotic-kloosterman-sums-cuspidal-sum}

By \eqref{eq:exotic-gauss-sum-as-kloosterman-sum}, for $h \in \GL_c\left(\finiteField\right)$, and a multiplicative character $\chi \colon \multiplicativegroup{\finiteFieldExtension{k}} \to \multiplicativegroup{\cComplex}$, the exotic matrix Kloosterman sum $\ExoticKloosterman\left(\chi, \fieldCharacter, h\right)$ is given as in \Cref{subsec:non-abelian-exotic-gauss-sums-for-f-k}, by the formula $$\ExoticKloosterman\left(\chi, \fieldCharacter, h\right) = \sum_{x} \frac{1}{\#{\conjugacyClass{\FieldNorm{k}{1}\left(x\right)}}} \chi\left(\det x\right) \fieldCharacter_k\left(\trace x\right),$$
where the sum is all over $x \in \GL_c\left(\finiteFieldExtension{k}\right)$ such that $h$ belongs to the conjugacy class $\conjugacyClass{\FieldNorm{k}{1}\left(x\right)}$.

\begin{remark}\label{rem:kloosterman-sum-with-one-variable}
	If $k=1$, we have that $\finiteFieldExtension{k} = \finiteField$ and $\FieldNorm{k}{1}\left(x\right) = x$. Therefore, in this case $x$ is summed over the conjugacy class of $h$. Since the determinant and the trace are constant on conjugacy classes of $\GL_c\left(\finiteField\right)$, we get in this case $\ExoticKloosterman\left(\chi, \fieldCharacter, h\right) = \chi\left(\det h\right) \fieldCharacter\left(\trace h\right)$.
\end{remark}

\begin{remark}
	If $c = 1$, then the conjugacy class $\conjugacyClass{\FieldNorm{k}{1}\left(x\right)}$ consists of the single element $\FieldNorm{k}{1}\left(x\right) \in \multiplicativegroup{\finiteField}$. Hence the sum in this case is $$\ExoticKloosterman\left(\chi,\fieldCharacter,h\right) = \sum_{\substack{x \in \multiplicativegroup{\finiteFieldExtension{k}}\\
	h = \FieldNorm{k}{1}\left(x\right)}} \chi\left(x\right) \fieldCharacter_k\left(x\right),$$
	which coincides with the definition in \Cref{subsec:exotic-kloosterman-sums}.
\end{remark}

\subsubsection{Exotic matrix Kloosterman sums for $\finiteFieldExtension{\lambda}$}\label{subsec:exotic-matrix-kloosterman-for-etale-algebra}
If $\lambda = \left(k_1,\dots,k_s\right)$ and $\alpha = \alpha_1 \times \dots \times \alpha_s$ are as in \Cref{subsec:exotic-gauss-sums-for-composite}, the exotic matrix Kloosterman sum $\ExoticKloosterman\left(\alpha,\fieldCharacter,h\right)$ is given by the convolution 
$$\ExoticKloosterman\left(\alpha,\fieldCharacter,h\right) = \sum_{\substack{h_1, \dots, h_s \in \GL_c\left(\finiteField\right)\\
		h_1 \cdot \dots \cdot h_s = h}} \prod_{i=1}^s \ExoticKloosterman\left(\alpha_i,\fieldCharacter,h_i\right).$$
\begin{remark}\label{rem:twisted-kloosterman-sum-case}
	If $\lambda = \left(1,1,\dots,1\right) \vdash k$, then $\alpha_1, \dots, \alpha_k \colon \multiplicativegroup{\finiteField} \to \multiplicativegroup{\cComplex}$ are characters. Hence, by \Cref{rem:kloosterman-sum-with-one-variable}, we have $\ExoticKloosterman\left(\alpha_i, \fieldCharacter, h_i\right) = \alpha_i\left(\det h_i\right) \fieldCharacter\left(\trace h_i\right)$. It thus follows in this case that $$\ExoticKloosterman\left(\alpha, \fieldCharacter, h\right) = \sum_{\substack{h_1,\dots,h_k \in \GL_c\left(\finiteField\right)\\
	h_1 \cdot \dots \cdot h_k = h}} \alpha_1\left(\det h_1\right) \dots \alpha_k\left(\det h_k\right) \fieldCharacter\left(\trace \left(h_1 + \dots + h_k\right)\right)$$ which is the twisted matrix Kloosterman sum studied in~\cite{Zelingher2023}.
\end{remark}

The next theorem allows us to reduce any exotic matrix Kloosterman sum to an exotic matrix Kloosterman sum associated to a tuple of regular characters.

\begin{theorem}\label{thm:reduction-of-matrix-kloosterman-sums-to-regular-case}
	Let $k = m k'$ and let $\chi' \colon \multiplicativegroup{\finiteFieldExtension{k'}} \to \multiplicativegroup{\cComplex}$ be a character. Suppose that $\chi = \chi' \circ \FieldNorm{k}{k'}$. Then for any $h \in \GL_c\left(\finiteField\right)$, \begin{equation}\label{eq:hasse-davenport-kloosterman-sum-identity}
		\ExoticKloosterman\left(\chi, \fieldCharacter, h\right) = \left(-1\right)^{c\left(m-1\right)} \ExoticKloosterman\left(\left(\chi'\right)^{\times m}, \fieldCharacter, h\right),
	\end{equation}
	where $\left(\chi'\right)^{\times m} \colon \left(\multiplicativegroup{\finiteField}\right)^m \to \multiplicativegroup{\cComplex}$ is the character $\chi' \times \dots \times \chi'$, where $\chi'$ is taken $m$ times.
\end{theorem}
\begin{proof}
	By \Cref{thm:hasse-davenport-kondo-gauss-sums}, we have that $$\TwistedGaussSum{\pi}{\chi}{\fieldCharacter} = \left(-1\right)^{c\left(m-1\right)} \TwistedGaussSum{\pi}{\chi'}{\fieldCharacter}^m = \left(-1\right)^{c\left(m-1\right)} \TwistedGaussSum{\pi}{\left(\chi'\right)^{\times m}}{\fieldCharacter}.$$
	By the definition of exotic matrix Kloosterman sums, the equality \eqref{eq:hasse-davenport-kloosterman-sum-identity} must hold for every $h \in \GL_c\left(\finiteField\right)$.
\end{proof}

It follows from \Cref{thm:reduction-of-matrix-kloosterman-sums-to-regular-case} that by replacing every non-regular character $\alpha_j$ with an appropriate number of copies of a regular character (corresponding to a subfield) associated to it, we can always assume that the characters $\alpha_1, \dots, \alpha_s$ are all regular. 

\section{Relation to Bessel--Speh functions}\label{sec:relation-to-bessel-speh-functions}

In this section, we explain a relation between exotic matrix Kloosterman sums and special values of Bessel functions attached to Speh representations associated with irreducible generic representations of $\GL_k\left(\finiteField\right)$. Our method of proof utilizes our results with Carmon about Ginzburg--Kaplan gamma factors~\cite{CarmonZelingher2025}.

\subsection{$\varepsilon_0$-factors}\label{subsec:epsilon-factors}
We recall the notion of tensor product $\varepsilon_0$-factors corresponding to a pair of irreducible representation $\pi$ and $\tau$ of $\GL_c\left(\finiteField\right)$ and $\GL_k\left(\finiteField\right)$, respectively~\cite[Theorem 4.2]{ye2021epsilon}.

Suppose first that $\pi$ and $\tau$ are irreducible cuspidal representations corresponding to the Frobenius orbits of regular characters $\beta \colon \multiplicativegroup{\finiteFieldExtension{c}} \to \multiplicativegroup{\cComplex}$ and $\alpha \colon \multiplicativegroup{\finiteFieldExtension{k}} \to \multiplicativegroup{\cComplex}$, respectively. In this case we have $$\varepsilon_0\left( \pi \times \tau, \fieldCharacter \right) = \left(-1\right)^{kc} q^{-\frac{kc}{2}} \GaussSumCharacter{c,k}{\beta^{-1}}{\alpha^{-1}}{\fieldCharacter}.$$

Next, given arbitrary irreducible representations $\pi$ and $\tau$ with cuspidal supports $\left\{\pi_1,\dots,\pi_t\right\}$ and $\left\{\tau_1,\dots,\tau_s\right\}$, respectively, we have
$$\varepsilon_0\left(\pi \times \tau, \fieldCharacter\right) = \prod_{i=1}^t \prod_{j=1}^s \varepsilon_0\left(\pi_i \times \tau_j, \fieldCharacter\right).$$
We will need the following proposition, which is not difficult to show using the formulas in this section and using \Cref{thm:exotic-gauss-sum-of-composite-character}.
\begin{proposition}\label{prop:non-abelian-exotic-gauss-sum-equals-epsilon-factor}
	Keep the notations as above. Suppose that for every $j$, $\tau_j$ corresponds to the Frobenius orbit of a regular character $\alpha_j \colon \multiplicativegroup{\finiteFieldExtension{k_j}} \to \multiplicativegroup{\cComplex}$. Denote $\alpha = \alpha_1 \times \dots \times \alpha_s$. Then $$\GKGaussSumScalar{\pi}{\alpha}{\fieldCharacter} = \left(-1\right)^{\left(k+s\right)c} \varepsilon_0\left(\pi^{\vee} \times \tau^{\vee}, \fieldCharacter\right).$$
\end{proposition}

\subsection{$\kcNotation{k}{c}{\fieldCharacter}$ vectors and Speh representations}

In this section, we recall the notion of $\kcNotation{k}{c}{\fieldCharacter}$ vectors, which generalizes the notion of $\fieldCharacter$-Whittaker vectors. We then discuss the notion of Speh representations attached to irreducible generic representations and explain the uniqueness property of their $\kcNotation{k}{c}{\fieldCharacter}$ vectors.

\subsubsection{$\kcNotation{k}{c}{\fieldCharacter}$ vectors}

Let $k$ and $c$ be positive integers. Define a character $\fieldCharacterkc{k}{c} \colon \UnipotentRadicalForWss{k}{c} \to \multiplicativegroup{\cComplex}$ by the formula
$$\fieldCharacterkc{k}{c} \begin{pmatrix}
	\IdentityMatrix{c} & X_1 & \ast & \ast  & \ast \\
	& \IdentityMatrix{c} & X_2 & \ast & \ast \\
	& & \ddots & \ddots & \ast  \\
	& & & \IdentityMatrix{c} &  X_{k-1} \\
	& & & & \IdentityMatrix{c}
\end{pmatrix} = \fieldCharacter\left( \sum_{j=1}^{k-1} \trace X_j \right).$$
Let $\pi$ be a representation of $\GL_{kc}\left(\finiteField\right)$. A vector $v \in \pi$ is called a \emph{$\kcNotation{k}{c}{\fieldCharacter}$ vector} if $0 \ne v$ and if for any $u \in \UnipotentRadicalForWss{k}{c}$ we have $\pi\left(u\right)v = \fieldCharacterkc{k}{c}\left(u\right) v$.

\subsubsection{Generic representations}\label{subsec:generic-representations}
When $c=1$, we denote $\UnipotentSubgroup_k = \UnipotentRadicalForWss{k}{1}$ and $\fieldCharacter = \fieldCharacterkc{k}{1} \colon \UnipotentSubgroup_k \to \multiplicativegroup{\cComplex}$. We also use the term \emph{$\fieldCharacter$-Whittaker vector} instead of $\kcNotation{k}{1}{\fieldCharacter}$ vector in this case. A representation $\pi$ of $\GL_k\left(\finiteField\right)$ is called \emph{generic} if it admits a $\fieldCharacter$-Whittaker vector. If $\pi$ is irreducible and generic, then by a well-known result of Gelfand--Graev~\cite[Corollary 5.6]{SilbergerZink00}, the $\fieldCharacter$-Whittaker vector of $\pi$ is unique (up to scalar multiplication). 

Another well-known result is that irreducible cuspidal representations of $\GL_k\left(\finiteField\right)$ are generic~\cite[Lemma 5.2]{SilbergerZink00}. More generally, for any choice of irreducible cuspidal representations $\tau_1$, $\dots$, $\tau_r$ of $\GL_{k_1}\left(\finiteField\right)$, $\dots$, $\GL_{k_r}\left(\finiteField\right)$, respectively, such that $k_1 + \dots + k_r = k$, there exists a unique irreducible generic representation $\tau$ of $\GL_k\left(\finiteField\right)$ with cuspidal support $\left\{ \tau_1, \dots, \tau_r \right\}$. The parameter $\varphi_{\tau}$ corresponding to this generic representation can be described as follows. Let $\tau'_1$, $\dots$, $\tau'_{r'}$ be all the different (equivalence classes of) representations among $\tau_1$, $\dots$, $\tau_r$ and for $1 \le j \le r'$, let $m_j$ be the number of times $\tau'_j$ appears among $\tau_1,\dots,\tau_r$. Then $\varphi_{\tau}$ is only supported on $\tau'_1$, $\dots$, $\tau'_{r'}$ and $\varphi_{\tau}\left(\tau'_j\right) = \left(m_j\right)$ for any $1 \le j \le r'$~\cite[Theorem 5.5]{SilbergerZink00}.

\subsubsection{Speh representations}\label{subsec:speh-representations}

Let $\tau$ be an irreducible generic representation of $\GL_k\left(\finiteField\right)$. For any $c \ge 1$, there exists an irreducible representation $\SpehRepresentation{\tau}{c}$ of $\GL_{kc}\left(\finiteField\right)$ which is called a \emph{Speh representation}. These representations were studied by Oded Carmon in his master's thesis~\cite{Carmon2023}. 

The parameter of the Speh representation $\SpehRepresentation{\tau}{c}$ can be described as follows. If $\tau$ is the unique irreducible generic representation with cuspidal support $\left\{\tau_1, \dots, \tau_r\right\}$ as in the end of \Cref{subsec:generic-representations}, then $\SpehRepresentation{\tau}{c}$ is the irreducible representation of $\GL_{kc}\left(\finiteField\right)$ with parameter $\varphi_{\SpehRepresentation{\tau}{c}}$  
supported only on $\tau'_1$, $\dots$, $\tau'_{r'}$, such that $$\varphi_{\SpehRepresentation{\tau}{c}}\left(\tau'_j\right) = \left(m_j^c\right) = \left(m_j,\dots,m_j\right),$$ for any $1 \le j \le r'$, where we keep the notation from the end of \Cref{subsec:generic-representations}. In particular, $\SpehRepresentation{\tau}{1} = \tau$. In the special case where $\tau$ is an irreducible cuspidal representation of $\GL_k\left(\finiteField\right)$, the representation $\SpehRepresentation{\tau}{c}$ is the irreducible representation of $\GL_{kc}\left(\finiteField\right)$ corresponding to the parameter $\varphi_{\SpehRepresentation{\tau}{c}}$ supported on $\tau$ with $\varphi_{\SpehRepresentation{\tau}{c}}\left(\tau\right) = \left(1^c\right) = \left(1,\dots,1\right)$. In the special case $k=1$, we have that $\tau = \chi \colon \multiplicativegroup{\finiteField} \to \multiplicativegroup{\cComplex}$ is a character and $\SpehRepresentation{\tau}{c} = \chi_{\GL_c} \colon \GL_c\left(\finiteField\right) \to \multiplicativegroup{\cComplex}$ is the character given by $$\chi_{\GL_c}\left(h\right) = \chi\left(\det h\right).$$

Carmon proved in~\cite[Theorem 6.18]{Carmon2023} that for any irreducible generic representation $\tau$ of $\GL_{kc}\left(\finiteField\right)$, the Speh representation $\SpehRepresentation{\tau}{c}$ admits a unique (up to scalar) $\kcNotation{k}{c}{\fieldCharacter}$-vector (this is analogous to~\cite[Theorem 4]{CaiFriedbergGourevitchKaplan2023}). Furthermore, he proved in~\cite[Theorem 6.21]{Carmon2023} that if $v \in \SpehRepresentation{\tau}{c}$ is a $\kcNotation{k}{c}{\fieldCharacter}$ vector and $h \in \GL_c\left(\finiteField\right)$, then $\SpehRepresentation{\tau}{c}\left(\diag^k\left(h\right)\right) v = \centralCharacter{\tau}\left(\det h\right) v$, where $\centralCharacter{\tau}$ is the central character of $\tau$ and $$\diag^k\left(h\right) = \diag\left(h,\dots,h\right) \in \GL_{kc}\left(\finiteField\right).$$
This result is analogous to~\cite[Lemma 12]{CaiFriedbergGourevitchKaplan2023}.

\subsection{Bessel functions}\label{subsec:bessel-functions}

Let $\tau$ be an irreducible generic representation of $\GL_k\left(\finiteField\right)$ and let $c \ge 1$. Choose an inner product $\innerproduct{\cdot}{\cdot}$ on the space of $\SpehRepresentation{\tau}{c}$ that is invariant under the action of $\GL_{kc}\left(\finiteField\right)$. Let $0 \ne v \in \SpehRepresentation{\tau}{c}$ be a $\kcNotation{k}{c}{\fieldCharacter}$ vector. The \emph{Bessel--Speh function of $\SpehRepresentation{\tau}{c}$} with respect to the additive character $\fieldCharacter$, which we denote by $\besselSpehFunction{\tau}{c}$, is defined as the following matrix coefficient of $\SpehRepresentation{\tau}{c}$. For any $g \in \GL_{kc}\left(\finiteField\right)$,
$$\besselSpehFunction{\tau}{c}\left(g\right) = \frac{\innerproduct{\SpehRepresentation{\tau}{c}\left(g\right)v}{v}}{\innerproduct{v}{v}}.$$
Since the $\kcNotation{k}{c}{\fieldCharacter}$ vector of $\SpehRepresentation{\tau}{c}$ is unique (up to scalar multiplication), this definition does not depend on the choice $v$. Since $\SpehRepresentation{\tau}{c}$ is irreducible, this definition also does not depend on the choice of the inner product $\innerproduct{\cdot}{\cdot}$.

When $c = 1$, we denote $\besselFunction_{\tau, \fieldCharacter} = \besselSpehFunction{\tau}{1}$ and simply call $\besselFunction_{\tau, \fieldCharacter}$ the \emph{Bessel function of $\tau$} with respect to additive character $\fieldCharacter$.

By its definition and by the last property listed in \Cref{subsec:speh-representations}, the Bessel--Speh function satisfies the following properties:
\begin{enumerate}
	\item $\besselSpehFunction{\tau}{c}\left(\IdentityMatrix{kc}\right) = 1$.
	\item $\besselSpehFunction{\tau}{c}\left(u_1 g u_2\right) = \fieldCharacterkc{k}{c}\left(u_1 u_2\right) \besselSpehFunction{\tau}{c}\left(g\right)$ for any $u_1, u_2 \in \UnipotentRadicalForWss{k}{c}$ and any $g \in \GL_{kc}\left(\finiteField\right)$.
	\item \label{property:diagonal-embedding-action} $\besselSpehFunction{\tau}{c}\left(\diag^k\left(h\right) g\right) = \besselSpehFunction{\tau}{c}\left(g \diag^k\left(h\right)\right) = \centralCharacter{\tau}\left(\det h\right) \besselSpehFunction{\tau}{c}\left(g\right)$ for any $g \in \GL_{kc}\left(\finiteField\right)$ and $h \in \GL_c\left(\finiteField\right)$.
\end{enumerate}

\subsubsection{Character averaging formula}\label{subsec:bessel-character-averaging-formula}
We may express the Bessel--Speh function using the character of $\SpehRepresentation{\tau}{c}$ as follows. Consider the projection operator $\ProjectionOperator_{\kcNotation{k}{c}{\fieldCharacter}}$ from $\SpehRepresentation{\tau}{c}$ to the space spanned by its $\kcNotation{k}{c}{\fieldCharacter}$ vectors. We have that $\ProjectionOperator_{\kcNotation{k}{c}{\fieldCharacter}}$ is given by the following formula. For $w \in \SpehRepresentation{\tau}{c}$, $$\ProjectionOperator_{\kcNotation{k}{c}{\fieldCharacter}} w = \frac{1}{\sizeof{\UnipotentRadicalForWss{k}{c}}}\sum_{u \in \UnipotentRadicalForWss{k}{c}} \fieldCharacterkc{k}{c}^{-1}\left(u\right) \SpehRepresentation{\tau}{c}\left(u\right) w.$$
On the other hand, if $0 \ne v \in \SpehRepresentation{\tau}{c}$ is a $\kcNotation{k}{c}{\fieldCharacter}$ vector then $\ProjectionOperator_{\kcNotation{k}{c}{\fieldCharacter}}$ is given by the formula
$$\ProjectionOperator_{\kcNotation{k}{c}{\fieldCharacter}} w = \frac{\innerproduct{w}{v}}{\innerproduct{v}{v}} v.$$
Completing $v$ to a basis of $\SpehRepresentation{\tau}{c}$ and comparing the traces of the two different expressions for $\ProjectionOperator_{\kcNotation{k}{c}{\fieldCharacter}} \circ \SpehRepresentation{\tau}{c} \left(g\right)$ for $g \in \GL_{kc}\left(\finiteField\right)$, we arrive at the formula
$$\besselSpehFunction{\tau}{c}\left(g\right) = \frac{1}{\sizeof{\UnipotentRadicalForWss{k}{c}}}\sum_{u \in \UnipotentRadicalForWss{k}{c}} \fieldCharacterkc{k}{c}^{-1}\left(u\right) \trace \left(\SpehRepresentation{\tau}{c}\left(ug\right)\right).$$ 

\subsubsection{The values $\specialBesselSpeh{\tau}$}

Let $\tau$ be an irreducible generic representation of $\GL_k\left(\finiteField\right)$. For any $c \ge 1$ and any $h \in \GL_c\left(\finiteField\right)$, we denote $$\specialBesselSpeh{\tau}\left(h\right) = \begin{dcases}
	\tau\left(\det h\right) \fieldCharacter\left(\trace h^{-1}\right) & k = 1, \\
	\besselSpehFunction{\tau}{\fieldCharacter}\begin{pmatrix}
		& \IdentityMatrix{\left(k-1\right)c}\\
		h
	\end{pmatrix} & k \ge 2.
\end{dcases}$$
Note that the notation $\specialBesselSpeh{\tau}\left(h\right)$ does not include the number $c$ nor the number $k$. These should be inferred from $h$ and $\tau$, respectively. It follows from Property (\ref{property:diagonal-embedding-action}) in \Cref{subsec:bessel-functions} that for any $c \ge 1$, the assignment $h \mapsto \specialBesselSpeh{\tau}\left(h\right)$ is a class function of $\GL_c\left(\finiteField\right)$.

We also define for any $h \in \GL_c\left(\finiteField\right)$ the normalized version, $$\specialBesselSpehNormalized{\tau}\left(h\right) = q^{\frac{\left(k-1\right) c^2}{2}} \specialBesselSpeh{\tau}\left(h\right).$$

In~\cite{CarmonZelingher2025}, we proved the following multiplicativity results.

\begin{theorem}\label{thm:multiplicativitiy-of-special-values-of-bessel-functions-with-respect-to-tau}
	Suppose that $k = k_1 + k_2$ and let $\tau_1$ and $\tau_2$ be irreducible generic representations of $\GL_{k_1}\left(\finiteField\right)$ and $\GL_{k_2}\left(\finiteField\right)$, respectively. Suppose that $\tau$ is the unique irreducible generic subrepresentation of the parabolically induced representation $\tau_1 \circ \tau_2$. Then for any $c \ge 1$ and any $h \in \GL_c\left(\finiteField\right)$, we have $$\specialBesselSpeh{\tau}\left(h\right) = q^{-c^2} \sum_{\substack{x,y \in \GL_c\left(\finiteField\right)\\
			xy = -h}} \specialBesselSpeh{\tau_1}\left(x\right) \specialBesselSpeh{\tau_2}\left(y\right).$$
	Equivalently,
	$$\specialBesselSpehNormalized{\tau}\left(h\right) = q^{-\frac{c^2}{2}} \sum_{\substack{x,y \in \GL_c\left(\finiteField\right)\\
			xy = -h}} \specialBesselSpehNormalized{\tau_1}\left(x\right) \specialBesselSpehNormalized{\tau_2}\left(y\right).$$
\end{theorem}

\begin{theorem}\label{thm:multiplicativitiy-of-special-values-of-bessel-functions-with-respect-to-conjugacy-class}
	Suppose that $c = c_1 + c_2$, and let $h_1 \in \GL_{c_1}\left(\finiteField\right)$ and $h_2 \in \GL_{c_2}\left(\finiteField\right)$. Then for $h = \diag\left(h_1, h_2\right) \in \GL_c\left(\finiteField\right)$, we have
	$$\frac{1}{\sizeof{\UnipotentRadical_{(c_1, c_2)}}} \sum_{n \in \UnipotentRadical_{(c_1, c_2)}} \specialBesselSpeh{\tau}\left(hn\right) = q^{-c_1 c_2 \left(k-1\right)} \specialBesselSpeh{\tau}\left(h_1\right) \specialBesselSpeh{\tau}\left(h_2\right).$$
	Equivalently,
	$$\frac{1}{\sizeof{\UnipotentRadical_{(c_1, c_2)}}} \sum_{n \in \UnipotentRadical_{(c_1, c_2)}} \specialBesselSpehNormalized{\tau}\left(hn\right) =  \specialBesselSpehNormalized{\tau}\left(h_1\right) \specialBesselSpehNormalized{\tau}\left(h_2\right).$$	
\end{theorem}

\subsubsection{Gamma factors}
Since for any $c \ge 1$, the assignment $\GL_c\left(\finiteField\right) \to \cComplex$ given by $h \mapsto \specialBesselSpeh{\tau}\left(h\right)$ is a class function, then for any representation $\pi$ of $\GL_c\left(\finiteField\right)$, the operator
$$\GKGaussSum{\pi}{\tau}{\fieldCharacter} \coloneq q^{\frac{\left(k-2\right)c^2}{2}} \sum_{h \in \GL_c\left(\finiteField\right)} \specialBesselSpeh{\tau}\left(h\right) \pi\left(h\right) = q^{-\frac{c^2}{2}} \sum_{h \in \GL_c\left(\finiteField\right)} \specialBesselSpehNormalized{\tau}\left(h\right) \pi\left(h\right)$$
is an element of $\Hom_{\GL_c\left(\finiteFieldExtension{k}\right)}\left(\pi, \pi\right)$. Therefore, if $\pi$ is irreducible, by Schur's lemma there exists a complex number $\GKPreGammaFactor{\pi}{\tau}{\fieldCharacter} \in \cComplex$ such that $$\GKGaussSum{\pi}{\tau}{\fieldCharacter} = \GKPreGammaFactor{\pi}{\tau}{\fieldCharacter} \cdot \idmap_\pi.$$
We denote $$\GKGammaFactor{\pi}{\tau}{\fieldCharacter} = \centralCharacter{\pi}\left(-1\right)^{k-1} \GKPreGammaFactor{\pi}{\tau}{\fieldCharacter}.$$
We call $\GKGammaFactor{\pi}{\tau}{\fieldCharacter}$ (and $\GKPreGammaFactor{\pi}{\tau}{\fieldCharacter}$) the \emph{Ginzburg--Kaplan gamma factor}. It is a finite field analog of a new construction of the gamma factor corresponding to the tensor product representation associated to the representations $\pi$ and $\tau$ due to Ginzburg~\cite{ginzburg2019tensor} in the global case, and due to Kaplan~\cite[Appendix A]{kaplan2018},~\cite{Kaplan2023} in the local case.

In~\cite[Theorem 3.24]{CarmonZelingher2025}, using the results of~\cite{Zelingher2024b}, we proved that this gamma factor equals the tensor product $\varepsilon_0$-factor corresponding to $\pi$ and $\tau$.
\begin{theorem}\label{thm:equality-of-GK-factors-and-epsilon-factors}
	For any irreducible representation $\pi$ of $\GL_c\left(\finiteField\right)$ and any irreducible generic representation $\tau$ of $\GL_k\left(\finiteField\right)$, we have the equality
	$$\GKGammaFactor{\pi}{\tau}{\fieldCharacter} = \varepsilon_0\left(\pi \times \tau, \fieldCharacter\right).$$
\end{theorem}
\begin{remark}
	For $k = 1$, this equality follows from Kondo's theorem~\cite{Kondo1963} (\Cref{thm:kondo-godement-jacquet-computation}). See also~\cite{Macdonald80}.
\end{remark}

\subsection{An identity}
The following identity shows a remarkable relation between special values of Bessel--Speh functions and exotic matrix Kloosterman sums.

\begin{theorem}\label{thm:bessel-speh-is-an-exotic-kloosterman-sum-generic}
	Let $\tau$ be an irreducible generic representation of $\GL_k\left(\finiteField\right)$ with cuspidal support $\left\{\tau_1, \dots, \tau_s\right\}$, where $\tau_j$ is an irreducible cuspidal of $\GL_{k_j}\left(\finiteField\right)$ corresponding to the Frobenius orbit $\left[\alpha_j\right]$, where $\alpha_j \colon \multiplicativegroup{\finiteFieldExtension{k_j}} \to \multiplicativegroup{\cComplex}$ is a regular character. Denote $\lambda = \left(k_1,\dots,k_s\right) \vdash k$ and $\alpha = \alpha_1 \times \dots \times \alpha_s \colon \multiplicativegroup{\finiteFieldExtension{\lambda}} \to \multiplicativegroup{\cComplex}$. Then for any $h \in \GL_c\left(\finiteField\right)$,\\
	\begin{align*}
		\specialBesselSpeh{\tau}\left(h\right) &= \left(-1\right)^{\left(k+s\right)c} q^{-\left(k-1\right)c^2} \ExoticKloosterman\left(\alpha^{-1}, \fieldCharacter,\left(-1\right)^{k-1} h^{-1}\right).
	\end{align*}
	Equivalently,
	$$\specialBesselSpehNormalized{\tau}\left(h\right) = \left(-1\right)^{\left(k+s\right)c} \ExoticKloostermanNormalized\left(\alpha^{-1}, \fieldCharacter, \left(-1\right)^{k-1} h^{-1}\right).$$
\end{theorem}
\begin{proof}
	Since $\specialBesselSpehNormalized{\tau}$ is a class function, by the definition of exotic matrix Kloosterman sums, it suffices to show that for any $c \ge 1$ and for any irreducible representation $\pi$ of $\GL_c\left(\finiteField\right)$, we have the equality
	$$\left(-1\right)^{\left(k+s\right)c} q^{-\frac{c^2}{2}} \sum_{h \in \GL_c\left(\finiteField\right)} \specialBesselSpehNormalized{\tau^{\vee}}\left( \left(-1\right)^{k-1} h^{-1} \right) \pi\left(h\right) = \TwistedGaussSum{\pi}{\alpha}{\fieldCharacter}.$$
	This is equivalent to $$\left(-1\right)^{\left(k+s\right)c} \centralCharacter{\pi}\left(-1\right)^{k-1} q^{-\frac{c^2}{2}} \sum_{h \in \GL_c\left(\finiteField\right)} \specialBesselSpehNormalized{\tau^{\vee}}\left(h\right) \trace \pi^{\vee}\left(h\right) = \GKGaussSumScalar{\pi}{\alpha}{\fieldCharacter} \cdot \dim \pi,$$
	which is equivalent to $$\left(-1\right)^{\left(k+s\right)c} \centralCharacter{\pi}\left(-1\right)^{k-1} \trace \GKGaussSum{\pi^{\vee}}{\tau^{\vee}}{\fieldCharacter} = \GKGaussSumScalar{\pi}{\alpha}{\fieldCharacter} \cdot \dim \pi.$$
	By \Cref{thm:equality-of-GK-factors-and-epsilon-factors}, we have the identity $$\GKGaussSum{\pi^{\vee}}{\tau^{\vee}}{\fieldCharacter} = \centralCharacter{\pi}\left(-1\right)^{k-1} \varepsilon_0\left(\pi^{\vee} \times \tau^{\vee}, \fieldCharacter\right) \cdot \idmap_\pi,$$
	and therefore by \Cref{prop:non-abelian-exotic-gauss-sum-equals-epsilon-factor} the result follows.
\end{proof}

\begin{remark}
	When $c=1$, this theorem specializes to a result of Curtis and Shinoda~\cite[Lemma 3.5]{curtis2004zeta}. Their method is as follows. They use the formula from \Cref{subsec:bessel-character-averaging-formula} (for $c=1$). In order to do so, they show that for any $x \in \multiplicativegroup{\finiteField}$, and any $u \in \UnipotentSubgroup_k$, the element \begin{equation}\label{eq:element-with-characteristic-polynomial}
		\begin{pmatrix}
			& \IdentityMatrix{k-1}\\
			x
		\end{pmatrix} u
	\end{equation} is a regular element (see \Cref{subsec:conjugacy-classes-of-gl-n}), and therefore its conjugacy class is determined by its characteristic polynomial. They show that for any such $u$, the value $\fieldCharacter\left(u\right)$ is determined by a coefficient of the characteristic polynomial of \eqref{eq:element-with-characteristic-polynomial}. They then reduce the computation of $\besselFunction_{\pi, \fieldCharacter}\left(\begin{smallmatrix}
		& \IdentityMatrix{c-1}\\
		x
	\end{smallmatrix}\right)$ to counting for each polynomial of degree $k$ over $\finiteField$, how many $u \in \UnipotentSubgroup_k$ exist such that \eqref{eq:element-with-characteristic-polynomial} has this given polynomial as its characteristic polynomial. Finally, they use the formula for Deligne--Lusztig characters to complete the computation. See \Cref{example:curtis-shinoda-trick} for a similar computation.
	
	This sort of computation does not generalize well for $c \ge 2$. The first difficulty is that determining the conjugacy class of $$\begin{pmatrix}
		& \IdentityMatrix{\left(k-1\right)c}\\
		h
	\end{pmatrix} u$$
	for $h \in \GL_c\left(\finiteField\right)$ and $u \in \UnipotentRadicalForWss{k}{c}$ seems much more difficult. The second difficulty is that it is hard to write down the character of the Speh representation $\SpehRepresentation{\tau}{c}$ explicitly, see Appendix \ref{appendix:character-of-speh-representation}.
	
	The proof we gave here is based on the relation to gamma factors, and is similar to a proof we gave (with Soudry) of the result of Curtis--Shinoda using gamma factors~\cite[Theorem 4.10]{SoudryZelingher2023}.
\end{remark}

\subsection{Multiplicativity property}
Let $\lambda = \left(k_1,\dots,k_s\right) \vdash k$ and let $\alpha = \alpha_1 \times \dots \times \alpha_s \colon \multiplicativegroup{\finiteFieldExtension{\lambda}} \to \multiplicativegroup{\cComplex}$ be a multiplicative character such that $\alpha_j \colon \multiplicativegroup{\finiteFieldExtension{k_j}} \to \multiplicativegroup{\cComplex}$ is a regular character for every $j$. Using the relation from \Cref{thm:bessel-speh-is-an-exotic-kloosterman-sum-generic} and \Cref{thm:multiplicativitiy-of-special-values-of-bessel-functions-with-respect-to-conjugacy-class}, we establish the following multiplicativity property of exotic matrix Kloosterman sums.

\begin{theorem}\label{thm:multiplicativity-of-exotic-kloosterman-sums-without-proof}
	Let $c = c_1 + c_2$, and let $h_1 \in \GL_{c_1}\left(\finiteField\right)$ and $h_2 \in \GL_{c_2}\left(\finiteField\right)$.
	\begin{enumerate}
		\item We have the identity
		$$\frac{1}{\sizeof{\UnipotentRadical_{(c_1, c_2)}}} \sum_{n \in \UnipotentRadical_{(c_1, c_2)}} \ExoticKloosterman\left(\alpha, \fieldCharacter, \diag\left(h_1,h_2\right)n\right) = q^{\left(k-1\right) c_1 c_2} \ExoticKloosterman\left(\alpha,\fieldCharacter, h_1\right) \ExoticKloosterman\left(\alpha,\fieldCharacter, h_2\right).$$
		Equivalently,
		$$\frac{1}{\sizeof{\UnipotentRadical_{(c_1, c_2)}}} \sum_{n \in \UnipotentRadical_{(c_1, c_2)}} \ExoticKloostermanNormalized\left(\alpha, \fieldCharacter, \diag\left(h_1,h_2\right)n\right) = \ExoticKloostermanNormalized\left(\alpha, \fieldCharacter, h_1\right) \ExoticKloostermanNormalized\left(\alpha, \fieldCharacter, h_2\right).$$
		\item Suppose that $h_1$ and $h_2$ do not have any common eigenvalues in the algebraic closure $\algebraicClosure{\finiteField}$, then
		$$\ExoticKloosterman\left( \alpha, \fieldCharacter, \diag\left(h_1, h_2\right) \right) = q^{\left(k - 1\right) c_1 c_2} \ExoticKloosterman\left( \alpha, \fieldCharacter, h_1 \right) \ExoticKloosterman\left( \alpha, \fieldCharacter, h_2 \right).$$
		Equivalently,
		$$\ExoticKloostermanNormalized\left( \alpha, \fieldCharacter, \diag\left(h_1, h_2\right) \right) = \ExoticKloostermanNormalized\left( \alpha, \fieldCharacter, h_1 \right) \ExoticKloostermanNormalized\left( \alpha, \fieldCharacter, h_2 \right).$$	
	\end{enumerate}
\end{theorem}
\begin{remark}
	The second part of this theorem was first given by \Erdelyi{}--\Toth{} in~\cite[Theorem 1.1]{erdelyi2021matrix} for the case $\lambda = \left(1,1\right)$ and $\alpha = 1$. Later with Carmon we generalized it to $\lambda = \left(1^k\right) = \left(1,\dots,1\right)$ and arbitrary $\alpha \colon \left(\multiplicativegroup{\finiteField}\right)^k \to \multiplicativegroup{\cComplex}$ in~\cite[Theorem 4.1]{CarmonZelingher2025}. Our proof in~\cite{CarmonZelingher2025} relies on a special case of \Cref{thm:bessel-speh-is-an-exotic-kloosterman-sum-generic}. We give a similar proof here.
	
	We also remark that the second part of this theorem will follow from our results in the next section (specifically, from \Cref{thm:general-formula-for-exotic-kloosterman-sum-of-conjugacy-class}).
\end{remark}
\begin{proof}
	For every $1 \le j \le s$, let $\tau_j$ be the irreducible cuspidal representation corresponding to the Frobenius orbit $\left[\alpha_j\right]$. Let $\tau$ be the unique irreducible generic representation of $\GL_k\left(\finiteField\right)$ with cuspidal support $\left\{\Contragradient{\tau_1},\dots,\Contragradient{\tau_s}\right\}$.
	Let $h = \diag\left(h_1, h_2\right)$. The first part of the theorem is straightforward by combining \Cref{thm:bessel-speh-is-an-exotic-kloosterman-sum-generic} and \Cref{thm:multiplicativitiy-of-special-values-of-bessel-functions-with-respect-to-conjugacy-class}.
	
	Regarding the second part, suppose that $h_1$ and $h_2$ have no common eigenvalues in $\algebraicClosure{\finiteField}$. Then $nh$ is conjugate to $h$ for any $n \in \UnipotentRadical_{\left(c_1,c_2\right)}$, and therefore, using the fact that $h' \mapsto \ExoticKloosterman\left(\alpha, \fieldCharacter, h'\right)$ is a class function, we get $$\ExoticKloosterman\left(\alpha, \fieldCharacter, h\right) = q^{-c_1 c_2} \sum_{n \in \UnipotentRadical_{\left(c_1, c_2\right)}} \ExoticKloosterman\left(\alpha, \fieldCharacter,n h \right).$$
	The second part now follows from the first part.
\end{proof}

\section{Relation to Hall--Littlewood polynomials}\label{sec:relation-to-hall-littlewood-polynomials}

In this section, we give an identity that expresses exotic matrix Kloosterman sums in terms of modified Hall--Littlewood polynomials corresponding to the conjugacy class structure of a given matrix. This is a generalization of~\cite{Zelingher2023}.

The proof we give here is different than the one of~\cite{Zelingher2023}. We utilize Macdonald's characteristic maps and use the results about non-abelian exotic Gauss sums from \Cref{subsec:non-abelian-exotic-gauss-sums} to establish our identity. This proof also gives us the image of the global exotic matrix Kloosterman sum function under both Macdonald characteristic maps.

We begin with reviewing the well known classification of conjugacy classes of $\GL_n\left(\finiteField\right)$. We then discuss Hall--Littlewood polynomials and Macdonald's characteristic maps. Using these, we prove our main result (\Cref{thm:general-formula-for-exotic-kloosterman-sum-of-conjugacy-class}).

\subsection{Conjugacy classes of $\GL_{n}\left(\finiteField\right)$}\label{subsec:conjugacy-classes-of-gl-n}
In this section, we review the well-known classification of conjugacy classes of $\GL_n\left(\finiteField\right)$.

Recall that if $x,x' \in \GL_n\left(\finiteField\right)$ are such that there exists an element $y \in \GL_n\left(\algebraicClosure{\finiteField}\right)$ such that $x' = yxy^{-1}$, then there exists $y' \in \GL_n\left(\finiteField\right)$ such that $x' = y' x (y')^{-1}$. See~\cite[Corollary 18 on Page 477]{DummitFoote2004}. Henceforth, when discussing conjugation of matrices in $\GL_n\left(\finiteField\right)$, we will never specify over which field, since if two elements are conjugate over a field extension, then they are also conjugate over $\finiteField$.

If $\xi \in \multiplicativegroup{\finiteFieldExtension{k}}$, we call the set $\left[\xi\right] \coloneq \{ \xi^{q^j} \mid j \ge 0 \}$ the \emph{Frobenius orbit} of $\xi$. Given such $\xi$, its \emph{degree}, denoted $\deg \xi$, is the size of its Frobenius orbit $\left[\xi\right]$. It is a number dividing $k$. An element $\xi$ has degree $d = \deg \xi$ if and only if $\xi$ is an element of $\finiteFieldExtension{d}$, but not of any proper subfield of $\finiteFieldExtension{d}$. Note that $\xi \in \multiplicativegroup{\finiteFieldExtension{k}}$ is of degree $k$ if and only if $\xi$ has an irreducible $\finiteField$-characteristic polynomial.

For any $\xi \in \multiplicativegroup{\finiteFieldExtension{d}}$ of degree $d$, let $h_{\xi} \in \GL_{d}\left(\finiteField\right)$ be a matrix whose eigenvalues over $\algebraicClosure{\finiteField}$ are $\xi, \xi^q, \dots, \xi^{q^{d-1}}$. For example, if the $\finiteField$-characteristic polynomial of $\xi$ is $\sum_{i=0}^d a_i X^i$, where $a_d = 1$ and $a_0,\dots,a_{d-1} \in \finiteField$, then the companion matrix of $\xi$ $$\begin{pmatrix}
	0 & 0 & 0 & \dots & 0 & -a_0 \\
	1 & 0 & 0 & \dots & 0 & -a_1 \\
	0 & 1 & 0 & \dots & 0 & -a_2 \\
	\vdots & \vdots & \ddots & \ddots & \vdots & \vdots \\ 
	0 & 0 & \dots &  1 & 0 & -a_{d-2} \\
	0 & 0 & 0 & \dots & 1 & -a_{d-1}
\end{pmatrix}$$
can be a taken as a choice of $h_{\xi}$. An element $h \in \GL_d\left(\finiteField\right)$ is called \emph{regular elliptic} if it is conjugate to a matrix of the form $h_{\xi}$ as above. Equivalently, $h \in \GL_d\left(\finiteField\right)$ is regular elliptic if its characteristic polynomial is irreducible over $\finiteField$.

Given $\xi$ and $h_{\xi}$ as above and an integer $m$, we define the generalized Jordan block $$J_{(m)}\left(h_{\xi}\right) = \begin{pmatrix}
	h_{\xi} & \IdentityMatrix{d}\\
	& h_{\xi} & \IdentityMatrix{d}\\
	& & \ddots & \ddots\\
	& & & h_{\xi} & \IdentityMatrix{d}\\
	& & & & h_{\xi}
\end{pmatrix} \in \GL_{md}\left(\finiteField\right).$$
This matrix is conjugate over the algebraic closure to $$\diag\left(J_{(m)}\left(\xi\right), J_{(m)}\left(\xi^q\right), \dots, J_{(m)}(\xi^{q^{d-1}}) \right),$$ where for $\xi \in \multiplicativegroup{\finiteFieldExtension{d}}$, $$J_{(m)}\left(\xi\right) = \begin{pmatrix}
	\xi & 1\\
	& \xi & 1\\
	& & \ddots & \ddots\\
	& & & \xi & 1\\
	& & & & \xi
\end{pmatrix} \in \GL_m\left(\finiteFieldExtension{d}\right).$$

Let $\mu = \left(m_1,\dots,m_t\right) \vdash m$ be a partition. We define the generalized Jordan matrix $$J_{\mu}\left(h_{\xi}\right) = \diag\left(J_{\left(m_1\right)}\left(h_{\xi}\right), \dots, J_{\left(m_t\right)}\left(h_{\xi}\right)\right).$$

Using these Jordan matrices we can state the result regarding classification of conjugacy classes of $\GL_n\left(\finiteField\right)$.

\begin{theorem}[{\cite[IV.2, Pages 270-271]{macdonald1998symmetric}}]\label{thm:conjugacy-class-classification}
	For every $h \in \GL_n\left(\finiteField\right)$ there exist tuples of positive integers $\left(a_1, b_1\right),\dots, \left(a_r, b_r\right)$ with $a_1b_1 + \dots + a_r b_r = n$, partitions $\mu_1 \vdash b_1$, $\dots$, $\mu_r \vdash b_r$ and elements $\xi_1 \in \multiplicativegroup{\finiteFieldExtension{a_1}}$, $\dots$, $\xi_r \in \multiplicativegroup{\finiteFieldExtension{a_r}}$, such that $h$ is conjugate to $$\diag\left( J_{\mu_1}\left(h_{\xi_1}\right), \dots, J_{\mu_r}\left(h_{\xi_r}\right) \right),$$
	where for every $i$, $\xi_i$ is of degree $a_i$, and for $i \ne j$, the Frobenius orbits of $\xi_i$ and $\xi_j$ are different.
	
	Moreover, for every $h \in \GL_n\left(\finiteField\right)$, the sequence of tuples $\left(\left(\left[\xi_i\right], \mu_i\right) \right)_{i=1}^r$
	is unique, up to reordering.
\end{theorem}

A matrix $h \in \GL_{n}\left(\finiteField\right)$ is called \emph{regular} if it is conjugate to a matrix as in the theorem with $\mu_1 = \left(b_1\right)$, $\dots$, $\mu_t = \left(b_t\right)$, i.e., $h$ is conjugate to $\diag\left(J_{\left(b_1\right)}\left(h_{\xi_1}\right),\dots,J_{\left(b_t\right)}\left(h_{\xi_t}\right)\right)$ (where $\left[\xi_i\right] \ne \left[\xi_j\right]$ for $i \ne j$). A matrix $h \in \GL_{n}\left(\finiteField\right)$ is called \emph{semi-simple} if it is conjugate to a matrix as in the theorem with $\mu_1 = \left(1^{b_1}\right)$, $\dots$, $\mu_t = \left(1^{b_t}\right)$, that is, if $h$ is conjugate to $\diag\left(h_{\xi_1},\dots,h_{\xi_1}, \dots, h_{\xi_t},\dots,h_{\xi_t}\right)$, where for every $j$, the element $h_{\xi_j}$ appears $b_j$ times.

\subsection{Macdonald's characteristic maps}

In this section, we review Macdonald's characteristic maps~\cite[Pages 276-286]{macdonald1998symmetric}. Our discussion is based on~\cite[Section 5.3.1]{Thiem2004} but we use different notations. In particular, we use Frobenius orbits instead of irreducible polynomials. We follow the convention of \cite{macdonald1998symmetric}, commonly used in the local Langlands correspondence, in which the partition $(n)$ corresponds to the Steinberg representation, while the partition $(1^n)$ corresponds to the Speh representation. We take this opportunity to mention that in~\cite{Zelingher2023} we followed Green's convention, except for the appendix, in which we followed Langlands's convention.

\subsubsection{Class functions}\label{subsec:class-functions}

Recall that for any $n \ge 0$ we set $\ClassFunctionsRing\left(\GL_n\left(\finiteField\right)\right)$ to be the ring of class functions on $\GL_n\left(\finiteField\right)$, see \Cref{subsubsec:class-functions}.

Suppose that  $n_1 + n_2 = n$. Following the notation of Green~\cite{Green55} and Macdonald~\cite[Page 274]{macdonald1998symmetric}, given $f_i \in \ClassFunctionsRing\left(\GL_{n_i}\left(\finiteField\right)\right)$ for $i = 1,2$, we define a class function $f_1 \circ f_2 \in \ClassFunctionsRing\left(\GL_n\left(\finiteField\right)\right)$ by the formula
$$\left(f_1 \circ f_2\right)\left(g\right) = \frac{1}{\sizeof{\ParabolicSubgroup_{(n_1, n_2)}}}\sum_{\substack{x \in \GL_n\left(\finiteField\right)\\
		h_1 \in \GL_{n_1}\left(\finiteField\right),\, h_2 \in \GL_{n_2}\left(\finiteField\right)\\
		x g x^{-1} \in \diag\left(h_1, h_2\right) \UnipotentRadical_{(n_1, n_2)}}} f_1\left(h_1\right) f_2\left(h_2\right).$$
The operation $\circ$ is commutative and associative. If $\pi_1$ and $\pi_2$ are representations of $\GL_{n_1}\left(\finiteField\right)$ and $\GL_{n_2}\left(\finiteField\right)$, respectively, then we have the identity
$$\trace\left(\pi_1\right) \circ \trace\left(\pi_2\right) = \trace\left(\pi_1 \circ \pi_2\right),$$
where $\pi_1 \circ \pi_2$ denotes the parabolic induction of $\pi_1$ and $\pi_2$, see \Cref{subsec:parabolic-indudction}.

Consider the direct sum $\ClassFunctionsGLRing = \bigoplus_{n = 0}^{\infty} \ClassFunctionsRing\left(\GL_n\left(\finiteField\right)\right)$, equipped with the inner product $\innerproduct{\cdot}{\cdot}$, extended so that the spaces in the direct sum are orthogonal. Then $\ClassFunctionsGLRing$ is an algebra over $\cComplex$, with respect to the multiplication operation $\circ$. By the discussion above, $\ClassFunctionsGLRing$ has two bases. One consists of characteristic functions $\left(\delta_C\right)_C$ and the other consists of characters of irreducible representations $\left(\trace \pi\right)_{\pi}$. Here, $C$ (respectively, $\pi$) runs over conjugacy classes (respectively, over equivalence classes of irreducible representations) of $\GL_n\left(\finiteField\right)$, for all $n$. Macdonald proved that $\ClassFunctionsGLRing$ is isomorphic to a ring of symmetric functions, by an isomorphism under which the images of these two bases are explicit polynomials. Characteristic functions of conjugacy classes are mapped under this isomorphism to Hall--Littlewood polynomials, which we describe in the next sections.

\subsubsection{Hall--Littlewood polynomials}

Let $X = \left(X_1,X_2,\dots\right)$ be an infinite set of variables. For every $n \ge 0$, the symmetric group $\SymmetricGroup_n$ acts on the ring of polynomials $\cComplex\left[X_1,\dots,X_n\right]$ by permuting the variables. We denote for every $n$ the ring of \emph{symmetric polynomials with $n$ variables}, consisting of polynomials that are invariant under the action of $\SymmetricGroup_n$, by $$\Lambda_n \coloneq \cComplex\left[X_1,\dots,X_n\right]^{\SymmetricGroup_n}.$$

Let $\lambda = \left(\lambda_1,\dots,\lambda_{\ell}\right)$ be a partition. For any $n \ge \ell$ and any $t \in \cComplex$, we consider the \emph{Hall--Littlewood polynomial} $\pHallLittlewood_{\lambda}\left(X_1,\dots,X_n; t\right) \in \Lambda_n$, defined by the formula~\cite[Definition 3.4.1]{Haiman2003},~\cite[Page 208]{macdonald1998symmetric} $$\pHallLittlewood_{\lambda}\left(X_1,\dots,X_n; t\right) = \frac{1}{\prod_{j=0}^{\infty} \left[\lambda\left(i\right)\right]_t !} \sum_{w \in \SymmetricGroup_n} w\left( X^{\lambda} \frac{\prod_{1 \le i < j \le n} \left(1 - t \frac{X_j}{X_i}\right)}{\prod_{1 \le i < j \le n} \left(1 - \frac{X_j}{X_i}\right)} \right),$$
where $X^{\lambda}$ is the monomial $X_1^{\lambda_1} \dots X_{\ell}^{\lambda_{\ell}}$, and where for $i \ge 1$, $\lambda\left(i\right)$ is the number of times that $i$ appears in the partition $\lambda$ and $\lambda\left(0\right)$ is defined so that $\sum_{j=0}^{\infty} \lambda\left(j\right) = n$. Here, $$\left[k\right]_t! = \prod_{j=1}^k \left[j\right]_t \,\,\,\, \text{ where }\,\,\,\, \left[k\right]_t = \prod_{j=1}^k \frac{1 - t^j}{1-t}.$$
We also denote $$\qHallLittlewood_{\lambda}\left(X_1,\dots,X_n; t\right) = \left(\left(1-t\right)^{\ell} \prod_{i=1}^{\infty} \left[\lambda\left(i\right)\right]_t!\right) \pHallLittlewood_{\lambda}\left(X_1,\dots,X_n;t\right).$$

The space $\Lambda_n$ has a basis consisting of polynomials called \emph{Schur polynomials}~\cite[Page 40]{macdonald1998symmetric} and denoted $s_{\lambda}\left(X_1,\dots,X_n\right)$ where $\lambda$ runs over all partitions of length $\le n$. For any partition $\lambda$ of length $\le n$, the Schur polynomial corresponding to $\lambda$ is given by $s_{\lambda}\left(X_1,\dots,X_n\right) = \pHallLittlewood_{\lambda}\left(X_1,\dots,X_n;0\right)$. We equip $\Lambda_n$ with an inner product $\innerproduct{\cdot}{\cdot}$ with respect to which the Schur polynomials $\left(s_{\lambda} \mid \lengthof\left(\lambda\right) \le n\right)$ form an orthonormal basis, i.e.,
$$\innerproduct{s_{\lambda}\left(X_1,\dots,X_n\right)}{s_{\mu}\left(X_1,\dots,X_n\right)} = \delta_{\lambda, \mu} = \begin{dcases}
	0 & \lambda \ne \mu,\\
	1 & \lambda = \mu,
\end{dcases}$$
where $\delta_{\lambda, \mu}$ is Kronecker's delta function.

We will also need the \emph{transformed Hall--Littlewood polynomials}~\cite[Corollary 3.4.12]{Haiman2003}. For any $k \ge 1$, let $$p_k\left(X_1,\dots,X_n\right) = \sum_{j=1}^n X_j^k \in \Lambda_n$$ be the $k$-th power sum polynomial. For a partition $\mu = \left(k_1,\dots,k_r\right)$ set $$p_{\mu}\left(X_1,\dots,X_n\right) = \prod_{j=1}^r p_{k_j}\left(X_1,\dots,X_n\right).$$ Then $\left(p_{\mu}\left(X_1,\dots,X_n\right)\right)_{\mu}$ forms a basis of the ring $\Lambda_n$. Thus the algebra $\Lambda_n$ is a free algebra over $\cComplex$ with generators $\left(p_j\left(X_1,\dots,X_n\right)\right)_{j=1}^{\infty}$. Consider the isomorphism of rings $\Lambda_n \to \Lambda_n$ given by mapping $$p_j\left(X_1,\dots,X_n\right) \mapsto \left(1-t^j\right)^{-1} p_j\left(X_1,\dots,X_n\right).$$
The image of $\qHallLittlewood_\lambda\left(X_1,\dots,X_n; t\right)$ under this isomorphism is called \emph{the transformed Hall--Littlewood polynomial}, which we denote by $\hHallLittlewood_{\lambda}\left(X_1,\dots,X_n;t\right)$. We have that for any $t \in \cComplex$ that is not a root of unity, the bases $\left(\pHallLittlewood_{\lambda}\left(X_1,\dots,X_n;t\right) \mid \lengthof\left(\lambda\right) \le n\right)$ and $\left(\hHallLittlewood_{\mu}\left(X_1,\dots,X_n;t\right) \mid \lengthof\left(\lambda\right) \le n\right)$ are dual with respect to the inner product $\innerproduct{\cdot}{\cdot}$, i.e., $$\innerproduct{\pHallLittlewood_{\lambda}\left(X_1,\dots,X_n;t\right)}{\hHallLittlewood_{\mu}\left(X_1,\dots,X_n;t\right)} = \delta_{\lambda, \mu}.$$

Finally, we also need the \emph{modified Hall--Littlewood polynomial}, defined by the formula
$$\htHallLittlewood_{\lambda}\left(X_1,\dots,X_n;t\right)= t^{\nOfPartition\left(\lambda\right)} \hHallLittlewood_{\lambda}\left(X_1,\dots,X_n;t^{-1}\right),$$
where for a partition $\lambda = \left(\lambda_1,\dots,\lambda_{\ell}\right)$, $$\nOfPartition\left(\lambda\right) = \sum_{j=1}^{\ell} \left(j-1\right)\lambda_j.$$
We also define $$\ptHallLittlewood_{\lambda}\left(X_1,\dots,X_n;t\right)= t^{-\nOfPartition\left(\lambda\right)} \pHallLittlewood_{\lambda}\left(X_1,\dots,X_n;t^{-1}\right).$$ It is clear that $(\ptHallLittlewood_{\lambda}\left(X_1,\dots,X_n;t\right))_{\lambda}$ and $(\htHallLittlewood_{\mu}\left(X_1,\dots,X_n;t\right))_{\mu}$ are dual with respect to the inner product $\innerproduct{\cdot}{\cdot}$.

The polynomials $\pHallLittlewood_{\lambda}\left(X_1,\dots,X_n;t\right)$, $\qHallLittlewood_{\lambda}\left(X_1,\dots,X_n;t\right)$, $\hHallLittlewood_{\lambda}\left(X_1,\dots,X_n;t\right)$ and their variants are all homogeneous polynomials of degree $\sizeof{\lambda}$.

\subsubsection{Flag formula for modified Hall--Littlewood polynomials}\label{subsec:flag-formula-for-modified-hall-littlewood-polynomials}

A \emph{weak composition} of length $\ell$ of a positive integer $b$ is a sequence $\left(b_1,\dots,b_{\ell}\right)$ of non-negative integers such that $b_1 + \dots + b_{\ell} = b$.

Let $a$ and $b$ be positive integers. Given a weak composition $\lambda = \left(b_1,\dots,b_{\ell}\right)$ of $b$, a \emph{weak flag} $\mathcal{F}$ in $\finiteFieldExtension{a}^b$ of type $\lambda$ is a sequence of $\finiteFieldExtension{a}$-linear subspaces
$$\mathcal{F} \colon 0 \subset V_1 \subset \dots \subset V_{\ell} = \finiteFieldExtension{a}^b$$ such that for every $1 \le j \le \ell$,
$$\dim_{\finiteFieldExtension{a}} V_j = b_1 + \dots + b_j.$$
The group $\GL_b\left(\finiteFieldExtension{a}\right)$ acts on the set of weak flags of type $\lambda$ in $\finiteFieldExtension{a}^b$ by the rule
$$h \mathcal{F} \colon 0 \subset h V_1 \subset \dots \subset h V_{\ell} = \finiteFieldExtension{a}^b,$$
where $h \in \GL_b\left(\finiteFieldExtension{a}\right)$.

We have the following formula for modified Hall--Littlewood polynomials evaluated at $t = q^a$ (see~\cite[Proposition 2.3]{Ram2023}). For any partition $\mu$ with $\lengthof\left(\mu\right) \le n$, we have
$$\htHallLittlewood_{\mu}\left(X_1,\dots,X_n;q^a\right) = \sum_{\lambda = \left(b_1,\dots,b_{n}\right)} \#\left\{ \mathcal{F} \text{ weak flag in } \finiteFieldExtension{a}^b \text{ of type } \lambda \mid J_{\mu}\left(1\right) \mathcal{F} = \mathcal{F} \right\} X_1^{b_1} \dots X_{n}^{b_{n}},$$
where the sum is over all weak compositions $\lambda = \left(b_1,\dots,b_{n}\right)$ of $b$ of length $n$.

\begin{example}\label{example:modified-hall-littlewood-for-k-is-h_k}
	If $\mu = \left(b\right)$ then for any $0 \le j \le b$ there exists a unique subspace of $\finiteFieldExtension{a}^b$ of dimension $j$ that is invariant under the action of $J_{\mu}\left(1\right) = J_{\left(b\right)}\left(1\right)$. Thus the modified Hall--Littlewood corresponding to $\mu = \left(b\right)$ is the complete homogeneous polynomial of degree $b$, i.e., $$\htHallLittlewood_{\left(b\right)}\left(X_1,\dots,X_n; t\right) = h_b\left(X_1,\dots,X_n\right) = \sum_{\substack{i_1, \dots, i_n \ge 0 \\
			i_1 + \dots + i_n = b}} X_1^{i_i} \dots X_n^{i_n}.$$
\end{example}
\begin{remark}\label{rem:hall-littlewood-in-terms-of-green-polynomials}
	Modified Hall--Littlewood polynomials are closely related to Green polynomials~\cite[Pages 246-248]{macdonald1998symmetric} via the expansion
	$$\tilde{\hHallLittlewood}_{\mu}\left(X_1,\dots,X_n;t\right) = \sum_{\lambda \vdash \sizeof{\mu}} \frac{Q_{\lambda}^{\mu}\left(t\right)}{z_{\lambda}} p_{\lambda}\left(X_1,\dots,X_k\right).$$
	See~\cite[Section 4.3]{Zelingher2023}
\end{remark}

\subsubsection{Symmetric functions}
We form the ring of symmetric functions $\Lambda$ as in~\cite[Pages 17-19]{macdonald1998symmetric}. For any $k \ge 0$, let $\Lambda_n^k$ be the $\cComplex$-linear subspace of $\Lambda_n$ spanned by homogeneous polynomials of degree $k$. Then $\Lambda_n$ is a graded ring with respect to the grading $\Lambda_n = \bigoplus_{k \ge 0} \Lambda_n^k$. Let $\Lambda^k = \displaystyle \lim_{\longleftarrow} \Lambda^k_n$ with respect to the transition maps $\Lambda^k_n \to \Lambda^k_m$ for $n > m$ that evaluate elements of $\Lambda_n^k$ at $X_{m+1} = \dots = X_n = 0$. We denote $\Lambda = \bigoplus_{k \ge 0} \Lambda^k$ and call $\Lambda$ the \emph{ring of symmetric functions in $X = \left(X_1,X_2,\dots\right)$}. The space $\Lambda$ is equipped with an inner product $\innerproduct{\cdot}{\cdot}$ compatible with the transition maps.

For any partition $\lambda$, and any $n > m \ge \lengthof\left(\lambda\right)$, we have that the families of polynomials from the previous section are mapped to themselves under the transition maps, i.e., under the transition map $$s_\lambda\left(X_1,\dots,X_n\right) \mapsto s_\lambda\left(X_1,\dots,X_m\right),$$ and similarly for the other polynomials. Hence these families define elements in the ring $\Lambda$. We denote the corresponding elements in $\Lambda$ by $p_{\mu} = p_{\mu}\left(X\right)$, $s_{\lambda} = s_{\lambda}\left(X\right)$, $\pHallLittlewood_{\lambda}\left(X; t\right)$, $\hHallLittlewood_{\lambda}\left(X;t\right)$, etc. We have that $\left(s_{\lambda}\left(X\right)\right)_{\lambda}$ forms a Hamel basis of $\Lambda$ that is orthonormal with respect to the inner product $\innerproduct{\cdot}{\cdot}$. Similarly, for any $t \in \cComplex$ that is not a root of unity, each of $\left(\pHallLittlewood_{\lambda}\left(X; t\right)\right)_{\lambda}$ and $\left(\hHallLittlewood_{\mu}\left(X; t\right)\right)_{\mu}$ forms a basis for $\Lambda$, and these bases are dual with respect to the inner product $\innerproduct{\cdot}{\cdot}$.

For our purposes, we will also need the completion $\hat{\Lambda}$ of $\Lambda$ with respect to filtration by degree. The ring $\hat{\Lambda}$ can be thought of as the ring of symmetric formal power series with infinitely many variables.

\subsubsection{Macdonald's correspondence}
In this section, we review Macdonald's correspondence. Our reference is~\cite{Macdonald80}.

For every $k \ge 1$, let $\charactergroup{k}$ be the character group of $\multiplicativegroup{\finiteFieldExtension{k}}$, consisting of all characters $\alpha \colon \multiplicativegroup{\finiteFieldExtension{k}} \to \multiplicativegroup{\cComplex}$. Consider the direct limit $$\Gamma = \lim_{\longrightarrow} \charactergroup{k}$$ with respect to the transition maps $\alpha_m \mapsto \alpha_m \circ \FieldNorm{n}{m}$ for any $m \mid n$ and any $\alpha_m \in \charactergroup{m}$. Since the norm maps $\FieldNorm{n}{m}$ are surjective, these transition maps are injective. If $\alpha \in \multiplicativegroup{\finiteFieldExtension{n}}$ for some $n \ge 1$, let $\alpha_{\Gamma} \in \Gamma$ denote the image of $\alpha$ in $\Gamma$. Since the transition maps are injective, the degree of any character is preserved under the transition maps, and therefore we have the notion of the degree of an element $\theta \in \Gamma$, which is simply the size of its Frobenius orbit, which by definition is the set $\left[\theta\right] \coloneq \{ \Frobenius^j \theta \mid j \ge 0 \}$, where $\Frobenius \theta = \theta^q$.

A \emph{Macdonald parameter} is an assignment $\phi \colon \Gamma \to \Partitions$, such that:
\begin{enumerate}
	\item For all but finitely many $\theta \in \Gamma$, $\phi\left(\theta\right) = \left(\right)$.
	\item $\phi$ is constant on Frobenius orbits, i.e., $\phi\left(\Frobenius \theta\right) = \phi\left(\theta\right)$ for every $\theta \in \Gamma$.
\end{enumerate}
Given a Macdonald parameter $\phi$, we define its size by
$$\sizeof{\phi} = \sum_{\theta \in \Gamma} \sizeof{\phi\left(\theta\right)} = \sum_{\left[\theta\right] \in \Frobenius \backslash \Gamma} \sizeof{\phi\left(\left[\theta\right]\right)} \cdot \deg \left[\theta\right].$$
Here, $\Frobenius \backslash \Gamma$ is the set consisting of all Frobenius orbits in $\Gamma$.

Macdonald established a bijection between equivalence classes of irreducible representations of $\GL_n\left(\finiteField\right)$ and Macdonald parameters of size $n$. This bijection is known as \emph{Macdonald's correspondence}.

Given an irreducible representation of $\GL_n\left(\finiteField\right)$ with parameter $\varphi$ and Macdonald parameter $\phi$, the parameters are related as follows. If $\theta \in \Gamma$ is of degree $d$, then $\theta$ is the image in $\Gamma$ of some regular character $\alpha \in \charactergroup{d}$, i.e., $\theta = \alpha_{\Gamma}$. Since $\alpha$ is a regular character, its Frobenius orbit $\left[\alpha\right]$ corresponds to an irreducible cuspidal representation $\sigma_{\left[\alpha\right]}$ of $\GL_d\left(\finiteField\right)$. We have the relation $$\varphi\left(\sigma_{\left[\alpha\right]}\right) = \phi\left(\left[\alpha_{\Gamma}\right]\right).$$
Hence the translation between $\varphi$ and $\phi$ is simply by replacing each irreducible cuspidal representation $\sigma$ with (the image in $\Gamma$ of) the Frobenius orbit it corresponds to.

\subsubsection{Macdonald's formulation}\label{subsec:macdonald-formulation}

For any Frobenius orbit $\left[\xi\right]$ of $\multiplicativegroup{\algebraicClosure{\finiteField}}$, let $\Lambda^{\left[\xi\right]}$ be a copy of $\Lambda$. This can be thought of as follows: $\Lambda^{\left[\xi\right]}$ is the ring of symmetric functions in other variables, labeled $(X_j^{\left[\xi\right]})_{j=1}^{\infty}$. We will occasionally use this realization of $\Lambda^{\left[\xi\right]}$. If $f \in \Lambda$, we denote by $f^{\left[\xi\right]}$ its image in $\Lambda^{\left[\xi\right]}$.

Consider the ring $\Lambda^{\finiteField} = \bigotimes_{\left[\xi\right]} \Lambda^{\left[\xi\right]}$, where $\left[\xi\right]$ runs over all Frobenius orbits of $\multiplicativegroup{\algebraicClosure{\finiteField}}$, equipped with the inner product defined on pure tensors by taking the product of inner products for each component.

Let $\CharacteristicMapF \colon \ClassFunctionsGLRing \to \Lambda^{\finiteField}$ be the following linear map. Let $C$ be the conjugacy class of $h$ given in \Cref{thm:conjugacy-class-classification}. Define \begin{equation}\label{eq:definition-of-characteristic-map-for-conjugacy-classes}
	\CharacteristicMapF \left(\delta_C\right) = \prod_{j = 1}^r \ptHallLittlewood^{\left[\xi_j\right]}_{\mu_j}\left(X; q^{a_j}\right).
\end{equation}
Macdonald proved that $\CharacteristicMapF$ is an isomorphism of $\cComplex$-algebras that is also an isometry with respect to the inner products.

This isomorphism can be used to describe Green's formula for the characters of the irreducible representations of $\GL_n\left(\finiteField\right)$ in a compact way. For any Frobenius orbit $\left[\theta\right]$ of $\theta 
\in \Gamma$, define a copy $\Lambda^{\left[\theta\right]}$ of $\Lambda$. Consider the tensor product $\Lambda^{\widehat{\finiteField}} = \bigotimes_{\left[\theta\right]} \Lambda^{\left[\theta\right]}$, equipped with an inner product similarly to above. We have an isomorphism of rings $\Lambda^{\widehat{\finiteField}} \to \Lambda^{\finiteField}$ defined on the generators $p_k^{\left[\theta\right]}$ by sending \begin{equation}\label{eq:homomomorphism-from-character-group-to-to-multiplicative-group}
	p_k^{\left[\theta\right]} \mapsto \left(-1\right)^{k \deg \theta - 1} \sum_{\xi \in \multiplicativegroup{\finiteFieldExtension{k \deg \theta}}} \alpha\left(\FieldNorm{k \deg \theta}{\deg \theta}\left(\xi\right)\right) p_{\frac{k \deg \theta}{\deg \xi}}^{\left[\xi\right]},
\end{equation}
where $\alpha \in \charactergroup{\deg \theta}$ is such that $\theta = \alpha_{\Gamma}$. Note that this map is independent of the representative of the Frobenius orbit $\left[\theta\right]$. It turns out that this map is also an isometry with respect to the inner products.

The inverse isomorphism is given as follows. It is defined on the generators $p^{[\xi]}_k$ by sending
\begin{equation}\label{eq:inverse-homomomorphism-from-character-group-to-to-multiplicative-group}
	p_k^{\left[\xi\right]} \mapsto 
	\frac{\left(-1\right)^{k \deg \xi-1}}{q^{k \deg \xi} - 1}\sum_{\alpha \in \charactergroup{k \deg \xi}} \alpha^{-1}\left(\xi\right) p_{\frac{k \deg \xi}{\deg \alpha}}^{\left[\alpha\right]},
\end{equation}
where for $f \in \Lambda$ and $\alpha \in \charactergroup{k}$ we denote $f^{\left[\alpha\right]} = f^{\left[\alpha_{\Gamma}\right]}$, i.e., we replace $\left[\alpha\right]$ with the Frobenius orbit $\left[\alpha_{\Gamma}\right]$ of its image in $\Gamma$. Notice again that this is independent of the choice of the representative for the Frobenius orbit $\left[\xi\right]$. We will denote the composition $$\ClassFunctionsGLRing \to \Lambda^{\finiteField} \to \Lambda^{\widehat{\finiteField}}$$ by $\CharacteristicMapFHat$.

Suppose that $\phi \colon \Gamma \to \Partitions$ is a Macdonald parameter for an irreducible representation $\pi_{\phi}$. Then the character $\trace \pi_{\phi}$ can be computed as follows. It is the image of the element \begin{equation}\label{eq:macdonald-parameter-schur-polynomial}
	\prod_{\left[\theta\right] \in \Frobenius \backslash \Gamma} s_{\phi\left(\left[\theta\right]\right)}^{\left[\theta\right]}
\end{equation} under the composition of maps $$\Lambda^{\widehat{\finiteField}} \to \Lambda^{\finiteField} \to \ClassFunctionsGLRing.$$
In order to compute $\trace \pi_{\chi}\left(C\right)$ for a conjugacy class $C$, one needs to expand the image of \eqref{eq:macdonald-parameter-schur-polynomial} under the isomorphism $\Lambda^{\widehat{\finiteField}} \to \Lambda^{\finiteField}$ with respect to the basis $\left(\CharacteristicMapF\left(\delta_{C'}\right)\right)_{C'}$ and to compute the coefficient of $\CharacteristicMapF\left(\delta_C\right)$ (given by \eqref{eq:definition-of-characteristic-map-for-conjugacy-classes}) in this expansion. To do so, one needs to expand the Schur polynomials $s^{\left[\theta\right]}_{\phi\left(\left[\theta\right]\right)}$ with respect to the generators $p_k^{\left[\theta\right]}$, to apply the isomorphism $\Lambda^{\widehat{\finiteField}} \to \Lambda^{\finiteField}$, then to express the products of $p_{k}^{\left[\xi\right]}$ in terms of products of polynomials $p_{\mu}^{[\xi]}$ (this is the most difficult step), and then, finally, to pair with the product of modified Hall--Littlewood polynomials $$\prod_{j=1}^r \tilde{\mathrm{H}}_{\mu_j}^{\left[\xi_j\right]}\left(X; q^{a_j}\right),$$
where one uses \Cref{rem:hall-littlewood-in-terms-of-green-polynomials} to express the modified Hall--Littlewood polynomials in terms of products of $p_{\lambda_i}^{\left[\xi_i\right]}$'s. The result is Green's formula.

\subsubsection{Dual space and Macdonald's formulation}
The space $\ClassFunctionsGLRing$ can be identified with the space $\CompactlyClassFunctionsGLRing$ consisting of functions $f \colon \bigcup_{n=0}^\infty \GL_n\left(\finiteField\right) \to \cComplex$ invariant under conjugation, supported on finitely many elements. Using the inner product $\innerproduct{\cdot}{\cdot}$ on $\ClassFunctionsGLRing$, the dual space $\DualClassFunctionsGLRing$ can be identified with the space $\FullClassFunctionsGL$ consisting of all functions $f \colon \bigcup_{n=0}^\infty \GL_n\left(\finiteField\right) \to \cComplex$ invariant under conjugation. Using this identification, the characteristic map $\CharacteristicMapF \colon \ClassFunctionsGLRing \to \Lambda^{\finiteField}$ extends to a map from the dual space $\DualClassFunctionsGLRing$ to the completion $\hat{\Lambda}^{\finiteField}$ of $\Lambda^{\finiteField}$ with respect to filtration by degree. We will denote this map by $\CharacteristicMapF$, and think about it as a map from $\FullClassFunctionsGL$ to the power series analog of $\Lambda^{\finiteField}$. Similarly, the map $\CharacteristicMapFHat \colon \ClassFunctionsGLRing \to \Lambda^{\widehat{\finiteField}}$ extends to a map $\DualClassFunctionsGLRing \to \hat{\Lambda}^{\widehat{\finiteField}}$, which we keep denoting by $\CharacteristicMapFHat$.

\subsection{Expression of exotic matrix Kloosterman sums in terms of Hall--Littlewood polynomials} \label{subsec:hall-littlewood-jordan-matrices}
Let $\lambda \vdash k$ and $\alpha \colon \multiplicativegroup{\finiteFieldExtension{\lambda}} \to \multiplicativegroup{\cComplex}$ be a character. The following theorem allows us to express exotic matrix Kloosterman sums using modified Hall--Littlewood polynomials and the roots of $\ExoticKloosterman\left(\alpha, \fieldCharacter\right)$.
\begin{theorem}\label{thm:general-formula-for-exotic-kloosterman-sum-of-conjugacy-class}
	Let $n = c$, and suppose that $h$ is as in \Cref{thm:conjugacy-class-classification}. Then $$\ExoticKloosterman\left(\alpha, \fieldCharacter, h\right) = \left(-1\right)^{\left(k-1\right) c} q^{\left(k-1\right) \binom{c}{2}} \prod_{j=1}^r \tilde{\mathrm{H}}_{\mu_j}\left(\omega_{1,\left[\xi_j\right]},\dots, \omega_{k,\left[\xi_j\right]}; q^{a_j}\right),$$
	where for every $1 \le j \le r$, the complex numbers $\omega_{1, \left[\xi_j\right]},\dots,\omega_{k,\left[\xi_j\right]}$ are the roots of $\ExoticKloosterman_{\finiteFieldExtension{a_j}}\left(\alpha, \fieldCharacter\right)$ at $\xi_j$. Equivalently, $$\ExoticKloostermanNormalized\left(\alpha, \fieldCharacter, h\right) = \prod_{j=1}^r \tilde{\mathrm{H}}_{\mu_j}\left(\left(-1\right)^{\left(k-1\right)a_j} \omega_{1, \left[\xi_j\right]}^{\ast},\dots, \left(-1\right)^{\left(k-1\right)a_j} \omega_{k, \left[\xi_j\right]}^{\ast}; q^{a_j}\right),$$
	where for every $1 \le j \le r$, the complex numbers $\omega_{1,\left[\xi_j\right]}^{\ast}, \dots, \omega_{k,\left[\xi_j\right]}^{\ast}$ are the normalized roots of $\ExoticKloosterman_{\finiteFieldExtension{a_j}}\left(\alpha, \fieldCharacter\right)$ at $\xi_j$.
\end{theorem}
We will prove this theorem in the next section using Macdonald's characteristic maps. Using \Cref{example:modified-hall-littlewood-for-k-is-h_k}, we obtain a formula for regular elements, expressing their exotic matrix Kloosterman sums in terms of symmetric powers of an exotic Kloosterman sheaf.
\begin{corollary}\label{cor:formula-for-regular-elements}
	If $h = \diag\left(J_{\left(b_1\right)}\left(h_{\xi_1}\right), \dots, J_{\left(b_r\right)}\left(h_{\xi_r}\right)\right) \in \GL_c\left(\finiteField\right)$ is a regular element, then
	$$\ExoticKloosterman^{\ast}\left(\finiteFieldExtension{\lambda}, \alpha, \fieldCharacter, h\right) = \left(-1\right)^{\left(k-1\right) c} q^{-\frac{\left(k-1\right)c}{2}} \prod_{j=1}^r \trace\left(\Frobenius\mid_{\xi_j}, \SymmetricPower^{b_j} \ExoticKloosterman_{\finiteFieldExtension{a_j}}\left(\finiteFieldExtension{\lambda}, \alpha, \fieldCharacter\right)_{\xi_j}\right),$$
	where $\SymmetricPower^{b_j} \ExoticKloosterman\left(\alpha, \fieldCharacter\right)_{\xi_j}$ is the geometric stalk at $\xi_j$ of the sheaf $\SymmetricPower^{b_j} \ExoticKloosterman_{\finiteFieldExtension{a_j}}\left(\finiteFieldExtension{\lambda}, \alpha, \fieldCharacter\right)$.
\end{corollary}

Combining Theorems \ref{thm:bessel-speh-is-an-exotic-kloosterman-sum-generic} and \ref{thm:general-formula-for-exotic-kloosterman-sum-of-conjugacy-class}, we arrive at the following explicit formula for the special values $\specialBesselSpeh{\tau}$.
\begin{theorem}\label{thm:explicit-formula-for-special-bessel-speh-function-as-hall-littlewood-polynomial}
	Let $\tau$ be an irreducible generic representation of $\GL_k\left(\finiteField\right)$ with cuspidal support $\left\{\tau_1,\dots,\tau_s\right\}$, where $k_1 + \dots + k_s = k$ and for every $1 \le j \le s$, $\tau_j$ is an irreducible cuspidal representation of $\GL_{k_j}\left(\finiteField\right)$ corresponding to the Frobenius orbit of a regular character $\alpha_j \colon \multiplicativegroup{\finiteFieldExtension{k_j}} \to \multiplicativegroup{\cComplex}$. Denote $\alpha = \alpha_1 \times \dots \times \alpha_s$. Then for $h = \diag\left(J_{\mu_1}\left(h_{\xi_1}\right),\dots,J_{\mu_r}\left(h_{\xi_r}\right)\right)$ where for every $j$, $\xi_j \in \multiplicativegroup{\finiteFieldExtension{a_j}}$ is of degree $a_j$, and $\mu_j \vdash b_j$ such that $a_1b_1 + \dots + a_rb_r = c$ and such that the Frobenius orbits $\left(\left[\xi_i\right]\right)_{i=1}^r$ are mutually disjoint, we have
	$$\specialBesselSpehNormalized{\tau}\left(h\right) = \prod_{j=1}^r \htHallLittlewood_{\mu_j}\left(\left(-1\right)^{\left(s-1\right)a_j} \omega_{1,\left[\xi_j\right]}^{\ast}, \dots, \left(-1\right)^{\left(s-1\right)a_j} \omega_{k,\left[\xi_j\right]}^{\ast}; q^{a_j}\right),$$
	where for every $1 \le j \le r$, $\omega_{1,\left[\xi_j\right]}^{\ast}$, $\dots$, $\omega_{k,\left[\xi_j\right]}^{\ast}$ are the normalized roots of $\ExoticKloosterman_{\finiteFieldExtension{a}}\left(\alpha, \fieldCharacter\right)$ at $\xi_j$.
\end{theorem}

\subsection{Realization under Macdonald's characteristic maps}
For any Frobenius orbit $\left[\xi\right]$, where $\xi \in \multiplicativegroup{\algebraicClosure{\finiteField}}$, let $\omega_{1, \left[\xi\right]}^{\ast}$, $\dots$, $\omega_{k, \left[\xi\right]}^{\ast}$ be the normalized roots of $\ExoticKloosterman_{\finiteFieldExtension{\deg \left[\xi\right]}}\left(\alpha, \fieldCharacter\right)$ at $\xi$. Denote $\omega_{j, \left[\xi\right]}^{\ast \ast} = \left(-1\right)^{\left(k-1\right)\deg\left[\xi\right]} \omega_{j, \left[\xi\right]}^{\ast}$.

Let $\KloostermanGlobalClassFunction \colon \bigcup_{c=0}^{\infty} \GL_c\left(\finiteField\right) \to \cComplex$ be the function  defined by 
$$\KloostermanGlobalClassFunction = \sum_{r = 0}^{\infty} \sum_{\left\{\left(\left[\xi_1\right], \mu_1\right), \dots, \left(\left[\xi_r\right], \mu_r\right)\right\}} \left(\prod_{j = 1}^{r} \htHallLittlewood_{\mu_j}\left(\omega_{1, \left[\xi_j\right]}^{\ast \ast}, \dots, \omega_{k, \left[\xi_j\right]}^{\ast \ast} ; q^{\deg \left[\xi_j\right]}\right)\right) \delta_{\diag\left(J_{\mu_1}\left(h_{\xi_1}\right), \dots, J_{\mu_r}\left(h_{\xi_r}\right)\right)},$$
where the sum goes over all choices of sets $\left\{\left(\left[\xi_1\right], \mu_1\right), \dots, \left(\left[\xi_r\right], \mu_r\right)\right\}$ where $\left[\xi_1\right]$,$\dots$,$\left[\xi_r\right]$ are different Frobenius orbits and $\mu_1$, $\dots$, $\mu_r$ are non-empty partitions, for every $r \ge 0$. Then $\KloostermanGlobalClassFunction \in \FullClassFunctionsGL$. Our goal is compute the elements that $\KloostermanGlobalClassFunction$ corresponds to in $\hat{\Lambda}^{\finiteField}$ and in $\hat{\Lambda}^{\widehat{\finiteField}}$.

\begin{theorem}\label{thm:macdonald-exotic-kloosterman-sum-geometric-basis}
	We have the identity
	$$\CharacteristicMapF \KloostermanGlobalClassFunction = \prod_{\left[\xi\right]} \prod_{i=1}^{\infty} L^{\ast}\left( \left(-1\right)^{\left(k-1\right) \deg \left[\xi\right]}  X_i^{\left[\xi\right]}, \ExoticKloosterman_{\finiteFieldExtension{\deg \left[\xi\right]}}\left(\alpha, \fieldCharacter, \xi\right)\right)^{-1},$$
	where $\left[\xi\right]$ runs over all the Frobenius orbits in $\multiplicativegroup{\algebraicClosure{\finiteField}}$.
\end{theorem}
\begin{proof}
	By the definitions of the map $\CharacteristicMapF$ and of the function $\KloostermanGlobalClassFunction$, we have that
	\begin{equation}\label{eq:infinite-product-for-global-matrix-kloosterman-sum}
		\CharacteristicMapF \KloostermanGlobalClassFunction = \prod_{\left[\xi\right]}\left(\sum_{\mu}  \htHallLittlewood_{\mu}\left(\omega_{1, \left[\xi\right]}^{\ast \ast}, \dots, \omega_{k, \left[\xi\right]}^{\ast \ast}; q^{\deg \left[\xi\right]}\right)  \ptHallLittlewood_{\mu}^{\left[\xi\right]}\left(X;  q^{\deg \left[\xi\right]}\right) \right),
	\end{equation}
	where $\left[\xi\right]$ runs over all the Frobenius orbits in $\multiplicativegroup{\algebraicClosure{\finiteField}}$.
	
	Recall the following variant of the Cauchy identity (this follows from~\cite[(4.6) on Page 63]{macdonald1998symmetric} and the fact that $\ptHallLittlewood_{\mu}\left(X;t\right)$ and $\htHallLittlewood_{\mu}\left(X;t\right)$ are dual):
	$$\sum_{\mu} \ptHallLittlewood_{\mu}\left(X; t\right)\htHallLittlewood_{\mu}\left(Y ; t\right) = \prod_{i,j \ge 1} \frac{1}{1 - X_i Y_j},$$
	where $X$ and $Y$ represent sets of infinite variables $\left(X_1,X_2,\dots\right)$ and $\left(Y_1,Y_2,\dots\right)$, respectively, and the sum goes over all the partitions $\mu$. It follows that \eqref{eq:infinite-product-for-global-matrix-kloosterman-sum} equals
	\begin{equation}\label{eq:geometric-product-for-global-modified-hall-littlewood-function}
		\prod_{\left[\xi\right]} \prod_{j=1}^k \prod_{i=1}^{\infty} \frac{1}{1 - X^{\left[\xi\right]}_i \omega_{j, \left[\xi\right]}^{\ast \ast}} = \prod_{\left[\xi\right]} \prod_{i=1}^{\infty} L^{\ast}\left( \left(-1\right)^{\left(k-1\right) \deg \left[\xi\right]}  X_i^{\left[\xi\right]}, \ExoticKloosterman_{\finiteFieldExtension{\deg \left[\xi\right]}}\left(\alpha, \fieldCharacter, \xi\right)\right)^{-1}.
	\end{equation}
	where $\left[\xi\right]$ runs over all the Frobenius orbits in $\multiplicativegroup{\algebraicClosure{\finiteField}}$.
\end{proof}

We now move to compute the element that $\KloostermanGlobalClassFunction$ corresponds to in $\hat{\Lambda}^{\widehat{\finiteField}}$. 
\begin{theorem}\label{thm:global-matrix-kloosterman-sum-in-character-basis}
	We have the identity
	$$\CharacteristicMapFHat \KloostermanGlobalClassFunction = \prod_{\left[\beta\right]} \prod_{i=1}^{\infty} \prod_{j=1}^{\infty} \left(1 - \left(\left(-1\right)^{s} q^{-\left(\frac{k-1}{2} + j\right)}\right)^{\deg \beta} \GaussSumCharacter{\lambda, \deg \beta}{\alpha}{\beta^{-1}}{\fieldCharacter} X_i^{\left[\beta\right]}\right)^{-1},$$
	where $\left[\beta\right]$ runs over all Frobenius orbits in $\Gamma$.
\end{theorem}
\begin{proof}
	Recall the formula~\cite[Page 279]{macdonald1998symmetric}
	\begin{equation}\label{eq:logarithm-of-geometric-series}
		\sum_{i,j \ge 1} \log\left(1 - X_i Y_j\right)^{-1} = \sum_{m = 1}^{\infty} \frac{1}{m} p_m\left(X\right) p_m\left(Y\right).
	\end{equation}
	Let $\xi \in \finiteFieldExtension{a}$ be an element of degree $a$. We substitute $Y_j = \omega_{j, \left[\xi\right]}^{\ast \ast}$ for $1 \le j \le k$ and $Y_j = 0$ for $j > k$ in \eqref{eq:logarithm-of-geometric-series} to get
	$$\sum_{j=1}^k \sum_{i \ge 1} \log\left(1 - X_i^{\left[\xi\right]} \omega_{j, \left[\xi\right]}^{\ast \ast}\right)^{-1} = \sum_{m = 1}^{\infty} \frac{1}{m} p_m^{\left[\xi\right]}\left(X\right) p_m\left(\omega_{1, \left[\xi\right]}^{\ast \ast}, \dots, \omega_{k, \left[\xi\right]}^{\ast \ast}\right).$$ Thus by \eqref{eq:geometric-product-for-global-modified-hall-littlewood-function},
	$$\log \CharacteristicMapF\left(\KloostermanGlobalClassFunction\right) = \sum_{\left[\xi\right]} \sum_{m = 1}^{\infty} \frac{1}{m} p_m^{\left[\xi\right]}\left(X\right) p_m\left(\omega_{1, \left[\xi\right]}^{\ast \ast}, \dots, \omega_{k, \left[\xi\right]}^{\ast \ast}\right).$$
	By \eqref{eq:inverse-homomomorphism-from-character-group-to-to-multiplicative-group} and \eqref{eq:power-sums-of-roots-at-xi} (recall that $\omega_{j, \left[\xi\right]}^{\ast \ast} = \left(-1\right)^{\left(k-1\right) \deg\left[\xi\right]} \omega_{j, \left[\xi\right]}^{\ast}$), $$\log \CharacteristicMapFHat \KloostermanGlobalClassFunction = \sum_{\left[\xi\right]} \sum_{m = 1}^{\infty} \frac{\left(-1\right)^{k\left(m \deg \left[\xi\right] - 1\right)}}{m \left({q^{m \deg \left[\xi\right]} - 1}\right)} \ExoticKloostermanNormalized_{m, \finiteFieldExtension{\deg \left[\xi\right]}}\left(\alpha, \fieldCharacter, \xi\right) \sum_{\beta \in \charactergroup{m \deg \xi}} \beta^{-1}\left(\xi\right) p_{\frac{m \deg \xi}{\deg \beta}}^{\left[\beta\right]}.$$
	We can replace the sum over the Frobenius orbits $\left[\xi\right]$ with a sum over $\xi \in \multiplicativegroup{\algebraicClosure{\finiteField}}$ by adding a factor of $\frac{1}{\deg \xi}$. We obtain the formula
	$$\log \CharacteristicMapFHat \KloostermanGlobalClassFunction = \sum_{\xi \in \multiplicativegroup{\algebraicClosure{\finiteField}}} \sum_{m = 1}^{\infty} \sum_{\beta \in \charactergroup{m \deg \xi}} \frac{\left(-1\right)^{k\left(m \deg \left[\xi\right] - 1\right)}}{m \deg \xi \left({q^{m \deg \xi} - 1}\right)} \ExoticKloostermanNormalized_{m, \finiteFieldExtension{\deg \xi}}\left(\alpha, \fieldCharacter, \xi\right)  \beta^{-1}\left(\xi\right) p_{\frac{m \deg \xi}{\deg \beta}}^{\left[\beta\right]}.$$
	Next, we can first sum over all $d$ and then over $\xi \in \multiplicativegroup{\algebraicClosure{\finiteField}}$ with $\deg \xi = d$. Introducing a variable $n = m d$ and changing order of summation, we arrive at the formula $$\log \CharacteristicMapFHat \KloostermanGlobalClassFunction = \sum_{n = 1}^{\infty} \sum_{\beta \in \charactergroup{n}} \sum_{d \mid n} \sum_{\substack{\xi \in \multiplicativegroup{\algebraicClosure{\finiteField}}\\
			\deg \xi = d}}  \frac{\left(-1\right)^{k\left(n - 1\right)}}{n \left({q^{n} - 1}\right)} \ExoticKloostermanNormalized_{n}\left(\alpha, \fieldCharacter, \xi\right)  \beta^{-1}\left(\xi\right) p_{\frac{n}{\deg \beta}}^{\left[\beta\right]}.$$
	This can be rewritten as
	$$\log \CharacteristicMapFHat \KloostermanGlobalClassFunction = \sum_{n = 1}^{\infty} \sum_{\beta \in \charactergroup{n}} \sum_{\xi \in \multiplicativegroup{\finiteFieldExtension{n}}} \frac{\left(-1\right)^{k\left(n - 1\right)}}{n \left({q^{n} - 1}\right)} \ExoticKloostermanNormalized_{n}\left(\alpha, \fieldCharacter, \xi\right)  \beta^{-1}\left(\xi\right) p_{\frac{n}{\deg \beta}}^{\left[\beta\right]}.$$
	By \eqref{eq:gauss-sum-as-kloosterman-sum-transform}, we have that this equals $$\sum_{n = 1}^{\infty} \sum_{\beta \in \charactergroup{n}} \frac{\left(-1\right)^{sn}}{n \left({q^{n} - 1}\right)} q^{-\frac{n\left(k-1\right)}{2}} \GaussSumCharacter{\lambda, n}{\alpha}{\beta^{-1}}{\fieldCharacter} p_{\frac{n}{\deg \beta}}^{\left[\beta\right]}.$$
	We replace the summation as follows. We sum over all the regular characters $\beta$ and then over $\beta \circ \FieldNorm{m \deg \beta}{\deg \beta}$ for every $m \ge 1$. We get that $\log \CharacteristicMapFHat \KloostermanGlobalClassFunction$ is given by
	$$\sum_{d=1}^{\infty} \sum_{\substack{\beta\\
			\deg \beta = d}} \sum_{m = 1}^{\infty} \frac{\left(-1\right)^{smd}}{m d \left({q^{m d} - 1}\right)} q^{-\frac{md\left(k-1\right)}{2}} \GaussSumCharacter{\lambda, md}{\alpha}{\beta^{-1} \circ \FieldNorm{md}{d}}{\fieldCharacter} p_{m}^{\left[\beta\right]}.$$
	Using the Hasse--Davenport lifting relation (\Cref{thm:hasse-davenport-for-exotic-gauss-sums}), this becomes
	$$\sum_{d=1}^{\infty} \sum_{\substack{\beta\\
			\deg \beta = d}} \sum_{m = 1}^{\infty} \frac{1}{md \left({q^{m d} - 1}\right)} \left(\left(-1\right)^{sd} q^{-\frac{d\left(k-1\right)}{2}} \GaussSumCharacter{\lambda, d}{\alpha}{\beta^{-1}}{\fieldCharacter}\right)^m p_{m}^{\left[\beta\right]}.$$
	Using the geometric expansion $$\frac{1}{q^{md}-1} = \sum_{j=1}^{\infty} q^{-m d j},$$
	we get the series
	$$\log \CharacteristicMapFHat \KloostermanGlobalClassFunction = \sum_{\left[\beta\right]} \sum_{m = 1}^{\infty} \frac{1}{m} p_{m}^{\left[\beta\right]} \sum_{j=1}^{\infty}  \left(\left(\left(-1\right)^{s} q^{-\frac{\left(k-1\right)}{2} -j}\right)^{\deg \beta} \GaussSumCharacter{\lambda, \deg \beta}{\alpha}{\beta^{-1}}{\fieldCharacter}\right)^m,$$
	where the sum over $\left[\beta\right]$ goes over all the Frobenius orbits in $\Gamma$. It follows now from \eqref{eq:logarithm-of-geometric-series} that $\CharacteristicMapFHat \KloostermanGlobalClassFunction$ equals \begin{equation*}
		\prod_{\left[\beta\right]} \prod_{i=1}^{\infty} \prod_{j=1}^{\infty} \left(1 - \left(\left(-1\right)^{s} q^{-\left(\frac{k-1}{2} + j\right)}\right)^{\deg \beta} \GaussSumCharacter{\lambda, \deg \beta}{\alpha}{\beta^{-1}}{\fieldCharacter} X_i^{\left[\beta\right]}\right)^{-1}.
	\end{equation*}
\end{proof}
We are now ready to prove \Cref{thm:general-formula-for-exotic-kloosterman-sum-of-conjugacy-class}.
\begin{proof}
	Recall that for any $c \ge 1$, the assignment $\GL_c\left(\finiteField\right) \to \cComplex$ given by $h \mapsto \ExoticKloostermanNormalized\left(\alpha, \fieldCharacter, h\right)$ is the unique class function such that for any irreducible representation $\pi$ of $\GL_c\left(\finiteField\right)$, the following equality holds:
	$$\TwistedGaussSum{\pi}{\alpha}{\fieldCharacter} = q^{-\frac{c^2}{2}} \sum_{h \in \GL_c\left(\finiteField\right)} \ExoticKloostermanNormalized\left(\alpha, \fieldCharacter, h\right) \pi\left(h\right).$$
	Since $\KloostermanGlobalClassFunction$ is a class function by definition, it suffices to prove that for any $c \ge 1$ and any irreducible representation $\pi$ of $\GL_c\left(\finiteField\right)$, $$\TwistedGaussSum{\pi}{\alpha}{\fieldCharacter} = q^{-\frac{c^2}{2}} \sum_{h \in \GL_c\left(\finiteField\right)} \KloostermanGlobalClassFunction\left(h\right) \pi\left(h\right).$$
	This is equivalent to showing that \begin{equation}\label{eq:inner-product-statement-for-hall-littlewood-theorem}
		\innerproduct{\KloostermanGlobalClassFunction}{\trace \pi^{\vee}} = \frac{q^{\frac{c^2}{2}}}{\sizeof{\GL_c\left(\finiteField\right)}} \cdot \dim \pi \cdot \GKGaussSumScalar{\pi}{\alpha}{\fieldCharacter}.
	\end{equation}
	
	Recall the Cauchy identity
	\begin{equation}\label{eq:cauchy-identity}
		\prod_{i,j \ge 1} \frac{1}{1 - X_i Y_j} = \sum_{\mu} s_{\mu}\left(X\right) s_{\mu}\left(Y\right),
	\end{equation}
	where $\mu$ runs over all the partitions.
	
	It follows from \Cref{thm:global-matrix-kloosterman-sum-in-character-basis} combined with the Cauchy identity and the fact that $s_{\mu}$ is homogeneous of degree $\sizeof{\mu}$ that
	$$\CharacteristicMapFHat \KloostermanGlobalClassFunction = \prod_{\left[\beta\right]} \left(\sum_{\mu} \left(\left(-1\right)^{s} q^{-\frac{k-1}{2}}\right)^{\sizeof{\mu}\deg \beta} \GaussSumCharacter{\lambda, \deg \beta}{\alpha}{\beta^{-1}}{\fieldCharacter}^{\sizeof{\mu}} s_{\mu}^{\left[\beta\right]}\left(X\right) s_{\mu}^{\left[\beta\right]}\left(Y\right)\right),$$
	where $Y^{\left[\beta\right]}_j = q^{-j \deg \beta}$ for $\left[\beta\right]$ and any $j \ge 1$.
	Expanding the sum, we see that $\CharacteristicMapFHat \KloostermanGlobalClassFunction$ is given by
	$$\sum_{r=0}^{\infty} \sum_{\left\{\left(\left[\beta_1\right], \mu_1\right), \dots, \left(\left[\beta_r\right], \mu_r\right)\right\}} \prod_{j=1}^r \left(\left(\left(-1\right)^{s} q^{-\frac{k-1}{2}}\right)^{\deg \beta_j} \GaussSumCharacter{\lambda, \deg \beta_j}{\alpha}{\beta_j^{-1}}{\fieldCharacter}\right)^{\sizeof{\mu_j}} s_{\mu_j}^{\left[\beta_j\right]}\left(X\right) s_{\mu_j}^{\left[\beta_j\right]}\left(Y\right),$$
	where the sum is over all sets $\left\{ \left(\left[\beta_1\right], \mu_1\right), \dots, \left(\left[\beta_r\right], \mu_r\right) \right\},$
	where $\left[\beta_1\right]$, $\dots$, $\left[\beta_r\right]$ are different Frobenius orbits of $\Gamma$ and $\mu_1$, $\dots$, $\mu_r$ are non-empty partitions.
	
	If $\pi$ is an irreducible representation of $\GL_c\left(\finiteField\right)$ corresponding to a Macdonald parameter $\phi$, then $\pi^{\vee}$ corresponds to the Macdonald parameter $\phi^{\vee}$ given by $$\phi^{\vee}\left(\theta\right) = \phi^{\vee}\left(\theta^{-1}\right)$$
	for any $\theta \in \Gamma$. Suppose that $\phi$ is supported on the different Frobenius orbits $\left[\beta_1\right]$, $\dots$, $\left[\beta_r\right]$ and that $\phi\left(\left[\beta_j\right]\right) = \mu_j$ and $\deg \beta_j = c_j$ for any $j$. Then $$\CharacteristicMapFHat \trace \pi^{\vee} = \prod_{j=1}^r s_{\mu_j}^{\left[\beta_j^{-1}\right]}\left(X\right).$$
	Since $\CharacteristicMapFHat$ preserves the inner product and since $\left(s_{\lambda}^{\left[\beta\right]}\right)_{\lambda}$ is an orthonormal basis for every $\left[\beta\right]$, we have
	$$\innerproduct{\KloostermanGlobalClassFunction}{\trace \pi^{\vee}} = \prod_{j=1}^r \left(\left(-1\right)^{s} q^{-\frac{k-1}{2}}\right)^{c_j \sizeof{\mu_j}} \GaussSumCharacter{\lambda, c_j}{\alpha}{\beta_j}{\fieldCharacter}^{\sizeof{\mu_j}} s_{\mu_j}^{\left[\beta_j^{-1}\right]}\left(Y\right).$$
	Recall that $\sum_{j=1}^r c_j \sizeof{\mu_j} = c$. Hence \begin{equation}\label{eq:inner-product-of-kloosterman-candidiate-and-trace-pi}
		\innerproduct{\KloostermanGlobalClassFunction}{\trace \pi^{\vee}} = \left(\left(-1\right)^{sc} q^{-\frac{kc}{2}} \prod_{j=1}^r \GaussSumCharacter{\lambda, c_j}{\alpha}{\beta_j}{\fieldCharacter}^{\sizeof{\mu_j}}\right) \cdot q^{\frac{c}{2}}  \cdot \prod_{j=1}^r s_{\mu_j}^{\left[\beta_j^{-1}\right]}\left(Y\right).
	\end{equation}
	Recall that $\sizeof{\mu_j}$ is the number of times that the cuspidal representation corresponding to the Frobenius orbit $\left[\beta_j\right]$ appears in the cuspidal support of $\pi$. It follows from \Cref{thm:exotic-gauss-sum-of-composite-character} that $$\innerproduct{\KloostermanGlobalClassFunction}{\trace \pi^{\vee}} = \GKGaussSumScalar{\pi}{\alpha}{\fieldCharacter} \cdot q^{\frac{c}{2}} \cdot \prod_{j=1}^r s_{\mu_j}\left(q^{-c_j}, q^{-2c_j}, q^{-3c_j}\dots \right).$$
	By~\cite[Page 286]{macdonald1998symmetric}, $$\dim \pi = \Phi_c\left(q\right) \cdot \prod_{j=1}^r s_{\mu_j}\left(q^{-c_j}, q^{-2c_j}, q^{-3c_j}\dots \right),$$ where
	$$\Phi_c\left(q\right) = \prod_{j=1}^c \left(q^j - 1\right) = q^{-\frac{c^2}{2} + \frac{c}{2}} \cdot \sizeof{\GL_c\left(\finiteField\right)}.$$
	Hence we showed that \eqref{eq:inner-product-statement-for-hall-littlewood-theorem} holds, as required.
\end{proof}
\begin{remark}
	When $k=1$, this proof can be modified to give an independent proof of Kondo's theorem. One needs to check that in this case $\KloostermanGlobalClassFunction$ is the function $\KloostermanGlobalClassFunction\left(h\right) = \alpha\left(\det h\right)\fieldCharacter\left(\trace h\right)$, and then Kondo's theorem follows eventually from \eqref{eq:inner-product-of-kloosterman-candidiate-and-trace-pi} and from the dimension formula discussed at the end of the proof. Such proof was given by Macdonald in~\cite[Page 289]{macdonald1998symmetric}.
\end{remark}

\section{Applications}\label{sec:applications}

\subsection{Identities for Bessel functions}\label{sec:identities-for-bessel-functions}
In this section, we use our results to deduce identities for exotic matrix Kloosterman sums and special values of the Bessel function.

We start by recalling~\cite[Theorem 2.14]{zelingher2022values}.
\begin{theorem}\label{thm:recursive-bessel-function-theorem}
	Let $\tau$ be an irreducible generic representation of $\GL_k\left(\finiteField\right)$. Let $h \in \GL_c\left(\finiteField\right)$. Denote $$\mathcal{F}_{\tau,c,\fieldCharacter}\left(h\right) = \frac{q^{-\frac{c \left(k-c-1\right)}{2}}}{\grpIndex{\GL_c\left(\finiteField\right)}{\UnipotentSubgroup_c}} \sum_{\pi} \dim \pi \cdot  \centralCharacter{\pi}\left(-1\right)^{k-1} \varepsilon_0\left( \Contragradient{\pi} \times \tau, \fieldCharacter\right) \cdot \besselFunction_{\pi, \fieldCharacter}\left(h\right),$$
	where $\pi$ runs over all the (equivalence classes of) irreducible generic representations of $\GL_c\left(\finiteField\right)$ and $\besselFunction_{\pi, \fieldCharacter}$ is the Bessel function of $\pi$ with respect to $\fieldCharacter$ (see \Cref{subsec:bessel-functions}). Then
	\begin{enumerate}
		\item If $c < k$,
		$$\mathcal{F}_{\tau,c,\fieldCharacter}\left(h\right) = \besselFunction_{\tau, \fieldCharacter}\begin{pmatrix}
			& \IdentityMatrix{k-c}\\
			h
		\end{pmatrix}.$$
		\item If $c = k$,
		$$ \mathcal{F}_{\tau,c,\fieldCharacter}\left(h\right) = \besselFunction_{\tau, \fieldCharacter}\left(h\right)\fieldCharacter\begin{pmatrix}
			\IdentityMatrix{k} & h^{-1}\\
			& \IdentityMatrix{k}
		\end{pmatrix}.$$
		\item If $c > k$ then $\mathcal{F}_{\tau,c,\fieldCharacter}\left(h\right) = 0$ unless $h = u_1 \left(\begin{smallmatrix}
			& -\IdentityMatrix{c-k}\\
			h
		\end{smallmatrix}\right) u_2$ for some $u_1, u_2 \in \UnipotentSubgroup_c$, in which case
		$$ \mathcal{F}_{\tau,c,\fieldCharacter}\left(h\right) = q^{c^2-k^2-\binom{c-k}{2}} \fieldCharacter\left(u_1 u_2\right) \besselFunction_{\tau, \fieldCharacter}\left(h\right).$$
	\end{enumerate}
\end{theorem}

We have the following simple proposition.
\begin{proposition}\label{prop:expression-of-special-value-of-bessel-speh-as-fourier-transform}
	Let $\tau$ be an irreducible generic representation of $\GL_k\left(\finiteField\right)$. Then for any $h \in \GL_c\left(\finiteField\right)$,
	$$ \specialBesselSpeh{\tau}\left(h\right) = \frac{q^{-\frac{\left(k-2\right)c^2}{2}}}{\sizeof{\GL_c\left(\finiteField\right)}} \sum_{\pi \in \Irr\left(\GL_c\left(\finiteField\right)\right)} \dim \pi \cdot \GKPreGammaFactor{\pi}{\tau}{\fieldCharacter} \cdot \trace \pi\left(h^{-1}\right).$$
\end{proposition}
\begin{proof}
	By taking the trace of $\GKGaussSum{\pi}{\tau}{\fieldCharacter}$, we have \begin{equation}\label{eq:fourier-transform-of-bessel-speh}
		\dim \pi \cdot \GKPreGammaFactor{\pi}{\tau}{\fieldCharacter} = q^{\frac{\left(k-2\right) c^2}{2}} \sum_{x \in \GL_c\left(\finiteField\right)} \specialBesselSpeh{\tau}\left(x\right) \trace \pi\left(x\right).
	\end{equation}
	Recall that the assignment $h \mapsto \specialBesselSpeh{\tau}\left(h\right)$ is a class function and that $\left(\trace \pi\right)_{\pi}$ is an orthonormal basis of the space of class functions $\ClassFunctionsRing\left(\GL_c\left(\finiteField\right)\right)$. Using $\trace \pi^{\vee} = \conjugate{\trace \pi}$, we see that \eqref{eq:fourier-transform-of-bessel-speh} gives an expression for the inner product $\innerproduct{\specialBesselSpeh{\tau}}{\trace \pi^{\vee}}$. The result follows since $\trace \pi^{\vee}\left(h\right) = \trace \pi\left(h^{-1}\right)$ for $h \in \GL_c\left(\finiteField\right)$.
\end{proof}

Our results yield the following formula, which allows one to express the Bessel function as a Whittaker transform of an exotic matrix Kloosterman sum.
\begin{theorem}\label{thm:whittaker-transform-of-exotic-matrix-kloosterman-sum}
	Let $\tau$ be an irreducible generic representation of $\GL_k\left(\finiteField\right)$. Let $h \in \GL_c\left(\finiteField\right)$. Then
	$$\mathcal{F}_{\tau,c,\fieldCharacter}\left(h\right) = q^{\left(k-1\right) \binom{c}{2}} \sum_{u \in \UnipotentSubgroup_c} \specialBesselSpeh{\tau}\left(hu\right) \fieldCharacter^{-1}\left(u\right).$$
\end{theorem}
\begin{proof}
	By \Cref{prop:expression-of-special-value-of-bessel-speh-as-fourier-transform}, $$ \specialBesselSpeh{\tau}\left(h\right) = \frac{q^{-\frac{\left(k-2\right)c^2}{2}}}{\sizeof{\GL_c\left(\finiteField\right)}} \sum_{\pi \in \Irr\left(\GL_c\left(\finiteField\right)\right)} \dim \pi \cdot \GKPreGammaFactor{\Contragradient{\pi}}{\tau}{\fieldCharacter} \cdot \trace \pi\left(h\right).$$
	Replacing $h$ with $hu$, averaging with $\fieldCharacter^{-1}\left(u\right)$ over $u \in \UnipotentSubgroup_c$, and using the identity $$ \sum_{u \in \UnipotentSubgroup_c} \trace \pi\left(hu\right) \fieldCharacter^{-1}\left(u\right) = \begin{cases}
		\sizeof{\UnipotentSubgroup_c} \besselFunction_{\pi, \fieldCharacter}\left(h\right) & \pi \text{ is generic,}\\
		0 & \text{otherwise},
	\end{cases} $$
	and the fact that $$\GKPreGammaFactor{\Contragradient{\pi}}{\tau}{\fieldCharacter} = \centralCharacter{\pi}\left(-1\right)^{k-1} \GKGammaFactor{\Contragradient{\pi}}{\tau}{\fieldCharacter} = \centralCharacter{\pi}\left(-1\right)^{k-1} \varepsilon_0\left(\Contragradient{\pi} \times \tau, \fieldCharacter\right),$$
	we get the result.
\end{proof}

\begin{example}\label{example:curtis-shinoda-trick}
	We show an application of \Cref{thm:whittaker-transform-of-exotic-matrix-kloosterman-sum} and the other results of the paper. Let $2 \le c \le k - 1$. We use \Cref{thm:whittaker-transform-of-exotic-matrix-kloosterman-sum} to find an expression for \begin{equation}\label{eq:bessel-function-of-triple}
		\besselFunction_{\pi, \fieldCharacter}\begin{pmatrix}
			& & \IdentityMatrix{k-c}\\
			& t_2 \IdentityMatrix{c-1}\\
			t_1
		\end{pmatrix} 
	\end{equation}
	where $\pi$ is an irreducible generic representation of $\GL_k\left(\finiteField\right)$ and where $t_1, t_2 \in \multiplicativegroup{\finiteField}$. By \Cref{thm:recursive-bessel-function-theorem} combined with Theorems \ref{thm:whittaker-transform-of-exotic-matrix-kloosterman-sum} and \ref{thm:bessel-speh-is-an-exotic-kloosterman-sum-generic}, the value \eqref{eq:bessel-function-of-triple} is given by
	$$\left(-1\right)^{\left(k+s\right) c} q^{-\frac{c\left(k-1\right)}{2}} \sum_{u \in \UnipotentSubgroup_{c}} \ExoticKloosterman^{\ast}\left(\alpha^{-1}, \fieldCharacter, \left(-1\right)^{k-1} u^{-1} \begin{pmatrix}
		& t_2 \IdentityMatrix{c-1}\\
		t_1
	\end{pmatrix}^{-1}\right) \fieldCharacter^{-1}\left(u\right),$$
	where $\alpha$ is as in \Cref{thm:bessel-speh-is-an-exotic-kloosterman-sum-generic}.
	By~\cite[Lemma 3.1]{curtis2004zeta}, for any $t \in \multiplicativegroup{\finiteField}$ and any $u \in \UnipotentSubgroup_{c}$, the matrix \begin{equation}\label{eq:voronoi-element-time-unipotent}
		u \begin{pmatrix}
			& \IdentityMatrix{c-1}\\
			t
		\end{pmatrix}^{-1}
	\end{equation} is a regular element of $\GL_c\left(\finiteField\right)$ and thus its conjugacy class is determined by its characteristic polynomial. Moreover by~\cite[Proof of Theorem 3.3]{curtis2004zeta}, for any regular element $h$ with determinant $\left(-1\right)^{c-1} t^{-1}$ there are exactly $q^{\binom{c-1}{2}}$ options for $u \in \UnipotentSubgroup_c$ such that \eqref{eq:voronoi-element-time-unipotent} is conjugate to $h$. 
	
	We get from these results that for any $u \in \UnipotentSubgroup_c$ the element \begin{equation}\label{eq:twisted-voronoi-element-time-unipotent}
		\left(-1\right)^{k-1} u^{-1} \begin{pmatrix}
			& t_2 \IdentityMatrix{c-1}\\
			t_1
		\end{pmatrix}^{-1}
	\end{equation} is a regular element, and that for any regular element $h \in \GL_c\left(\finiteField\right)$ with determinant $\left(-1\right)^{ck-1} t_1^{-1} t_2^{-\left(c-1\right)}$ there are exactly $q^{\binom{c-1}{2}}$ options for $u$ such that \eqref{eq:twisted-voronoi-element-time-unipotent} is conjugate to $h$. Moreover, it is easy to check that if $u$ is of the form \begin{equation*}
	u = \begin{pmatrix}
		1 & x_1 & \ast & \ast  & \ast \\
		& 1 & x_2 & \ast & \ast \\
		& & \ddots & \ddots & \ast  \\
		& & & 1 &  x_{c-1} \\
		& & & & 1
	\end{pmatrix}
	\end{equation*} then $$-\sum_{i=1}^{c-1} x_i = \left(-1\right)^{k-1} t_2 \trace \left( \left(-1\right)^{k-1} u^{-1} \left(\begin{smallmatrix}
	& t_2\IdentityMatrix{c-1}\\
	t_1
	\end{smallmatrix}\right)^{-1}\right).$$
	
	Recall that an \emph{effective $\finiteField$ zero-cycle of degree $d$} is a $\zIntegers$-linear combination of the form
	$$\mathfrak{c} = \sum_{\left[\xi\right]} m_{\left[\xi\right]} \left[\xi\right],$$
	where $\left[\xi\right]$ runs over all Frobenius orbits in $\multiplicativegroup{\algebraicClosure{\finiteField}}$ and where $m_{\left[\xi\right]}$ is a non-negative integer for any $\xi$ such that $$d = \deg \mathfrak{c} \coloneq \sum_{\left[\xi\right]} m_{\left[\xi\right]} \deg \left[\xi\right].$$ We define $$\trace \mathfrak{c} = \sum_{\left[\xi\right]} m_{\left[\xi\right]} \trace \left[\xi\right] \in \finiteField$$
	and $$\aFieldNorm\left(\mathfrak{c}\right) = \prod_{\left[\xi\right]} \FieldNorm{\deg \left[\xi\right]}{1}\left(\left[\xi\right]\right)^{m_{\left[\xi\right]}} \in \multiplicativegroup{\finiteField}.$$
	
	Given a zero-cycle $\mathfrak{c}$ as above, we define $$\ExoticKloosterman^{\ast}\left(\alpha^{-1}, \fieldCharacter, \mathfrak{c}\right) = \left(-1\right)^{\left(k-1\right) \deg \mathfrak{c}} q^{-\frac{\left(k-1\right) \deg \mathfrak{c}}{2}} \prod_{\left[\xi\right]} \trace \left(\Frobenius \mid_{\xi}, \SymmetricPower^{m_{\left[\xi\right]}} \ExoticKloosterman_{\finiteFieldExtension{\deg \xi}}\left(\finiteFieldExtension{\lambda}, \alpha^{-1}, \fieldCharacter\right)_{\xi}\right).$$
	
	Since regular conjugacy classes of $\GL_c\left(\finiteField\right)$ are in bijection with effective $\finiteField$ zero-cycles of degree $c$, it follows from the work above combined with \Cref{cor:formula-for-regular-elements} that
	\begin{align*}
		&q^{\frac{1}{2}\left( \left(c-1\right) + \left(k - c\right) + \left(c-1\right)\left(k-c\right) \right)} \besselFunction_{\pi, \fieldCharacter}\begin{pmatrix}
			& & \IdentityMatrix{k-c}\\
			& t_2 \IdentityMatrix{c-1}\\
			t_1
		\end{pmatrix} \\
		=&\left(-1\right)^{\left(k+s\right)c}  q^{-\frac{c-1}{2}} \sum_{\substack{\mathfrak{c}\\
				\aFieldNorm\left(\mathfrak{c}\right) = \left(-1\right)^{ck-1} t_1^{-1} t_2^{-\left(c-1\right)}}} \ExoticKloosterman^{\ast}\left(\alpha^{-1}, \fieldCharacter, \mathfrak{c}\right) \fieldCharacter\left(\left(-1\right)^{k-1} t_2 \trace \mathfrak{c}\right),
	\end{align*}
	where the sum is over all effective $\finiteField$ zero-cycles $\mathfrak{c}$ of degree $c$. The coefficient of the Bessel function on the left hand side is the square root of the size of the Bruhat cell corresponding to the present argument of the Bessel function. This formula shows that these values of the Bessel function are more involved and have more knowledge about the corresponding exotic Kloosterman sheaf.
\end{example}

Using the results of~\cite{zelingher2022values}, we obtain the following identity.
\begin{theorem}\label{thm:generating-series-identity}
	Let $\tau$ be an irreducible generic representation of $\GL_k\left(\finiteField\right)$. Then for any $x \in \multiplicativegroup{\finiteField}$, we have the following identity of generating functions \begin{align*}
		\left(\sum_{r = 0}^k q^{\frac{r\left(k-r\right)}{2}} \besselFunctionOfFiniteFieldRepresentation{\tau}\begin{pmatrix}
			& \IdentityMatrix{k - r}\\
			x \IdentityMatrix{r} &
		\end{pmatrix} T^r\right)^{-1} &= 1 + \sum_{r = 1}^{\infty} \left(-1\right)^{r} q^{\frac{(k-1)}{2}r^2} \specialBesselSpeh{\tau}\left(J_{\left(r\right)}\left(x\right)\right) T^r,
	\end{align*}
	where $J_{\left(r\right)}\left(x\right)$ is the Jordan block of size $r$ corresponding to $x$ (see \Cref{subsec:conjugacy-classes-of-gl-n}).
\end{theorem}
\begin{proof}
	We use the notations of \Cref{thm:explicit-formula-for-special-bessel-speh-function-as-hall-littlewood-polynomial}. Let $\omega_1, \dots, \omega_k$ be the roots of $\ExoticKloosterman_{\finiteField}\left(\alpha, \fieldCharacter\right)$ at $\left(-1\right)^{k-1} x^{-1}$.
	By~\cite[Theorem 4.7]{zelingher2022values}, \begin{equation}\label{eq:bessel-function-special-value-as-elementary-symmetric-polynomial}
		\left(-1\right)^{sr} q^{-\frac{\left(k-1\right)r}{2}} e_r \left(\omega_1,\dots,\omega_k\right) = \begin{dcases}
			\left(-1\right)^r q^{\frac{r\left(k-r\right)}{2}} \besselFunctionOfFiniteFieldRepresentation{\tau}\begin{pmatrix}
				& \IdentityMatrix{k - r}\\
				x \IdentityMatrix{r} &
			\end{pmatrix} & 0 \le r \le k\\
			0 & r > k
		\end{dcases},
	\end{equation}
	where $$ e_r\left(\omega_1,\dots,\omega_k\right) =  \sum_{1 \le i_1 < \dots < i_r \le k} \omega_{i_1} \dots \omega_{i_r} $$
	is the elementary symmetric polynomial of degree $r$.
	By \Cref{thm:explicit-formula-for-special-bessel-speh-function-as-hall-littlewood-polynomial} and \Cref{example:modified-hall-littlewood-for-k-is-h_k}, we have
	$$ q^{\frac{\left(k-1\right)r^2}{2}} \specialBesselSpeh{\tau}\left(J_{\left(r\right)}\left(x\right)\right) = \left(-1\right)^{\left(s-1\right)r} q^{-\frac{\left(k-1\right)r}{2}} h_r\left(\omega_1,\dots,\omega_k\right),$$
	where $$h_r\left(\omega_1,\dots,\omega_k\right) = \sum_{\substack{i_1, \dots, i_k \ge 0\\
			i_1 + \dots + i_k = r}} \omega_1^{i_1} \dots \omega_k^{i_k}$$ is the complete homogeneous symmetric polynomial of degree $r$.
	The result now follows from the identity $$\left(\sum_{r=0}^k \left(-1\right)^r e_r\left(\omega_1,\dots,\omega_k\right) S^r\right)^{-1} = \prod_{i=1}^{k} \frac{1}{1 - \omega_i S} = \sum_{r=0}^{\infty} h_r\left(\omega_1,\dots,\omega_k\right) S^r,$$
	with $S = \left(-1\right)^{s-1} q^{-\frac{k-1}{2}} T$.
\end{proof}
As an immediate corollary of the previous theorem and of~\cite[Theorem 4.7]{zelingher2022values}, we obtain the following result.
\begin{corollary}\label{cor:l-function-identity}
	Let $\tau$ be the unique irreducible generic representation of $\GL_k\left(\finiteField\right)$ with cuspidal support $\left\{\tau_1,\dots,\tau_s\right\}$. Then
	$$L^{\ast}\left(\left(-1\right)^s T, \ExoticKloosterman_{\finiteField}\left(\alpha^{-1}, \fieldCharacter, \left(-1\right)^{k-1} x^{-1}\right) \right)^{-1} = 1 + \sum_{c = 1}^{\infty} \left(-1\right)^{c} \specialBesselSpehNormalized{\tau}\left(J_{\left(c\right)}\left(x\right)\right) T^c,$$
	where $L^{\ast}$ is the normalized L-function described in \Cref{subsec:roots-of-kloosterman-sum-at-xi}.
\end{corollary}
\begin{remark}
	The $L$-factor in \Cref{cor:l-function-identity} is closely related to a standard local $L$-factor of an automorphic representation of $\GL_k$ over a function field, see \Cref{rem:kloosterman-eigenvalues-are-satake-parameters-for-a-function-field} for more details.
\end{remark}

\subsection{Inequalities}\label{subsec:inequalities}

Similarly to~\cite[Section 4.3.4]{Zelingher2023}, using the fact that the normalized roots of $\ExoticKloosterman\left(\alpha, \fieldCharacter\right)$ at $\xi$ all have absolute value $1$ (\Cref{subsec:roots-of-kloosterman-sum-at-xi}), we obtain from Theorems \ref{thm:general-formula-for-exotic-kloosterman-sum-of-conjugacy-class} and \ref{thm:explicit-formula-for-special-bessel-speh-function-as-hall-littlewood-polynomial} and from the formula in \Cref{subsec:flag-formula-for-modified-hall-littlewood-polynomials} the following inequalities.

\begin{corollary}
	Using the notations of \Cref{thm:explicit-formula-for-special-bessel-speh-function-as-hall-littlewood-polynomial}, we have $$\abs{\specialBesselSpehNormalized{\tau}\left(h\right)} = \abs{\ExoticKloostermanNormalized\left(\alpha, \fieldCharacter, h\right)} \le \prod_{i=1}^r \#\left\{\mathcal{F} \text{ is a weak flag in } \finiteFieldExtension{a_i}^{b_i} \text{ of length } k \mid J_{\mu_i}\left(1\right) \mathcal{F} = \mathcal{F}\right\}.$$
	In the special case where $\mu_i = \left(b_i\right)$ for every $i$, i.e., where $h$ is a regular element, we have
	$$\abs{\specialBesselSpehNormalized{\tau}\left(h\right)} = \abs{\ExoticKloostermanNormalized\left(\alpha, \fieldCharacter, h\right)} \le \prod_{i=1}^r \binom{b_i + k - 1}{b_i}.$$
\end{corollary}

\subsection{Comparison with spherical vector formulas}\label{subsec:formulas-for-spherical-elements}

We mention a similarity between formulas for special values of Bessel--Speh functions and formulas for special values of spherical vectors in $\kcNotation{k}{c}{\fieldCharacter}$ models associated to unramified Speh representations.

Let $\localField$ be a non-archimedean local field with ring of integers $\ringOfIntegers \subset \localField$ and maximal idea $\maximalIdeal$ and residue field $\finiteField$. Let $\uniformizer \in \maximalIdeal \setminus \maximalIdeal^2$ be a uniformizer, and normalize $\abs{\cdot}$ such that $\abs{\uniformizer} = q^{-1}$. Consider the (normalized) parabolically induced representation $$\depthZeroRepresentation^{\left(z_1,\dots,z_k\right)} \coloneq \depthZeroRepresentation_{1,1}^{\left(z_1,\dots,z_k\right)} = \Ind{\ParabolicSubgroup_{\left(1^k\right)}\left(\localField\right)}{\GL_k\left(\localField\right)}{\abs{\cdot}^{z_1} \boxtimes \dots \boxtimes \abs{\cdot}^{z_k}}.$$ This is the unramified representation of $\GL_k\left(\localField\right)$ with Satake parameters $q^{-z_1}, \dots, q^{-z_k}$, where $z_1,\dots,z_k \in \cComplex$. For $c \ge 1$, let $$\SpehRepresentation{\depthZeroRepresentation^{\left(z_1,\dots,z_k\right)}}{c} = \Ind{\ParabolicSubgroup_{\left(c^k\right)}\left(\localField\right)}{\GL_{kc}\left(\localField\right)}{\abs{{\det}_{\GL_c}}^{z_1} \boxtimes \dots \boxtimes \abs{{\det}_{\GL_c}}^{z_k}}.$$
Note that $\SpehRepresentation{\depthZeroRepresentation^{\left(z_1,\dots,z_k\right)}}{1} = \depthZeroRepresentation^{\left(z_1,\dots,z_k\right)}$. For any $c \ge 1$, the representation $\SpehRepresentation{\depthZeroRepresentation^{\left(z_1,\dots,z_k\right)}}{c}$ admits a unique (up to scalar multiplication) vector which is invariant under the action of the maximal compact subgroup $\GL_{kc}\left(\ringOfIntegers\right)$.

Let $\fieldCharacter \colon \localField \to \multiplicativegroup{\cComplex}$ be a non-trivial character with conductor $\ringOfIntegers$, i.e., $\fieldCharacter$ is trivial on $\ringOfIntegers$ but not on $\uniformizer^{-1} \ringOfIntegers$. For $\left(z_1,\dots,z_k\right)$ in generic position (depending on $c$), the representations $\depthZeroRepresentation^{\left(z_1,\dots,z_k\right)}$ and $\SpehRepresentation{\depthZeroRepresentation^{\left(z_1,\dots,z_k\right)}}{c}$ are irreducible. See~\cite[Section 2.4]{Zelingher2025} for more details. In this case we say that $\SpehRepresentation{\depthZeroRepresentation^{\left(z_1,\dots,z_k\right)}}{c}$ is an \emph{unramified Speh representation}~\cite{CaiFriedbergGourevitchKaplan2023, LapidMao2020}. For such choice of $\left(z_1,\dots,z_k\right)$, the representation $\depthZeroRepresentation^{\left(z_1,\dots,z_k\right)}$ is generic and the representation $\SpehRepresentation{\depthZeroRepresentation^{\left(z_1,\dots,z_k\right)}}{c}$ is a $\kcNotation{k}{c}{\fieldCharacter}$ representation~\cite{CaiFriedbergGourevitchKaplan2023}. In particular, this means that $$\dim \Hom_{\UnipotentRadicalForWss{k}{c}\left(\localField\right)}\left(\SpehRepresentation{\depthZeroRepresentation^{\left(z_1,\dots,z_k\right)}}{c}, \fieldCharacterkc{k}{c}\right) = 1,$$
where $\UnipotentRadicalForWss{k}{c}\left(\localField\right)$ and $\fieldCharacterkc{k}{c}$ are defined analogously to the definition in the finite field case. Thus from Frobenius reciprocity there exists a unique subspace of $\Ind{\UnipotentRadicalForWss{k}{c}}{\GL_{kc}\left(\localField\right)}{\fieldCharacterkc{k}{c}}$ which is isomorphic to $\SpehRepresentation{\depthZeroRepresentation^{\left(z_1,\dots,z_k\right)}}{c}$. We denote this subspace by $\Whittaker\left(\SpehRepresentation{\depthZeroRepresentation^{\left(z_1,\dots,z_k\right)}}{c},\fieldCharacterkc{k}{c}\right)$ and call it the \emph{$\kcNotation{k}{c}{\fieldCharacter}$ model} of $\depthZeroRepresentation^{\left(z_1,\dots,z_k\right)}$. This space contains a unique element $W_{c,\circ}^{\left(z_1,\dots,z_k\right)}$ that is invariant under the action of $\GL_{kc}\left(\ringOfIntegers\right)$ and such that $W_{c,\circ}^{\left(z_1,\dots,z_k\right)}\left(\IdentityMatrix{kc}\right) = 1$. In~\cite{Zelingher2025} we studied the assignment $\mathcal{S}^{\circ}_{{\left(z_1,\dots,z_k\right)}, \fieldCharacter} \colon \GL_c\left(\localField\right) \to \cComplex$ given by $$\mathcal{S}^{\circ}_{{\left(z_1,\dots,z_k\right)},c,\fieldCharacter}\left(h\right) = W_{c,\circ}^{\left(z_1,\dots,z_k\right)}\begin{pmatrix}
h\\
& \IdentityMatrix{\left(k-1\right)c}
\end{pmatrix}.$$
This assignment is the local field analog of $\specialBesselSpeh{\tau}$. We have that $\mathcal{S}^{\circ}_{{\left(z_1,\dots,z_k\right)},c,\fieldCharacter}$ is bi-invariant under $\GL_c\left(\ringOfIntegers\right)$ translations. In~\cite[Theorem 3.4]{Zelingher2025} we showed that for $\mu = \left(\mu_1,\dots,\mu_{\ell}\right)$ a sequence of weakly decreasing integers with $\ell \le c$,
$$\delta^{-1 \slash 2}_{\ParabolicForSpeh{c}{k}}\left(\uniformizer^{\mu, kc}\right) \mathcal{S}^{\circ}_{{\left(z_1,\dots,z_k\right)},c,\fieldCharacter}\left(\uniformizer^{\mu,c}\right) = \begin{dcases}
	0 & \mu_{\ell} < 0,\\
	\htHallLittlewood_{\mu}\left(q^{-z_1},\dots,q^{-z_k};q\right) & \mu_{\ell} \ge 1.
\end{dcases}$$
Here $\uniformizer^{\mu, c} = \diag\left(\uniformizer^{\mu_1},\dots,\uniformizer^{\mu_{\ell}},\IdentityMatrix{c-\ell}\right)$ and $\delta^{-1 \slash 2}_{\ParabolicForSpeh{c}{k}}\left(\uniformizer^{\mu, kc}\right) = q^{\frac{c \left(k-1\right) \sizeof{\mu}}{2}}$ where $\sizeof{\mu} = \mu_1 + \dots + \mu_{\ell}$.

We see a resemblance between this formula and the formula obtained in \Cref{thm:explicit-formula-for-special-bessel-speh-function-as-hall-littlewood-polynomial}, where the Satake parameters $q^{-z_1},\dots,q^{-z_k}$ are replaced with $\left(-1\right)^{\left(s-1\right)a} \omega^{\ast}_{1}, \dots, \left(-1\right)^{\left(s-1\right)a} \omega^{\ast}_{k}$, where $\omega^{\ast}_{1}$, $\dots$, $\omega^{\ast}_{k}$ are the normalized roots of $\ExoticKloosterman_{\finiteFieldExtension{a}}\left(\alpha^{-1}, \fieldCharacter\right)$ at $\left(-1\right)^{k-1} \xi^{-1} \in \multiplicativegroup{\finiteFieldExtension{a}}$. In fact, the results in the present paper and our previous results with Carmon~\cite{CarmonZelingher2025} led us to find the formula in~\cite{Zelingher2025}.

A similar resemblance occurs with special values of Bessel functions. The equality \eqref{eq:bessel-function-special-value-as-elementary-symmetric-polynomial} (given by~\cite[Theorem 4.7]{zelingher2022values}) can be rewritten as
$$q^{\frac{r\left(k-r\right)}{2}} \besselFunctionOfFiniteFieldRepresentation{\tau}\begin{pmatrix}
	& \IdentityMatrix{k - r}\\
	x \IdentityMatrix{r} &
\end{pmatrix} = e_r \left(\left(-1\right)^{s-1} \omega_1^{\ast},\dots,\left(-1\right)^{s-1} \omega^{\ast}_k\right),$$ where $r \le k$ and $\omega^{\ast}_{1}$, $\dots$, $\omega^{\ast}_{k}$ are the normalized roots of $\ExoticKloosterman_{\finiteField}\left(\alpha^{-1}, \fieldCharacter\right)$ at $\left(-1\right)^{k-1} x^{-1} \in \multiplicativegroup{\finiteField}$. This is similar to the following special case of the Shintani--Casselman--Shalika formula~\cite{shintani1976explicit, CasselmanShalika1980}
$$q^{\frac{r\left(k-r\right)}{2}} W_{1,\circ}^{\left(z_1,\dots,z_k\right)}\begin{pmatrix}
	\uniformizer \IdentityMatrix{r}\\
	& \IdentityMatrix{k-r} \end{pmatrix} = e_r\left(q^{-z_1},\dots,q^{-z_k}\right).$$

These resemblances suggest that there is a deep connection between unramified principal series representations of $\GL_n\left(\localField\right)$ and generic representations of $\GL_n\left(\finiteField\right)$, where the Satake parameters match corresponding the roots of $\ExoticKloosterman\left(\alpha^{-1}, \fieldCharacter\right)$. It would be very interesting to make this connection precise.

\begin{remark}\label{rem:kloosterman-eigenvalues-are-satake-parameters-for-a-function-field}
	The exotic Kloosterman sheaf $\ExoticKloosterman\left(\alpha^{-1}, \fieldCharacter\right)$ gives rise to an automorphic representation $\Pi$ of $\GL_{k}\left(\mathbb{A}\right)$, where $\mathbb{A}$ are the adeles of the function field $\finiteField\left(T\right)$. Any Frobenius orbit of $\left[\xi\right]$ for $\xi \in \multiplicativegroup{\algebraicClosure{\finiteField}}$ defines a place of the function field $\finiteField\left(T\right)$. The Satake parameter of $\Pi$ at $\left[\xi\right]$ is the conjugacy class of $\diag\left(\omega^{\ast}_1, \dots, \omega^{\ast}_k\right)$, where $\omega^{\ast}_1, \dots, \omega^{\ast}_k$ are the roots of $\ExoticKloosterman\left(\alpha^{-1}, \fieldCharacter\right)$ at $\xi$. See~\cite[Section 1.2]{Yun2017}.
\end{remark}

\appendix

\section{Properties of gamma factors over finite fields}
In this appendix, we list properties similar to the ten commandments of Lapid--Rallis~\cite{LapidRallis2005} that tensor product gamma factors attached to finite general linear groups are expected to satisfy. This list is similar to~\cite[Theorem 27]{Kaplan2023}.

By a tensor product gamma factor we mean an assignment $\left(\pi, \tau\right) \mapsto \gamma\left(\pi \times \tau, \fieldCharacter\right) \in \cComplex$ which is defined for a non-trivial character $\fieldCharacter \colon \finiteField \to \multiplicativegroup{\cComplex}$ and for irreducible representations $\pi$ and $\tau$ of $\GL_c\left(\finiteField\right)$ and $\GL_k\left(\finiteField\right)$, respectively, for any $k, c \ge 1$. The value $\gamma\left(\pi \times \tau, \fieldCharacter\right)$ might not be defined for all inputs $\pi$ and $\tau$. For example, the Shahidi gamma factor from~\cite{SoudryZelingher2023} is only defined when $\pi$ and $\tau$ are irreducible generic representations and the Ginzburg--Kaplan gamma factor from~\cite{CarmonZelingher2025} is defined for all irreducible representations $\pi$ but only for irreducible generic representations $\tau$.
\begin{theorem}
	Suppose that $\pi$ and $\tau$ are irreducible representations of $\GL_c\left(\finiteField\right)$ and $\GL_k\left(\finiteField\right)$, respectively.
	\begin{enumerate}
		\item (Dependence on $\fieldCharacter$) For any $a \in \multiplicativegroup{\finiteField}$ let $\fieldCharacter^a \colon \finiteField \to \multiplicativegroup{\cComplex}$ be the character $\fieldCharacter^a\left(x\right) = \fieldCharacter\left(ax\right)$. Suppose that $\gamma\left(\pi \times \tau, \fieldCharacter\right)$ is defined. Then $\gamma\left(\pi \times \tau, \fieldCharacter^a\right)$ is defined and $$\gamma\left(\pi \times \tau, \fieldCharacter^a\right) = \centralCharacter{\pi}\left(a\right)^k \centralCharacter{\tau}\left(a\right)^c \gamma\left(\pi \times \tau, \fieldCharacter\right),$$
		where $\centralCharacter{\pi}$ and $\centralCharacter{\tau}$ are the central characters of $\pi$ and $\tau$, respectively.
		\item (Contragredient) If $\gamma\left(\pi \times \tau, \fieldCharacter\right)$ is defined then so is $\gamma\left(\Contragradient{\pi} \times \Contragradient{\tau}, \fieldCharacter^{-1}\right)$ and we have the equality
		$$\gamma\left(\Contragradient{\pi} \times \Contragradient{\tau}, \fieldCharacter^{-1}\right) = \conjugate{\gamma\left(\pi \times \tau, \fieldCharacter\right)},$$
		where the right hand side is the complex conjugate of $\gamma\left(\pi \times \tau, \fieldCharacter\right)$.
		\item \label{item:commandements:symmetry} (Symmetry) If $\gamma\left(\pi \times \tau, \fieldCharacter\right)$ and $\gamma\left(\tau \times \pi, \fieldCharacter\right)$ are both defined then 
		$$\gamma\left(\pi \times \tau, \fieldCharacter\right) = \gamma\left(\tau \times \pi, \fieldCharacter\right).$$
		\item \label{item:commandements:multiplicativity} (Multiplicativity) Suppose that $\pi_1$, $\dots$, $\pi_t$ are irreducible representations of $\GL_{c_1}\left(\finiteField\right)$, $\dots$, $\GL_{c_t}\left(\finiteField\right)$, respectively, and suppose that $\tau_1$, $\dots$, $\tau_s$ are irreducible representations of $\GL_{k_1}\left(\finiteField\right)$, $\dots$, $\GL_{k_s}\left(\finiteField\right)$ such that for every $1 \le i \le t$ and every $1 \le j \le s$ the gamma factor $\gamma\left(\pi_i \times \tau_j, \fieldCharacter\right)$ is defined. Suppose that $\pi$ is an irreducible subrepresentation of the parabolically induced representation $\pi_1 \circ \dots \circ \pi_t$ and that $\tau$ is an irreducible subrepresentation of the parabolically induced representation $\tau_1 \circ \dots \circ \tau_s$ such that $\gamma\left(\pi \times \tau, \fieldCharacter\right)$ is defined. Then $$\gamma\left(\pi \times \tau, \fieldCharacter\right) = \prod_{i=1}^t \prod_{j=1}^s \gamma\left(\pi_i \times \tau_j, \fieldCharacter\right).$$
		\item\label{item:commandements:functional-equation} (Functional equation in the regular case) Suppose that $\gamma\left(\pi \times \tau, \fieldCharacter\right)$ is defined and that $\pi$ and $\Contragradient{\tau}$ do not have common elements in their cuspidal supports. Then $\gamma\left(\Contragradient{\pi} \times \Contragradient{\tau}, \fieldCharacter\right)$ is also defined and $$\gamma\left(\pi \times \tau, \fieldCharacter\right) \gamma\left(\pi^{\vee} \times \tau^{\vee}, \fieldCharacter^{-1}\right) = 1.$$
		\item (Cuspidal $\times$ character) Suppose that $\pi$ is an irreducible cuspidal representation and that $\chi \colon \multiplicativegroup{\finiteField} \to \multiplicativegroup{\cComplex}$ is a character. Then $\gamma\left(\pi \times \chi, \fieldCharacter\right)$ and $\gamma\left(\chi \times \pi, \fieldCharacter\right)$ are defined and $$\gamma\left(\pi \times \chi, \fieldCharacter\right) = \gamma\left(\chi \times \pi, \fieldCharacter\right) = \GaussSumScalar{\Contragradient{\pi} \times \chi^{-1}}{\fieldCharacter},$$
		where $\GaussSumScalar{\Contragradient{\pi} \times \chi^{-1}}{\fieldCharacter}$ is Kondo's Gauss sum (see \Cref{subsec:kondo-gauss-sum}).
		\item \label{item:commandements:cuspdial-exceptional}(Cuspidal exceptional case) Suppose that $\pi$ and $\tau$ are irreducible cuspidal representations. Then $\gamma\left(\pi \times \tau, \fieldCharacter\right)$ is defined. In the special case where $k=c$ and $\tau = \Contragradient{\pi}$ we have
		$$\gamma\left(\pi \times \Contragradient{\pi}, \fieldCharacter\right) = -\centralCharacter{\pi}\left(-1\right)^{c-1} q^{-\frac{c}{2}}.$$
	\end{enumerate}
\end{theorem}
\begin{remark}
	For the Shahidi gamma factor, proofs of all of these properties without appealing to local fields are known, see~\cite{SoudryZelingher2023}. It is worth mentioning that the proofs of the functional equation in the regular case (\ref{item:commandements:functional-equation}) and of the exceptional case in (\ref{item:commandements:cuspdial-exceptional}) rely on the equality with the Jacquet--Piatetski-Shapiro--Shalika gamma factor established in~\cite{SoudryZelingher2023}. For the Ginzburg--Kaplan gamma factor, we do not know how to prove the symmetry (\ref{item:commandements:symmetry}) and the formula for the exceptional case (\ref{item:commandements:cuspdial-exceptional}) without appealing to local fields.
\end{remark}

\begin{remark}
	More generally we expect that whenever $\gamma\left(\pi \times \tau, \fieldCharacter\right)$ is defined it equals $\varepsilon_0\left(\pi \times \tau, \fieldCharacter\right)$ (see \Cref{subsec:epsilon-factors}). In~\cite{zelingher2022values}, using the results of~\cite{Ye18},~\cite{ye2021epsilon} and~\cite{SoudryZelingher2023}, we confirmed that this is true for the Shahidi gamma factor. In~\cite{Zelingher2024b}, using the results of~\cite{CarmonZelingher2025}, we confirmed that this is also true for the Ginzburg--Kaplan gamma factor. These works rely on~\cite{Ye18} and~\cite{Zelingher2024b} which go through local fields.	
\end{remark}

Invoking properties (\ref{item:commandements:multiplicativity}),  (\ref{item:commandements:functional-equation}) and (\ref{item:commandements:cuspdial-exceptional}) we obtain the following corollary~\cite[Theorem 4.3]{SoudryZelingher2023}.

\begin{corollary}\label{cor:cuspidal-support-and-absolute-value}
	Suppose that $\pi$ and $\tau$ are irreducible representations of $\GL_c\left(\finiteField\right)$ and $\GL_k\left(\finiteField\right)$, respectively. Assume that $\tau$ is cuspidal and that $\gamma\left(\pi \times \Contragradient{\tau}, \fieldCharacter\right)$ is defined. Then $$\abs{\gamma\left(\pi \times \Contragradient{\tau}, \fieldCharacter\right)} = q^{-\frac{k d_{\pi}\left(\tau\right)}{2}},$$
	where $d_{\pi}\left(\tau\right)$ is the number of times that $\tau$ appears in the cuspidal support of $\pi$.
\end{corollary}

\newpage
\section{Characters of Speh representations}\label{appendix:character-of-speh-representation}
The purpose of this appendix is to describe the formula for the character of $\SpehRepresentation{\tau}{c}$ where $\tau$ is an irreducible cuspidal representation of $\GL_k\left(\finiteField\right)$.

Suppose that $\tau$ is the irreducible cuspidal representation of $\GL_k\left(\finiteField\right)$ corresponding to the Frobenius orbit $\left[\alpha\right]$ where $\alpha \colon \multiplicativegroup{\finiteFieldExtension{k}} \to \multiplicativegroup{\cComplex}$ is a regular character. For any $m \ge 1$ let $\jordanSupportedCharacter{\left(km\right)} \colon \GL_{km}\left(\finiteField\right) \to \cComplex$ be the following Jordan block supported class function. It is supported only on conjugacy classes of the form $J_{\mu}\left(h_{\xi}\right)$ where $\mu \vdash b$ and $\xi \in \multiplicativegroup{\finiteFieldExtension{a}}$ is of degree $a$, where $a,b \ge 1$ satisfy $ab = km$, and is defined on these conjugacy classes by the formula $$\jordanSupportedCharacter{\left(km\right)}\left(J_{\mu}\left(h_{\xi}\right)\right) = \left(-1\right)^{\lengthof\left(\mu\right) - 1} \prod_{j=1}^{\lengthof\left(\mu\right) - 1}\left(q^{aj} - 1\right) \cdot \sum_{j=0}^{a-1} \alpha\left(\FieldNorm{km}{k}\left(\xi^{q^j}\right)\right).$$

For any partition $\lambda = \left(\lambda_1,\dots,\lambda_{\ell}\right)$, let $k \lambda = \left(k\lambda_1,\dots,k \lambda_{\ell}\right)$. We define $$\jordanSupportedCharacter{k \lambda} = \jordanSupportedCharacter{\left(k \lambda_1\right)} \circ \dots \circ \jordanSupportedCharacter{\left(k \lambda_{\ell}\right)},$$
where $\circ$ is parabolic induction (see \Cref{subsec:class-functions}). The character of $\SpehRepresentation{\tau}{c}$ is given by the following formula, see~\cite[Definition 7.3]{Green55} or~\cite[Section 3]{SilbergerZink00}:
$$\trace \SpehRepresentation{\tau}{c} = \left(-1\right)^{\left(k-1\right)c} \sum_{\lambda \vdash c} \frac{\jordanSupportedCharacter{k \lambda}}{z_{\lambda}},$$
where $z_{\lambda} = \prod_{j=1}^{\infty} \left(\lambda\left(j\right)! \cdot j^{\lambda\left(j\right)}\right),$
where $\lambda\left(j\right)$ is the number of times $j$ appears in the partition $\lambda$.

In~\cite{Green55}, Green gives a formula for the computation of the basic character $\jordanSupportedCharacter{k \lambda}$. However, the computation of $\trace \SpehRepresentation{\tau}{c}$ using this formula is difficult. We provide here the table for the case $k = c = 2$. This example shows that $\trace \SpehRepresentation{\tau}{2}$ is supported on many different types of conjugacy classes.

\begin{table}[H]
	\begin{center}
		\caption{The character table of $\trace \SpehRepresentation{\tau}{2}$ for an irreducible cuspidal representation of $\GL_2\left(\finiteField\right)$ corresponding to a regular character $\theta_2 \colon \multiplicativegroup{\finiteFieldExtension{2}} \to \multiplicativegroup{\cComplex}$.} 
		\begin{tabular}{|l|l|}
			\hline $C$ & $\trace \SpehRepresentation{\tau}{2}\left(C\right)$ \\
			\hline $h_{\xi_4}$ & $\theta_2\left(\FieldNorm{4}{2}\left(\xi_4\right)\right)+\theta_2^q\left(\FieldNorm{4}{2}\left(\xi_4\right)\right)$ \\
			\hline $\operatorname{diag}\left(h_{\xi_{2,1}}, h_{\xi_{2,2}}\right)$ & $\left(\theta_2\left(\xi_{2,1}\right) + \theta^q_2\left(\xi_{2,1}\right)\right)\left(\theta_2\left(\xi_{2,2}\right) + \theta^q_2\left(\xi_{2,2}\right)\right)$ \\
			\hline $\operatorname{diag}\left(h_{\xi_2}, J_{(2)}\left(\xi_1\right)\right)$ & $\theta_2\left(\xi_1\right) \theta_2\left(\xi_2\right)+\theta_2\left(\xi_1\right) \theta_2^q\left(\xi_2\right)$ \\
			\hline $\operatorname{diag}\left(h_{\xi_2}, J_{\left(1^2\right)}\left(\xi_1\right)\right)$ & $(-q+1)\left(\theta_2\left(\xi_1\right) \theta_2\left(\xi_2\right)+\theta_2\left(\xi_1\right) \theta_2^q\left(\xi_2\right)\right)$ \\
			\hline $J_{(2)}\left(h_{\xi_2}\right)$ & $\theta_2\left(\xi_2\right)^2+\theta_2^q\left(\xi_2\right)^2+\theta_2\left(\xi_2\right) \theta_2^q\left(\xi_2\right)$ \\
			\hline $J_{\left(1^2\right)}\left(h_{\xi_2}\right)$ & $\left(q^2+1\right) \theta_2\left(\xi_2\right) \theta_2^q\left(\xi_2\right)+\theta_2\left(\xi_2\right)^2+\theta_2^q\left(\xi_2\right)^2$ \\
			\hline $\operatorname{diag}\left(J_{(2)}\left(\xi_{1,1}\right), J_{(2)}\left(\xi_{1,2}\right)\right)$ & $\theta_2\left(\xi_{1,1}\right) \theta_2\left(\xi_{1,2}\right)$ \\
			\hline $\operatorname{diag}\left(J_{(2)}\left(\xi_{1,1}\right), J_{\left(1^2\right)}\left(\xi_{1,2}\right)\right)$ & $(-q+1) \theta_2\left(\xi_{1,1}\right) \theta_2\left(\xi_{1,2}\right)$ \\
			\hline $\operatorname{diag}\left(J_{\left(1^2\right)}\left(\xi_{1,1}\right), J_{\left(1^2\right)}\left(\xi_{1,2}\right)\right)$ & $\left(q^2-2 q+1\right) \theta_2\left(\xi_{1,1}\right) \theta_2\left(\xi_{1,2}\right)$ \\
			\hline $J_{(4)}\left(\xi_1\right)$ & $\theta_2\left(\xi_1\right)^2$ \\
			\hline $J_{(1,3)}\left(\xi_1\right)$ & $(-q+1) \theta_2\left(\xi_1\right)^2$ \\
			\hline $J_{\left(2^2\right)}\left(\xi_1\right)$ & $\left(q^2-q+1\right) \theta_2\left(\xi_1\right)^2$ \\
			\hline $J_{\left(1^2, 2\right)}\left(\xi_1\right)$ & $(-q+1) \theta_2\left(\xi_1\right)^2$ \\
			\hline $J_{\left(1^4\right)}\left(\xi_1\right)$ & $\left(q^4-q^3-q+1\right) \theta_2\left(\xi_1\right)^2$ \\
			\hline
		\end{tabular}
	\end{center}
\end{table}
In the table, $\xi_m$ represents an element of $\multiplicativegroup{\finiteFieldExtension{m}}$ of degree $m$. The elements $\xi_{m,1}, \xi_{m,2}$ represent elements of  $\multiplicativegroup{\finiteFieldExtension{m}}$ of degree $m$ lying in different Frobenius orbits. The character of $\SpehRepresentation{\tau}{2}$ vanishes on conjugacy classes that do not appear in the table.

\bibliographystyle{abbrv}
\bibliography{references}
\end{document}